\documentclass[11pt]{amsart}

\usepackage{amsmath}
\usepackage{amsthm}
\usepackage{hyperref}
\usepackage[margin=1.0in]{geometry}

\numberwithin{equation}{section}

\newtheorem{prop}{Proposition}
\newtheorem{lemma}[prop]{Lemma}

\newtheorem{thm}[prop]{Theorem}
\newtheorem{cor}[prop]{Corollary}
\newtheorem{conj}[prop]{Conjecture}
\numberwithin{prop}{section}

\theoremstyle{definition}
\newtheorem{defn}[prop]{Definition}

\newtheorem{rmk}[prop]{Remark}

\newcommand{\del}{\partial}
\newcommand{\delb}{\bar{\partial}}\newcommand{\dt}{\frac{\partial}{\partial t}}
\newcommand{\brs}[1]{\left| #1 \right|}

\newcommand{\gG}{\Gamma}
\renewcommand{\gg}{\gamma}
\newcommand{\gD}{\Delta}
\newcommand{\gd}{\delta}
\newcommand{\gs}{\sigma}

\newcommand{\gU}{\Upsilon}
\newcommand{\gL}{\Lambda}

\newcommand{\gl}{\lambda}

\newcommand{\gw}{\omega}
\newcommand{\ga}{\alpha}
\newcommand{\gb}{\beta}
\renewcommand{\ge}{\epsilon}
\newcommand{\N}{\nabla}
\newcommand{\FF}{\mathcal F}

\newcommand{\EE}{\mathcal E}

\newcommand{\til}[1]{\widetilde{#1}}
\newcommand{\nm}[2]{\brs{\brs{ #1}}_{#2}}

\renewcommand{\bar}[1]{\overline{#1}}

\renewcommand{\i}{\sqrt{-1}}

\newcommand{\bga}{\bar{\alpha}}
\newcommand{\bgb}{\bar{\beta}}
\newcommand{\bb}{\bar{b}}

\newcommand{\bi}{\bar{i}}
\newcommand{\bj}{\bar{j}}
\newcommand{\bk}{\bar{k}}
\newcommand{\bl}{\bar{l}}
\newcommand{\bm}{\bar{m}}
\newcommand{\bn}{\bar{n}}
\newcommand{\bp}{\bar{p}}
\newcommand{\bq}{\bar{q}}
\newcommand{\br}{\bar{r}}
\newcommand{\bs}{\bar{s}}
\newcommand{\bu}{\bar{u}}
\newcommand{\bv}{\bar{v}}

\newcommand{\IP}[1]{\left<#1\right>}

\newcommand{\HH}{\mathcal{H}}

\newcommand{\seven}{\mbox{VII}}

\DeclareMathOperator{\Rc}{Rc}

\DeclareMathOperator{\tr}{tr}

\DeclareMathOperator{\Id}{Id}

\DeclareMathOperator{\osc}{osc}

\DeclareMathOperator{\rank}{rank}
\DeclareMathOperator{\End}{End}

\begin{document}

\title[Global existence results for the
pluriclosed flow]{Pluriclosed flow, Born-Infeld geometry, and rigidity results
for generalized K\"ahler manifolds}

\begin{abstract} We prove long time existence and convergence results for the
pluriclosed flow, which imply
geometric and topological classification theorems for generalized K\"ahler
structures.  Our approach centers on the reduction of pluriclosed flow to a
degenerate
parabolic equation for a $(1,0)$-form, introduced in \cite{ST2}.  We observe a
number of differential inequalities satisfied by this system which lead to a
priori $L^{\infty}$ estimates for the metric along the flow.  Moreover we
observe an unexpected connection to ``Born-Infeld geometry'' which leads to a
sharp differential inequality which can be used to derive an Evans-Krylov type
estimate for the degenerate parabolic system of equations.  To show convergence
of the flow we generalize Yau's oscillation estimate to the setting of 
generalized K\"ahler geometry.
\end{abstract}

\author{Jeffrey Streets}
\address{Rowland Hall\\
         University of California, Irvine\\
         Irvine, CA 92617}
\email{\href{mailto:jstreets@uci.edu}{jstreets@uci.edu}}
\thanks{J. Streets gratefully acknowledges support from the NSF via
DMS-1301864, and from an Alfred P. Sloan Fellowship.}

\date{February 7th, 2015}

\maketitle

\section{Introduction}
\subsection{Global existence on negatively curved backgrounds}

Given $(M^{2n}, g, J)$ a Hermitian manifold, we say that the metric is
\emph{pluriclosed} if
\begin{align*}
 \del\delb \gw = 0.
\end{align*}
Now consider the \emph{pluriclosed flow} equation
\begin{align} \label{PCF}
 \dt \gw =&\ \del \del^*_{g} \gw + \delb \delb^*_{g} \gw +
\frac{\sqrt{-1}}{2} \del\delb \log \det g.
\end{align}
This equation was introduced in \cite{ST1} as a natural geometric flow on
complex manifolds which preserves pluriclosed metrics.  In \cite{ST2} it was
established that this flow is the gradient flow of a Perelman-type quantity, and
a conjectural framework was established for understanding the singularity
formation of solutions to this flow.  Further works \cite{Boling,Dai,Enr} on
pluriclosed flow
have also appeared.

Our goal in this paper is to present global existence and convergence results
for this flow and the geometric and topological rigidity results which follow as
corollaries.  The first principal result shows 
long time existence and convergence for the pluriclosed flow with arbitrary
initial data on certain
complex manifolds.

\begin{thm} \label{PCFLTE} Let $(M^{2n}, h,J)$ be a compact Hermitian manifold.
\begin{enumerate}
\item If $h$ has nonpositive
bisectional curvature, then the solution to (normalized)
pluriclosed flow with arbitrary initial data exists smoothly on
$[0,\infty)$.
\item If $(M^{2n}, J)$ is biholomorphic to a torus, the solution to
pluriclosed flow with arbitrary initial data converges exponentially
to a flat K\"ahler metric.
\item If $h$ has constant negative bisectional curvature, the
solution to normalized pluriclosed flow with arbitrary initial data exists for
all time and converges exponentially to $g_{KE}$, the unique K\"ahler-Einstein
metric on
$(M^{2n}, J)$.
\end{enumerate}
\end{thm}

This theorem confirms the intuition that the pluriclosed flow does not develop
``local" singularities akin to the Ricci flow neckpinch, as in principle one
would see these for instance in the case of the torus.  Rather, this suggests
that the singularities are based on topological obstructions, as in the case of
K\"ahler-Ricci flow.  Moreover we point out that the convergence to flat metrics
on the tori stands in contrast to the parabolic Monge-Ampere equation in this
setting, which always converges (\cite{Gil}), but admits an infinite dimensional
space of fixed points with arbitrary torsion tensor.

\subsection{Rigidity of generalized K\"ahler manifolds via pluriclosed flow}

We also prove a new long time existence and
convergence result for the pluriclosed flow in the setting of commuting
generalized K\"ahler geometry.  A generalized K\"ahler structure on a compact
manifold $M$ is a triple
$(g,J_A,J_B)$ of a Riemannian metric and integrable complex structures so that
\begin{align*}
d^c_{A} \gw_A =&\ - d^c_B \gw_B\\
d d^c_A \gw_A =&\ - d d^c_B \gw_B = 0.
\end{align*}
These equations first arose in the context of supersymmetric
sigma models \cite{Gates}.  Later these equations were given a purely geometric
interpretation
in the language of Hitchin's ``generalized geometry,'' \cite{Gualtieri,
Gualtthesis,Hitchin}.

In \cite{ST3} the author and Tian showed that the pluriclosed flow preserves
generalized K\"ahler structure, provided we allow the complex structures to
evolve by diffeomorphisms.  Later, in \cite{SPCFSTB} we showed that in the
special case when $[J_A,J_B] = 0$, the complex structures remain fixed and
moreover the flow reduces to a fully nonlinear scalar PDE.  Moreover, in
\cite{SPCFSTB} we gave a nearly complete description of the long time behavior
of the pluriclosed flow in the setting when $n=2$.  The reason for the
dimensional restriction stems largely from the fact that when $n=2$ the factors
in the splitting of the tangent bundle arising from the commuting complex
structures (see \S \ref{GKsetup}) are both line bundles.

In this paper we develop a number of new a priori estimates for the pluriclosed
flow in this setting, leading to a general long time existence and convergence
result given certain topological constraints.  This leads to a rigidity result
showing that under certain topological conditions generalized K\"ahler
structures are automatically covered by products of Calabi-Yau manifolds. 
First we state a general long time existence result (see \S \ref{GEGK} for the
definition of $\chi(J_A,J_B)$, which is a characteristic class associated to a
generalized K\"ahler manifold akin to the first Chern class of a K\"ahler
manifold).  

\begin{thm} \label{CYtype} Let $(M^{2n}, g, J_A, J_B)$ be a commuting
generalized K\"ahler
manifold satisfying the conditions $c_1^{BC}(J_A) \leq 0$, $\chi(J_A, J_B) =
0$, and $\rank
T_-^{1,0} =
1$.  The solution to pluriclosed flow with initial condition $g$ exists smoothly
on $[0,\infty)$.
\end{thm}

The simplest example of a manifold satisfying the hypotheses of
Theorem \ref{CYtype} is to let $(N_i,J_i)$ be Calabi-Yau manifolds with $\dim
N_1 = 1$, and let $M = N_1 \times N_2$, $J_A = J_1 \oplus J_2$, $J_B = J_1
\oplus (-J_2)$.  With one further hypothesis on the background complex
manifolds, namely that the torsion class $[\del \gw_A] $ vanishes, we can use
the pluriclosed flow to show that, up to coverings, this is the \emph{only} way
to construct such manifolds.  This follows by establishing convergence of the
flow at infinity to a K\"ahler-Einstein metric.  One important input to
obtain this
convergence is a generalization of Yau's oscillation estimate for the
K\"ahler potential \cite{Yau} to the setting of generalized K\"ahler geometry
(Theorem \ref{Yauosc}). 
Note that in the theorem below we do not assume that either of the given complex
structures $J_A, J_B$ admits K\"ahler metrics, rather, this is a consequence.

\begin{thm} \label{CYconv} Let $(M^{2n}, g, J_A, J_B)$ be a commuting
generalized K\"ahler
manifold such that $c_1^{BC}(J_A) = 0$, $\chi(J_A, J_B) = 0$, $\rank T_-^{1,0} =
1$, and $[\del \gw_A] = 0 \in H^{2,1}$.  The solution to pluriclosed flow with
initial
condition $g$ exists smoothly
on $[0,\infty)$ and converges exponentially as $t \to \infty$ to a Calabi-Yau
metric.  In particular, $(M^{2n}, J_{A,B})$ are both Calabi-Yau manifolds, and
the universal cover of $(M, J_A,J_B)$ is biholomorphic to a product manifold.
\end{thm}

We can use this theorem to state a clearer corollary, which says that any rank
$1$ commuting generalized K\"ahler structure for which the relevant
characteristic classes vanish is automatically Calabi-Yau, and moreover covered
by a product structure.

\begin{cor} \label{CYcor} Let $(M^{2n}, g, J_A, J_B)$ be a commuting generalized
K\"ahler
manifold such that $c_1^{BC}(T_{\pm}^{1,0}) = 0$, $\rank T_-^{1,0} =
1$, and $[\del \gw_A] = 0 \in H^{2,1}$.  Then $(M,J_{A,B})$ are Calabi-Yau
manifolds, and
the
universal cover of $(M, J_A, J_B)$ is biholomorphic to a product manifold.
\end{cor}

\begin{rmk} We emphasize that the hypotheses of Theorem \ref{CYtype}, Theorem
\ref{CYconv} and Corollary \ref{CYcor} are all topological in nature.  In
particular, locally, there is an infinite dimensional family of rank $1$
commuting generalized K\"ahler manifolds.  By adding global, topological
hypotheses we are able to show the existence of rigid metrics on these manifolds
via the pluriclosed flow.
\end{rmk}

\begin{rmk} The standard generalized K\"ahler structure on the Hopf surface
satisfies $\chi(J_A,J_B) = 0$ and $\rank T_-^{1,0} = 1$.  Thus we see that to
obtain K\"ahler rigidity, some sort of extra hypothesis is necessary.  The Hopf
surface example satisfies both $[\del \gw_A] \neq 0$ and $c_1^{BC}(J_A) \neq 0$.
 It is not clear yet if either of the hypotheses $[\del \gw_A] = 0$ or
$c_1^{BC}(J_A)$ can be removed while keeping rigidity.
\end{rmk}

\subsection{A priori estimates}

Proving these theorems requires the development of several new a priori
estimates for the pluriclosed flow, which we now outline.
A pluriclosed metric reduces locally to $\gw_{\ga} := \delb \ga + \del \bga$ for
some $\ga \in \Lambda^{1,0}$.  Based on this observation, in \cite{ST3} the
author
and Tian reduced the pluriclosed flow to a degenerate parabolic equation for a
$(1,0)$-form.  Locally this equation takes the form
\begin{align} \label{PCF1form}
 \dt \ga =&\ \delb^*_{g_{\ga}} \gw_{\ga} - \frac{\i}{2} \del \log \det g_{\ga}.
\end{align}
Making this reduction global requires certain choices of background data, made
precise in \S \ref{oneformsec}.  This is akin to the reduction of K\"ahler-Ricci
flow to the parabolic complex Monge Ampere equation, which requires certain
choices of background data.  Equation (\ref{PCF1form}) is a degenerate parabolic
equation for $\ga$, which can be further reduced to the parabolic complex
Monge-Ampere in the case that the underlying metric is K\"ahler.  In this paper
we approach the pluriclosed flow entirely from the point of view of equation
(\ref{PCF1form}), which as it turns out holds the key to many useful
differential inequalities which can be used to obtain a priori estimates.

The first principal estimate is an a priori $C^{\ga}$ estimate for the
metric in the presence of upper and lower bounds on the metric (Theorem
\ref{EKthm} below).  Thinking of the pluriclosed flow as a parabolic system of
equations for the Hermitian metric $g$, this estimate is analogous to the
DeGiorgi-Nash-Moser/Krylov-Safonov \cite{DeGi,Nash,Moser,KS1,KS2} estimate for
uniformly parabolic equations.  On the other hand, the corresponding estimate
for the K\"ahler-Ricci flow is analogous to a $C^{2,\ga}$ estimate for the
potential, and the techniques of Evans-Krylov \cite{Evans, Krylov} can be
applied to obtain this estimate.  We emphasize that these are only analogies, as
indeed we do NOT have a scalar reduction for our system, let alone a convex one,
so the result of Evans-Krylov cannot apply.  Moreover, the
DeGiorgi-Nash-Moser/Krylov-Safonov results are
known to be false in general for \emph{systems} of equations \cite{DeGiorgiCE},
which is the setting here.  Lastly, we point out that in complex coordinates the
pluriclosed flow is a quasilinear parabolic system with ``first order quadratic
nonlinearity."  This is the type of nonlinearity arising for instance in
harmonic maps, and Struwe \cite{Struwe} has shown the corresponding $C^{\ga}$
estimate in this setting with the further assumption that the nonlinear term is
small with respect to ellipticity constants, an assumption not available in this
setting.

Despite the complexity of the pluriclosed flow system, our method of proof is
related to that of
Evans-Krylov for the scalar PDE setting, which we briefly recount here.  Recall
that these results yield $C^{2,\ga}$ estimates for $C^{1,1}$ solutions to
uniformly parabolic, fully nonlinear, \emph{convex} equations.  This convexity
is exploited most crucially to show that every second directional derivative of
the given function is a subsolution to a uniformly parabolic equation.  One thus
obtains a weak Harnack estimate which is used in conjunction with the original
fully nonlinear equation to obtain the full $C^{2,\ga}$ regularity.  This method
certainly does not directly apply since we not even have a scalar PDE underlying
the pluriclosed flow, let alone a convex one.

Nonetheless, by a careful study of the $1$-form potential $\ga$ along the
pluriclosed flow, we discover a judicious combination of the first derivatives
of $\ga$ into a Hermitian $2n\times 2n$ matrix $W$ such that $W(v,v)$ is a
subsolution to a uniformly parabolic equation for every $v$, and such that $\det
W = 1$.  The form of this matrix was inspired by the author's previous joint
work with M. Warren \cite{SW} in which a closely related quantity was discovered
and applied to obtain Evans-Krylov type regularity for certain nonconvex
parabolic equations, arising partly from the pluriclosed flow in the generalized
K\"ahler setting \cite{SPCFSTB}, where the flow reduces to a scalar PDE.  In
that case the matrix playing the role of $W$ admits a clear interpretation as
the Hessian of a function obtained by applying a partial Legendre transformation
to the solution to the given PDE.  Somewhat miraculously, a very similar
quantity obeys remarkable partial differential inequalities for the general
pluriclosed flow, 
which is a parabolic system.  This matrix $W$ admits an interpretation as
a ``Born-Infeld" metric on the generalized tangent bundle $T \oplus T^*$.  This
seems at first glance to deepen the apparent connection between the pluriclosed
flow and generalized geometry \cite{STdual,SPCFSTB,STGK}.  We will discuss this
further in \S \ref{eksec}.

Before stating the theorem we introduce a piece of notation.  For a
Hermitian manifold $(M^{2n}, J,g)$, let $h$ denote an auxiliary Hermitian
metric,
and let $\gU(g,h)= \N_g - \N_h$ be the difference of the two Chern connections. 
Furthermore, let
\begin{align*}
f_k = f_k(g,h) := \sum_{j=0}^k \brs{  \N_g^{j} \gU(g,h)}^{\frac{2}{1+j}}.
\end{align*}
The quantity $f_k$ is a natural measure of the $k+2$-th derivatives of the
metric
which scales as the inverse of the metric.

\begin{thm} \label{EKthm}  Let $(M^{2n}, J)$ be a compact complex manifold.
 Suppose $g_t$ is a solution to the pluriclosed flow on $[0,\tau)$, $\tau \leq
1$, with
$\ga_t$ a solution to the $(\hat{g}_t,h,\mu)$-reduced flow (see \S
\ref{reducedflowsec}).  Suppose there
exist constants $\gl,\gL$ such that
\begin{align} \label{Ekthmhyp}
\gl g_0 \leq g_t \leq \gL g_0, \qquad \brs{\del \ga}^2 \leq \gL.
\end{align}
Given $k \in \mathbb N$ there exists a constant $C =
C(n,k,g_0,\hat{g},h,\mu,\gl,\gL)$ such that
\begin{align*}
\sup_{M \times \{t\}} t f_k(g_t,h) \leq C. 
\end{align*}
\end{thm}

The second principal estimate of this paper is a general upper bound for the
metric in terms of a lower bound.  The proof exploits the very favorable
evolution equations arising from the $1$-form reduction of pluriclosed flow to
control certain torsion terms arising in the evolution of metric quantities.

\begin{thm} \label{uplowbndthm} Let $(M^{2n}, J)$ be a compact complex manifold.
 Suppose $g_t$ is a solution to the pluriclosed flow on $[0,\tau)$, with
$\ga_t$ a solution to the $(\hat{g}_t,h,\mu)$-reduced flow. Assume there is a
constant
$\gl$ such that for all $t \in [0,\tau)$,
\begin{align*}
\gl g_0 \leq g_t.
\end{align*}
There exists a constant $\gL = \gL(n,g_0,\hat{g}, h, \mu,\gl)$ such that for all
$t \in
[0,\tau)$,
\begin{align*}
g_t \leq \gL (1 + t) g_0, \qquad \brs{\del \ga}^2 \leq \gL.
\end{align*}
\end{thm}

Here is an outline of the rest of this paper.  In \S \ref{bcksec} we recall
notation and some basic facts concerning the pluriclosed flow.  Then in \S
\ref{oneformsec} we develop the one-form reduction of the pluriclosed flow
introduced in \cite{ST2}.  Next in \S \ref{tpevsec} we introduce the ``torsion
potential" along a solution to pluriclosed flow.  In \S \ref{eksec} we give the
proof of Theorems \ref{EKthm} and
\ref{uplowbndthm}.  We use these in \S \ref{ltesec} and \S \ref{GEGK} to prove
Theorems \ref{PCFLTE}, \ref{CYtype} and \ref{CYconv}, and Corollary \ref{CYcor}.

\vskip 0.1in

\textbf{Acknowledgements:} The author would like to thank Marco Gualtieri and
Luis Ugarte for
useful conversations.

\section{Background} \label{bcksec}
\subsection{Different Faces of Pluriclosed Flow}

In this subsection we  give three equivalent formulations of the pluriclosed
flow equation, each of which is useful in easily displaying certain properties
of the equation.

\subsubsection{Hodge operators formulation}

First we express the pluriclosed flow equation using differential operators
appearing in Hodge theory.  This point of view makes manifest the fact that the
flow preserves the pluriclosed condition, and moreover is essential to
constructing the $1$-form formulation of pluriclosed flow analyzed below.  Let
$(M^{2n}, g, J)$ be a Hermitian manifold.  One has the natural decomposition
$d = \del + \delb$, and $d^* = \del^* + \delb^*$.  Recall that $\gw$ is
pluriclosed if $\del\delb \gw = 0$.  Since the local generality of pluriclosed
metrics is that of a $(1,0)$-form (see Lemma \ref{localpluriclosed}), to find a
parabolic flow of such metrics it is natural to consider the ansatz $\dt \gw =
\del \ga + \delb \bar{\ga} + \gg$, where $\ga \in \Lambda^{0,1}$ and $\gg$ is
closed.  Taking inspiration from K\"ahler-Ricci flow, it is natural to let $\gg
= - c_1(M,\gw)$.  Since we want a second-order flow, one needs $\ga$ to be a
first-order operator on the metric, and one has little choice but to set $\ga =
\del^* \gw$.  This point of view leads one to the pluriclosed flow equation
\begin{align*}
\dt \gw = \del \del^*_{g} \gw + \delb \delb^*_{g} \gw +
\i \del\delb \log \det g.
\end{align*}
As shown in \cite{ST1}, this is a strictly parabolic equation with pluriclosed
initial condition $\gw_0$, and admits short-time solutions on compact manifolds.

\subsubsection{Chern connection formulation}

Given $(M^{2n}, g, J)$ a Hermitian manifold, the Chern connection is the unique
connection $\N$ on $T^{1,0}(M)$ such that $\N g \equiv 0$, $\N J \equiv 0$ and
the torsion of $\N$ has vanishing $(1,1)$ piece.  In local complex coordinates
the connection
coefficients are
\begin{align*}
\gG_{ij}^k =&\ g^{\bl k} g_{j \bl,i}.
\end{align*}
The torsion of the Chern connection takes the form
\begin{align*}
T_{ij \bk} = g_{l \bk} \left[ \gG_{i j}^l - \gG_{k i}^l \right] = g_{j \bk,i} -
g_{i \bk,j}.
\end{align*}
The metric is K\"ahler if and only if $T \equiv 0$.  Also the Chern curvature
takes the form
\begin{align*}
\Omega_{i \bj k}^l =&\ - \del_{\bj} \gG_{i k}^l = - \del_{\bj} \left( g^{\bm l}
g_{k \bm,i} \right) = - g^{\bm l} g_{k \bm,i \bj} + g^{\bm p} g^{\bq l} g_{p
\bq,\bj} g_{k \bm,i}.
\end{align*}
Due to the fact that $\N$, in general, has torsion, there are various ``Ricci
curvatures" which can be defined using this connection.  We concentrate on one
of them,
\begin{align*}
S_{i \bj} =&\ g^{\bq p} \Omega_{p \bq i \bj}.
\end{align*}
Observe that this is the ``Ricci curvature" which appears in the theory of
Hermitian Yang-Mills theory, although in that setting the connection $\N$ is
some Hermitian connection on a complex vector bundle over $M$, which is
independent of the Hermitian metric chosen on $M$.  We also define a certain
quadratic expression in torsion, namely
\begin{align*}
Q_{i \bj} =&\ g^{\bl k} g^{\bn m} T_{i k \bn} T_{\bj \bl m}.
\end{align*}
With these definitions made,
we can express the pluriclosed flow equation (\cite{ST1} Proposition 3.3) as
\begin{align} \label{PCFCC}
\dt g =&\ - S + Q.
\end{align}

\subsubsection{Bismut connection formulation}

Let $(M^{2n}, g, J)$ be a Hermitian manifold.  Recall the operator $d^c = \i
(\delb - \del)$, and note that in particular one has
\begin{align*}
 d^c \gw(X,Y,Z) = - d \gw(JX,JY,JZ).
\end{align*}
The Bismut connection \cite{Bismut} is the unique connection on $TM$ for which
$\N g \equiv
0$, $\N J \equiv 0$, and which has skew symmetric torsion.  It follows that
\begin{align*}
 g(\N^B_X Y,Z) = g(\N_X^{LC} Y,Z) + \frac{1}{2} d^c \gw(X,Y,Z).
\end{align*}
This connection induces a connection on the canonical bundle, and its curvature
can be computed as
\begin{align*}
\rho^B(X,Y) = g( R^B(X,Y)e_i,Je_i).
\end{align*}
The form $\rho$ is closed by the Bianchi identity.  But since $J$ is no
longer parallel, $\rho$ is not necessarily a $(1,1)$-form.  A computation shows
that then the pluriclosed flow is equivalent to
\begin{align*}
\dt \gw = - \rho_B^{1,1}.
\end{align*}
Using this formulation one is able to show that solutions to the pluriclosed
flow are gauge-modified solutions to the $B$-field renormalization group flow,
which then exhibits pluriclosed flow as a gradient flow which admits
Perelman-type monotonicity formulas \cite{ST2}.

Lastly, for certain applications it is natural to add a normalizing term to the
pluriclosed flow similar to the normalization frequently imposed on
K\"ahler-Ricci flow.  Consider 
\begin{align} \label{npcf}
 \dt \gw =&\ - (\rho^B)^{1,1} - \gw
\end{align}

\subsection{Local generality of pluriclosed metrics and Aeppli classes}

We recall that, locally, any K\"ahler metric may be expressed as $\i \del\delb
f$ for some smooth real function $f$.  In this section we prove a similar result
for pluriclosed metrics, exhibiting the local generality of such metrics in
terms of $(1,0)$-forms.

\begin{lemma} \label{localpluriclosed} Let $U \subset \mathbb C^n$ be an open
subset homeomorphic to a
ball, and suppose $\gw \in \Lambda^{1,1}_{\mathbb R}$ is a pluriclosed form on
$U$.  There exists $\ga \in \Lambda^{1,0}$ such that
\begin{align*}
\gw = \delb \ga + \del \bar{\ga}.
\end{align*}

\begin{proof} Since the form $\del \gw$ is $d$-closed, so by the local
$\del\delb$ lemma we obtain $\gb \in \Lambda^{1,0}$ such that
\begin{align*}
\del \gw = \del\delb \gb.
\end{align*}
Now consider the form
\begin{align*}
\gg := \gw - \delb \gb - \del \bar{\gb}.
\end{align*}
Note that
\begin{align*}
\delb \gg = \delb \gw - \delb \del \bar{\gb} = 0, \qquad \del \gg = \del \gw -
\del
\delb \gb = 0.
\end{align*}
Since $\gg \in \Lambda^{1,1}_{\mathbb R}$ is $d$-closed, it follows again by the
$\del\delb$-lemma that there exists $f \in C^{\infty}(M, \mathbb R)$ such that
$\gg = \sqrt{-1} \del \delb f$.  Finally, set
\begin{align*}
\ga = \gb - \frac{\i}{2} \del f.
\end{align*}
We then directly compute
\begin{align*}
\delb \ga + \del \bar{\ga} = \delb \gb + \del \bar{\gb} + \i \del \delb f =
\delb
\gb + \del \bar{\gb} + \gg = \gw.
\end{align*}
\end{proof}
\end{lemma}

\noindent Given this lemma, on compact manifolds it is natural to consider
pluriclosed
metrics up to equivalence of adding $\delb \ga + \del \bar{\ga}$, in analogy
with K\"ahler classes in K\"ahler geometry.  

\begin{defn} \label{pluriclosedclass} Let $(M^{2n}, \gw, J)$ be a complex
manifold with pluriclosed
metric $\gw$.  We set
\begin{align*}
\HH_{\gw} := \{ \ga \in \Lambda^{1,0} |\ \gw_{\ga} := \gw + \delb \ga + \del
\bar{\ga} > 0 \}.
\end{align*}
\end{defn}

Note also that each metric $\gw_{\ga} \in \HH_{\gw}$ can be
described by an infinite dimensional equivalence class, by adding any element of
the image of
$\del : C^{\infty}(M) \to \Lambda^{1,0}$.  We formalize this in the next
definition.

\begin{defn} \label{gaugegroup} Let $(M^{2n}, J)$ be a complex manifold.  We let
\begin{align*}
 \mathcal G = \{ \del f | f \in C^{\infty}(M,\mathbb R) \}.
\end{align*}
Moreover, we let $\Lambda^{1,0} / \mathcal G$ denote the corresponding space of
equivalence classes of $(0,1)$-forms.
\end{defn}

\begin{rmk}
 Thus, as the K\"ahler form is written as
\begin{align*}
 \gw^{\ga} = \gw + \delb \ga + \del \bar{\ga},
\end{align*}
we have the coordinate expression
\begin{align*}
 \gw^{\ga}_{i \bj} = \gw_{i \bj} + \left( \delb \ga + \del \bar{\ga} \right)_{i
\bj} = \gw_{i \bj} + \left( \ga_{\bj,i} - \ga_{i,\bj} \right).
\end{align*}
This means that the \emph{metric coefficients} take the form
\begin{align} \label{oneformmetric}
 g^{\ga}_{i \bj} = g_{i \bj} - \i \left( \ga_{\bj, i} - \ga_{i,\bj} \right) =
g_{i \bj} + \i \left( \ga_{i,\bj} - \ga_{\bj,i} \right).
\end{align}
\end{rmk}

\subsection{The positive cone}

\begin{defn} Let $(M^{2n}, J)$ be a complex manifold.  Let
\begin{align*}
H^{1,1}_{\del + \delb} := \frac{ \left\{ \psi \in
\Lambda^{1,1}_{\mathbb R} |\ \del\delb \psi = 0 \right\} }{ \left\{ \delb \ga +
\del \bar{\ga} |\ \ga \in \Lambda^{1,0} \right\}}.
\end{align*}
This is referred to as the $(1,1)$-Aeppli cohomology, defined in \cite{Aeppli}. 
Next, in analogy with the K\"ahler cone, we next define the cone of
classes in $H^{1,1}_{\del + \delb}$ which admit pluriclosed metrics.
\end{defn}

\begin{defn} Let $(M^{2n}, J)$ be a complex manifold.  Let
\begin{align*}
\mathcal P := \left\{ [\psi] \in H^{1,1}_{\del + \delb}\ |\ \exists\ \gw \in
[\psi] , \gw > 0 \right\}.
\end{align*}
The space $\mathcal P$ is an open cone in $\mathcal H^{1,1}_{\del + \delb}$,
which is nonempty if and only if $M$ admits pluriclosed metrics.  We refer to
$\mathcal P$ as the \emph{positive cone}.
\end{defn}

\subsection{The formal existence time}

Observe that, if $\gw_t$ is a solution to pluriclosed flow, then $[\gw_t]$
defines a path in $\mathcal P$.  Moreover, since $H^{1,1}_{\mathbb R} \subset
H^{1,1}_{\del + \delb}$, we can interpret the first Chern class of $(M, J)$ as
an element of $H^{1,1}_{\del + \delb}$.  With this point of view we observe that
a solution to pluriclosed flow satisfies 
\begin{align*}
[\gw_t] = [\gw_0] - t c_1.
\end{align*}
Meanwhile, a solution to (\ref{npcf}) satisfies
\begin{align*}
 [\gw_t] = [ - \rho(h) + e^{ - t} (\gw_0 + \rho(h))]
\end{align*}
Certainly the flow cannot exist smoothly if the boundary of $\mathcal P$ is
reached along these paths.  We state this for emphasis.

\begin{defn} \label{taustardef} Let $(M^{2n}, J)$ be a compact complex manifold,
and suppose
$g_0$ is a pluriclosed metric on $M$.  Let
\begin{align*}
\tau^*(g_0) := \sup \{ t | [\gw_0] - t c_1 \in \mathcal P \}.
\end{align*}
\end{defn}

\begin{lemma} \label{taustarlemma} Let $(M^{2n}, J)$ be a compact complex
manifold, and suppose
$g_0$ is a pluriclosed metric on $M$.  If $\tau$ denotes the maximal existence
time of the solution to pluriclosed flow
with initial condition $g_0$, then $\tau \leq \tau^*(g_0)$.
\end{lemma}

The main guiding conjecture behind our study of pluriclosed flow is that in fact
reaching the boundary of the cone is the \emph{only} way to have a singularity.

\begin{conj} \label{existenceconj} \textbf{\emph{Weak existence conjecture:}}
Let $(M^{2n}, g_0, J)$ be a compact complex
manifold with
pluriclosed metric.  Then the solution to pluriclosed flow with initial
condition $g_0$ exists smoothly on $[0, \tau^*(g_0))$.
\end{conj}

\section{The 1-form reduction of pluriclosed flow} \label{oneformsec}

In this section we develop a reduction of the pluriclosed flow into a degenerate
parabolic system for a $(1,0)$-form.  Since a K\"ahler metric depends locally on
a single function and the Ricci flow preserves the K\"ahler condition, it is
natural to expect that the flow reduces to that of a single function, up to
finite dimensional (i.e. cohomological) obstructions.  This is easily borne out
by computations to show that the K\"ahler-Ricci flow is locally equivalent to
the parabolic complex Monge-Ampere equation.

Applying a similar line of thought, based on Lemma \ref{localpluriclosed} one
would expect the pluriclosed flow to reduce to a flow of a $(1,0)$-form.  Note
however that, on a compact K\"ahler manifold, the kernel of $\i \del \delb$
acting on functions is given just by constant functions, and therefore one
expects a strictly parabolic equation for the potential function, as indeed is
the case.   For the pluriclosed case, there is always an infinite dimensional
kernel to the description of a pluriclosed metric by a $(1,0)$-form (see
Definitions \ref{pluriclosedclass}, \ref{gaugegroup}).  Thus one expects to be
able to reduce the pluriclosed flow to a degenerate parabolic equation for a
$(1,0)$-form, and this was indeed observed in (\cite{ST2} Theorem 5.16,
Proposition 5.18) (n.b. our notation and conventions within differ slightly from
that paper).

Our purpose in this section is to further develop this point of view on the
pluriclosed flow to obtain a priori estimates.  At first the value of this
reduction to a $(1,0)$-form may seem doubtful since we have taken an equation
which is \emph{strictly} parabolic in complex coordinates and reduced to one
which is \emph{degenerate} parabolic.  Nonetheless this reduction gives us
access to quantities which are otherwise invisible, including a gauge-invariant
potential for the torsion of the flowing metric which obeys a remarkably simple
evolution equation along the flow, as detailed in \S \ref{tpevsec}.

\subsection{Reduction of pluriclosed flow} \label{reducedflowsec}

Let $(M^{2n}, g_t, J)$ be a smooth solution to pluriclosed flow on $[0,\tau]$. 
We fix a background Hermitian metric $h$.  By Lemma
\ref{taustarlemma}, we know $\tau < \tau^*(g_0)$, and hence there exists $\mu
\in
\Lambda^{1,0}$ such that
\begin{align} \label{omegahatdef}
\hat{\gw}_{\tau} := \gw_{0} - \tau \rho(h) + \delb \mu + \del \bar{\mu} > 0.
\end{align}
Now consider the smooth one-parameter family of K\"ahler forms
\begin{align*}
\hat{\gw}_t := \frac{t}{\tau} \hat{\gw}_{\tau} + \frac{(\tau-t)}{\tau} \gw_0.
\end{align*}
Observe that, as Aeppli cohomology classes, 
\begin{align*}
\frac{\del}{\del t} [\hat{\gw}_t] =&\ \frac{1}{\tau} \left[ \hat{\gw}_{\tau} -
\gw_0 \right] = \left[ \i
\del \delb \log \det h + \frac{1}{\tau} \left( \delb \mu + \del \bar{\mu}
\right) \right] = -c_1
\end{align*}
and hence $[\hat{\gw}_t] =  [\gw_0] - t c_1$, and so $\gw_t$ serves as an
appropriate family of background metrics.  

\begin{defn} Let $(M^{2n}, g_t, J)$ be a smooth solution to pluriclosed flow on
$[0,\tau]$.  Given choices $\hat{g}_t, h, \mu$ as above, for a one parameter
family $\ga_t \in \Lambda^{1,0}$ let
\begin{align*}
\gw_{\ga} := \hat{\gw}_t + \delb \ga_t + \del \bar{\ga_t}.
\end{align*}
We say that a one-parameter family $\ga_t \in \Lambda^{1,0}$ is a solution to
\emph{$(\hat{g}_t,h,\mu)$-reduced pluriclosed flow} if
\begin{gather} \label{reducedflow}
\begin{split}
\dt \ga =&\ \delb^*_{g_{\ga}} \gw_{\ga} - \frac{\i}{2} \del \log
\frac{\det g_{\ga}}{\det h} - \frac{\mu}{\tau},\\
\ga_0 =&\ 0.
\end{split}
\end{gather}
In the sequel, if no further context is given we will refer to a solution of
$1$-form reduced pluriclosed flow with some arbitrary choices of $\hat{g}_t,
h$ and
$\mu$
having been made.  Observe that there is always an infinite dimensional
equivalence class
of solutions corresponding to the
$\mathcal G$-orbit.
\end{defn}

\begin{lemma} \label{reduction} Let $(M^{2n}, g_t, J)$ be a smooth solution to
pluriclosed flow on $[0,\tau]$.  Given choices $\hat{g}_t, h, \mu$ as above,
there exists a
solution $\ga_t$ to $(\hat{g}_t,h,\mu)$-reduced pluriclosed flow, and for any
solution $\ga_t$ to (\ref{reducedflow}), one has $g_{\ga_t} = g_t$.
\begin{proof} Given the setup, we let $\ga_t$ be the solution to the ordinary
differential equation
\begin{gather} \label{reduxode}
\begin{split}
\dt \ga =&\ \delb^*_{g_t} \gw_t - \frac{\i}{2} \del \log \frac{\det g_t}{\det h}
- \frac{\mu}{\tau}\\
\ga_0=&\ 0.
\end{split}
\end{gather}
We thus compute
\begin{align*}
\dt \gw_{\ga_t} =&\ \dt \hat{\gw}_t + \delb \dot{\ga} + \del \dot{\bar{\ga}}\\
=&\ \left( \i \del \delb \log \det h + \frac{1}{\tau} (\delb \mu +
\del
\bar{\mu}) \right)\\
&\ + \delb \left(\delb^*_{g_{t}} \gw_{t} - \frac{\i}{2} \del
\log \frac{\det g_{t}}{\det h} - \frac{\mu}{\tau} \right) + \del
\left( \del^*_{g_t} \gw_{t} + \frac{\i}{2} \delb \log
\frac{\det {g}_{t}}{\det h} - \frac{\bar{\mu}}{\tau}\right)\\
=&\ \del \del^*_{g_{t}} \gw_{t} +  \delb \delb^*_{g_{t}} \gw_{t} + \i \del \delb
\log \det g.
\end{align*}
It follows that $\dt \left( \gw_t - \gw_{\ga_t} \right) = 0$, and so $\gw_t =
\gw_{\ga_t}$.  Plugging this into (\ref{reduxode}) gives that $\ga_t$ is indeed
a solution to (\ref{reducedflow}).  Also, given \emph{any} solution to
(\ref{reducedflow}), a calculation nearly identical to that above yields that
the family of metrics $g_{\ga_t} = g_t$.
\end{proof}
\end{lemma}

\begin{rmk} Note that this lemma does not claim an a priori construction of the
solution to (\ref{reducedflow}), in the sense that we already require the
solution to pluriclosed flow to obtain the reduction.  Nonetheless it is
possible to construct an a priori solution, and we comment on this below.
\end{rmk}

\subsection{The one-form differential operator}

In this subsection we begin our analysis of $1$-form reduced pluriclosed flow. 
We simplify matters by defining a differential operator lying at the heart of
the analysis.  

\begin{defn} Let $(M^{2n}, J)$ be a complex manifold.  Let $\hat{g}$ denote a
pluriclosed metric on $M$, let $h$ denote any Hermitian metric on $M$ , and $\mu
\in \Lambda^{1,0}$.  Given $\ga \in \HH_{\hat{\gw}}$, let
\begin{align*}
\Psi(\hat{g},h,\mu,\ga) =&\ \delb^*_{g_{\ga}} \gw_{\ga} - \frac{\i}{2} \del
\log
\frac{\det
g_{\ga}}{\det h} + \mu.
\end{align*}
Since the main object of interest in this operator is $\ga$, we will frequently
abbreviate $\Psi(\ga) = \Psi(\hat{g},h,\mu,\ga)$ when background choices of
the other objects is clear from context.
\end{defn}

\begin{lemma} \label{phicoordcalc} Given $\hat{g}, h, \mu, \ga$ as above, in
local complex coordinates we have
\begin{align*}
\Psi(\hat{g},h,\mu,\ga)_{i}  =&\ g_{\ga}^{\bq p} \left[ \ga_{i,p\bq} -
\frac{1}{2} \left( \ga_{p,\bq i} +
\ga_{\bq,pi} \right) \right] + \i g_{\ga}^{\bq p} \left[ \frac{1}{2}
\hat{g}_{p\bq,i}
- \hat{g}_{i \bq,p} \right]+\frac{\i}{2} h^{\bq p} h_{p \bq,i} + \mu_i.
 \end{align*}
 \begin{proof} A basic calculation in coordinates shows that
  \begin{align*}
 \left( \delb^*_{g} \gw - \frac{\sqrt{-1}}{2} \del \log \frac{\det g}{\det h}
\right)_{i} =&\ \sqrt{-1} g^{\bq p} \left[ g_{p\bq,i} - g_{i\bq,p} \right] -
\frac{\sqrt{-1}}{2} \left[ g^{\bq p} g_{p \bq,i} - h^{\bq p} h_{p\bq,i}
\right]\\
 =&\ \sqrt{-1} g^{\bq p} \left[ \frac{1}{2} g_{p\bq,i} - g_{i\bq,p} \right] +
\frac{\sqrt{-1}}{2} h^{\bq p} h_{p\bq,i}.
 \end{align*}
Using this we have
  \begin{align*}
\Psi(\hat{g},h,0,\ga)_{i} =&\ \left( \delb^*_{g_{\ga}} \gw_{\ga} -
\frac{\sqrt{-1}}{2}
\del \log
\frac{\det g_{\ga}}{\det h}
\right)_{i}\\
=&\ \sqrt{-1} g_{\ga}^{\bq p} \left[ \frac{1}{2} g^{\ga}_{p\bq,i} -
g^{\ga}_{i\bq,p} \right] +
\frac{\sqrt{-1}}{2} h^{\bq p} h_{p\bq,i}\\
 =&\ g_{\ga}^{\bq p} \left[ \ga_{i,p\bq} - \frac{1}{2} \left( \ga_{p,\bq i} +
\ga_{\bq,pi} \right) \right] + \i {g}_{\ga}^{\bq p} \left[ \frac{1}{2}
\hat{g}_{p\bq,i}
- \hat{g}_{i \bq,p} \right] + \frac{\i}{2} h^{\bq p} h_{p \bq,i}.
 \end{align*}
The result follows.
 \end{proof}
\end{lemma}

\begin{lemma} \label{phicalc} Given $\hat{g}, h, \mu, \ga$ as above, in
local complex coordinates we have
\begin{align*}
 \Psi(\hat{g},h,\mu,\ga)_{i} =&\ \gD_{g_{\ga}} \ga_i - \left( T_{\ga} \circ
\delb \ga
\right)_i + \i g_{\ga}^{\bq p} \hat{T}_{i p \bq} - \frac{\i}{2} \N_{i} \left[
\tr_{g_{\ga}} \hat{g} + \log \frac{\det g_{\ga}}{\det
h} - 2 \i \Re \bar{\N}^{*} \ga \right] + \mu_i,
 \end{align*}
 where
 \begin{align*}
 \left( T_{\ga} \circ \delb \ga \right)_i =&\ g_{\ga}^{\bl k}
g_{\ga}^{\bq p} T^{\ga}_{ik \bq}
\N_{\bl} \ga_{p}.
 \end{align*}
\begin{proof} We begin with three preliminary calculations,
\begin{align*}
 {\gD}_{g_{\ga}} \ga_{i} =&\ g_{\ga}^{\bl k} \left[ \N \bar{\N}
\ga \right]_{k \bl i}\\
 =&\ g_{\ga}^{\bl k} \left[ \del_{k} \left[ \bar{\N} \ga \right]_{\bl i} -
\gG_{k
i}^{p} (\bar{\N} \ga)_{\bl p} \right]\\
=&\ g_{\ga}^{\bl k} \left[ \ga_{i,k \bl} - g_{\ga}^{\bq p} g^{\ga}_{i \bq,k}
\N_{\bl} \ga_p \right].
\end{align*}
Also
\begin{align*}
\left(\N \N^* \ga \right)_{i} =&\ g_{\ga}^{\bl k} \N_{i}
\N_{\bl} \ga_k\\
=&\ g_{\ga}^{\bl k} \left[ \del_{i} (\N_{\bl} \ga_k) - \gG_{i k}^{p}
\N_{\bl} \ga_p \right]\\
=&\ g_{\ga}^{\bl k} \left[ \ga_{k,\bl i} - g^{\bq p}_{\ga} g^{\ga}_{k\bq,i}
\N_{\bl} \ga_p \right].
\end{align*}
Next
\begin{align*}
 \left( \N \bar{\N}^* \bar{\ga} \right)_{i} =&\ g_{\ga}^{\bl k} \N_{i} \N_k
\ga_{\bl}\\
 =&\ g_{\ga}^{\bl k} \left[ \del_{i} (\N_k \ga_{\bl}) - \gG_{i k}^{p}
(\N_p \ga_{\bl}) \right]\\
 =&\ g_{\ga}^{\bl k} \left[ \ga_{\bl,ik} - g_{\ga}^{\bq p} g^{\ga}_{k \bq,i}
\N_p \ga_{\bl} \right].
\end{align*}
Combining these yields
\begin{align*}
 g_{\ga}^{\bl k}& \left[ \ga_{i,\bl k} - \frac{1}{2} \left( \ga_{k,\bl i} +
\ga_{\bl,ki} \right) \right]\\
 =&\ \gD_{g_{\ga}} \ga_i - \frac{1}{2} \N_i \left( \N^* \ga + \bar{\N}^* \bga
\right) + g_{\ga}^{\bl k} g_{\ga}^{\bq p} \left[ g^{\ga}_{i \bq,k}
\N_{\bl} \ga_p - \frac{1}{2} g^{\ga}_{k\bq,i}
\N_{\bl} \ga_p - \frac{1}{2} g^{\ga}_{k \bq,i}
\N_p \ga_{\bl} \right]\\
=&\ \gD_{g_{\ga}} \ga_i - \frac{1}{2} \N_i \left( \N^* \ga + \bar{\N}^* \bga
\right) + g_{\ga}^{\bl k} g_{\ga}^{\bq p} \left[ T^{\ga}_{k i \bq}\N_{\bl} \ga_p
+ \frac{1}{2} g^{\ga}_{k \bq,i} (\N_{\bl} \ga_p - \N_{p} \ga_{\bl}) \right]\\
=&\ \gD_{g_{\ga}} \ga_i - \frac{1}{2} \N_i \left( \N^* \ga + \bar{\N}^* \bga
\right) + g_{\ga}^{\bl k} g_{\ga}^{\bq p} \left[ T^{\ga}_{k i \bq}
\N_{\bl} \ga_{p} + \frac{\i}{2} g^{\ga}_{k \bq,i} (\hat{g}_{p\bl} -
g^{\ga}_{p\bl})
\right]\\
=&\ \gD_{g_{\ga}} \ga_i - \frac{1}{2} \N_i \left( \N^* \ga + \bar{\N}^* \bga
\right) + g_{\ga}^{\bl k} g_{\ga}^{\bq p} T^{\ga}_{k i \bq}
\N_{\bl} \ga_{p} - \frac{\i}{2} g_{\ga}^{\bl k} g_{k \bl,i}^{\ga} + \frac{\i}{2}
g_{\ga}^{\bl k} g_{\ga}^{\bq p} \hat{g}_{p \bl} g^{\ga}_{k \bq,i}.
\end{align*}
Plugging this into the result of Lemma \ref{phicoordcalc} yields
\begin{align*}
 \Psi_{i} =&\ \gD_{g_{\ga}} \ga_i - \N_i \Re \N^* \ga + g_{\ga}^{\bl k}
g_{\ga}^{\bq p} T^{\ga}_{k i \bq}
\N_{\bl} \ga_{p} - \frac{\i}{2} g_{\ga}^{\bl k} g_{k \bl,i}^{\ga} + \frac{\i}{2}
g_{\ga}^{\bl k} g_{\ga}^{\bq p} \hat{g}_{p \bl} g^{\ga}_{k \bq,i}\\
&\ + \i g_{\ga}^{\bq p} \left[ \frac{1}{2} \hat{g}_{p\bq,i}
- \hat{g}_{i \bq,p} \right] + \frac{\i}{2} h^{\bq p} h_{p \bq,i} + \mu_i.
\end{align*}
One further simplification yields
\begin{align*}
 \i g_{\ga}^{\bq p} \left[ - \hat{g}_{i \bq,p} + \frac{1}{2} \hat{g}_{p\bq,i}  +
\frac{1}{2} \hat{g}_{\ga}^{\bl k} g_{p \bl} g^{\ga}_{k\bq,i}\right] =&\ \i
g_{\ga}^{\bq p}
\left[ \hat{T}_{i p \bq} - \frac{1}{2} \hat{g}_{p\bq,i} + \frac{1}{2}
g_{\ga}^{\bl k}
\hat{g}_{p\bl} g^{\ga}_{k \bq,i} \right]\\
=&\ \i g_{\ga}^{\bq p} \hat{T}_{i p \bq} - \frac{\i}{2} \N_{i} \tr_{g_{\ga}}
\hat{g}.
\end{align*}
Combining these calculations yields the final result.
\end{proof}
\end{lemma}

\begin{prop} Given $\hat{g}, h, \mu$ as above, the map
$\Psi(\hat{g},h,\mu,\cdot) : \Lambda^{1,0} \to \Lambda^{1,0}$ is a second
order degenerate elliptic operator.
\begin{proof} A simple calculation using Lemma \ref{phicoordcalc} shows that
\begin{align*}
\left[\gs\left( D_{\ga} \Psi \right)(\gb) \right](\xi)_{i} =&\ g_{\ga}^{\bq p}
\left[ \gb_{i} \xi_{p} {\xi}_{\bq} - \frac{1}{2} \left( \gb_{p} \xi_{\bq}
\xi_{i} + \gb_{\bq} \xi_p \xi_{i} \right) \right].
\end{align*}
Thus
\begin{align*}
\IP{ \left[\gs\left( D_{\ga} \Psi \right)(\gb) \right](\xi),\gb} =&\
\brs{\gb}_{g_{\ga}}^2\brs{\xi}^2 - \frac{1}{2} g_{\ga}^{\bj i} g_{\ga}^{\bq p}
\left( \bar{\gb}_{\bj} \gb_{p} \xi_{\bq}
\xi_{i} + \gb_{\bq} \xi_p \xi_{i} \right)\\
=&\ \brs{\gb}_{g_{\ga}}^2\brs{\xi}^2 - \brs{\IP{\gb,\xi}_{g_{\ga}}}^2\\
\geq&\ 0,
\end{align*}
where the last line follows by the Cauchy-Schwarz inequality.
\end{proof}
\end{prop}

\begin{rmk} By direct calculation one can show that the kernel corresponds to
the image of $\del : C^{\infty}(M) \to \Lambda^{1,0}$, corresponding to the
$\mathcal G$-orbit, as expected.
\end{rmk}

\subsection{A splitting of the \texorpdfstring{$1$}{}-form system}

In this subsection we exhibit an essentially canonical way to ``split'' the
$1$-form equation into a simpler $1$-form equation and an equation for a scalar
quantity.

\begin{prop} \label{oneformsplitprop} Let $(M^{2n}, g_t, J)$ be a solution to
pluriclosed flow.  Given choices $\hat{g}_t, h, \mu$ as above, suppose
$(\gb_t, f_t)$ is a one-parameter family of $(1,0)$-forms $\gb$ and smooth
functions $f$ such that
\begin{gather} \label{oneformsplit}
\begin{split}
 \dt \gb =&\ \gD_{g_t} \gb - T_{g_t} \circ \delb \gb + \i \tr_{g_{\ga}}
T + \mu,\\
\dt f =&\ \gD_{g_t} f + \tr_{g_t} g + \log \frac{\det g_{t}}{\det h},\\
\ga_0 =&\ \gb_0 - \i \del f_0.\\
\end{split}
\end{gather}
Then $\ga_t := \gb_t - \i \del f_t$ is a solution to the $1$-form flow.
\begin{proof} 
We directly compute that
 \begin{align*}
  \dt \ga =&\ \dt \gb - \i \del \dt f\\
  =&\ {\gD}_{g_t} \gb - T_{\ga} \circ \delb \gb + \i \tr_{g_{\ga}} T + \mu - \i
\del \left( {\gD}_{g_t} f + \tr_{g_{t}} g + \log \frac{\det g_{t}}{\det h}
\right).
 \end{align*}
Next, we compute
\begin{align*}
\N_i \gD f =&\ g_{\ga}^{\bl k} \N_i \N_{k} \N_{\bl} f\\
=&\ g_{\ga}^{\bl k} \left[ \N_k \N_i \N_{\bl} f - g_{\ga}^{\bq p} T^{\ga}_{ik
\bq} \N_p \N_{\bl} f \right]\\
=&\ g_{\ga}^{\bl k} \left[ \N_k \N_i \N_{\bl} f - g_{\ga}^{\bq p} T^{\ga}_{ik
\bq} \N_p \N_{\bl} f \right]\\
=&\ \left[\gD \del f - T_{\ga} \circ \delb \del f \right]_i.
\end{align*}
Thus plugging this in above and comparing against Lemma \ref{phicalc} we
conclude that
\begin{align*}
 \dt \ga =&\ \gD_{g_t} \left( \gb - \i \del f \right) - T_{\ga} \circ \delb
\left(\gb - \i \del f \right) + \i \tr_{g_{t}} T + \mu\\
 &\ - \i \del \left( \tr_{g_{t}} g + \log \frac{\det g_{t}}{\det h} \right).
\end{align*}
It follows that the $1$-parameter family of metrics $\gw_{\ga} := \gw + \del \ga
+ \delb \bar{\ga}$ is a solution to pluriclosed flow with the given initial
condition.  The proposition follows.
\end{proof}
\end{prop}

\section{Torsion potential evolution equations} \label{tpevsec}

In this section we derive evolution equations and estimates for the torsion
potential along a solution to the $1$-form pluriclosed flow.  As we will see
below, a miraculous cancellation of nonlinear terms occurs which allows for a
very clean estimate of the torsion potential.  At the end of the section we
specialize these estimates to the case when one has a special background metric,
which allows for even stronger estimates.  These play a crucial role in the
proofs of Theorem \ref{PCFLTE} and \ref{CYtype}.

\begin{defn} \label{torspot} Let $(M^{2n}, \gw, J)$ be a complex manifold with
pluriclosed
metric.  Given $\ga \in \mathcal H_{\gw}$, we say that $\mu \in \Lambda^{2,0}$
is a \emph{torsion potential for
$\gw_{\ga}$} if
\begin{align*}
T_{\ga} = \del \gw - \delb \mu.
\end{align*}
Observe that 
$\del \ga \in \Lambda^{2,0}$ is a torsion potential for $\gw_{\ga}$.  We
say that $\phi \in \Lambda^{1,0}$ is a \emph{torsion pluripotential for
$\gw_{\ga}$} if
\begin{align*}
 T_{\ga} = \del \gw - \delb \del \phi.
\end{align*}
Again, one notes that $\ga$ is a torsion pluripotential for $\gw_{\ga}$.
\end{defn}

\begin{rmk} The reason for the terminology is that the torsion of the Chern
connection associated to a Hermitian metric is $\del \omega$.  In the case of
metrics in $\HH_{\gw}$ this reduces to $\del \gw - \delb \del \bar{\ga}$.  Thus
the quantity $\del \bar{\ga}$ governs the torsion tensor, up to a background
term of $\del \gw$.  From this perspective, we see that the reduction to the
$1$-form equation has allowed us to get some control over the torsion, a crucial
quantity to control in obtaining estimates, from a quantity which has one fewer
derivative.  As we will see this quantity satisfies a particularly nice
evolution equation.
\end{rmk}

\subsection{Evolution of torsion pluripotentials}

In this subsection we analyze the evolution of the torsion pluripotential along
a solution to pluriclosed flow.  A crucial input comes from a cancellation in
the application of a Bochner formula observed in \cite{SPCFSTB}, which we record
below.

\begin{lemma} (\cite{SPCFSTB} Lemma 4.7) \label{01formgenev} Let $(M^{2n},
g_t, J)$ be a solution to
pluriclosed flow, and
suppose $\gb_t \in \Lambda^{p,0}$ is a one-parameter family satisfying
\begin{align*}
 \dt \gb =&\ {\gD}_{g_t} \gb + \mu,
\end{align*}
where $\mu_t \in \Lambda^{p,0}$.  Then
\begin{align} \label{oneformflow}
 \dt \brs{\gb}^2 =&\ \gD \brs{\gb}^2 - \brs{\N \gb}^2 - \brs{\bar{\N} \gb}^2 -
p \IP{Q, \tr_g \left(\gb \otimes \bar{\gb} \right)} + 2 \Re \IP{\gb,\mu}.
\end{align}
\end{lemma}

\begin{prop} \label{splittingevprop} Let $(M^{2n}, g_t, J)$ be a solution to
the pluriclosed flow.  Given choices $\hat{g}_t, h, \mu$ as above, suppose
$(\gb_t, f_t)$ is a solution to (\ref{oneformsplit}).  Then
\begin{align*}
\dt \brs{\gb}^2 =&\ \gD \brs{\gb}^2 - \brs{\N \gb}^2 - \brs{\bar{\N} \gb}^2 -
\IP{Q, \gb \otimes \bar{\gb}} + 2 \Re \IP{\gb,T^{\ga} \circ \del \gb +
\tr_{g_{\ga}} T}.
\end{align*}
\begin{proof} This follows directly from Proposition \ref{oneformsplitprop} and
Lemma \ref{01formgenev}.
\end{proof}
\end{prop}

\noindent By exploiting some special identities in dimension $4$ we can obtain
an improved estimate of the evolution of $\brs{\gb}^2$.

\begin{cor} Let $(M^{4}, g_t, J)$ be a solution to the
pluriclosed flow.  Given choices $\hat{g}_t, h, \mu$ as above, suppose
$(\gb_t, f_t)$ is a solution to (\ref{oneformsplit}).  Then
\begin{align*}
\dt \brs{\gb}^2 \leq&\ \gD \brs{\gb}^2 - \brs{\bar{\N} \gb}^2 + 2 \Re \IP{\gb,
\tr_{g_{\ga}} T}.
\end{align*}
\begin{proof} We adapt the result of Proposition \ref{splittingevprop}.  Since
we are in real dimension $4$, (\cite{ST1} Lemma 4.4) implies that
$Q = \frac{1}{2} \brs{T}^2 g$, thus $- \IP{Q, \gb \otimes \bar{\gb}} = -
\frac{1}{2} \brs{T}^2 \brs{\gb}^2$.  Also note that, in a unitary frame for
$g_{\ga}$, using the skew-symmetry of $T$ one has
\begin{align*}
 \brs{T}^2 =&\ \sum_{i,j,k=1}^2 T_{i j \bk} T_{\bi \bj k} = 2 (\brs{T_{1 2
\bar{1}}}^2 + \brs{T_{1 2 \bar{2}}}^2).
\end{align*}
Let us now analyze the term $\Re \IP{\gb,
T^{\ga} \circ \del \gb}$.  Working again in a unitary frame for $g_{\ga}$ we see
\begin{align*}
 \IP{\gb, T^{\ga} \circ \del \gb} =&\ \sum_{i,j,l = 1}^2 T^{\ga}_{\bi \bj l}
\N_i \gb_{\bl} \gb_{j}\\
 =&\ \gb_1 \sum_{l=1}^2 T^{\ga}_{\bar{2} \bar{1} l} \N_2 \gb_{\bl} + \gb_2
\sum_{l=1}^2 T^{\ga}_{\bar{1} \bar{2} l} \N_1 \gb_{\bl}\\
=&\ \gb_1 T^{\ga}_{\bar{2} \bar{1} 1} \N_2 \gb_{\bar{1}} + \gb_1
T^{\ga}_{\bar{2} \bar{1} 2} \N_2 \gb_{\bar{2}} + \gb_2 T^{\ga}_{\bar{1}\bar{2}
1} \N_1 \gb_{\bar{1}} + \gb_2 T^{\ga}_{\bar{1}\bar{2} 2} \N_1 \gb_{\bar{2}}\\
\leq&\ \frac{1}{2} \left[ \brs{\gb_1}^2 \brs{T^{\ga}_{\bar{2}\bar{1} 1}}^2 +
\brs{\N_2 \gb_{\bar{1}}}^2 + \brs{\gb_1}^2 \brs{T^{\ga}_{\bar{2}\bar{1} 2}}^2 +
\brs{\N_2 \gb_{\bar{2}}}^2 \right.\\
&\ \left. \qquad + \brs{\gb_2}^2 \brs{T^{\ga}_{\bar{1}\bar{2} 1}}^2 + \brs{\N_1
\gb_{\bar{1}}}^2 + \brs{\gb_2}^2 \brs{T^{\ga}_{\bar{1}\bar{2} 2}}^2 + \brs{\N_1
\gb_{\bar{2}}}^2 \right]\\
=&\ \frac{1}{2} \left[ (\brs{\gb_1}^2 + \brs{\gb_2}^2) \left(
\brs{T^{\ga}_{\bar{2} \bar{1} 1}}^2 + \brs{T^{\ga}_{\bar{2} \bar{1} 2}}^2
\right) + \brs{\N \gb}^2 \right]\\
=&\ \frac{1}{2} \brs{\N \gb}^2 + \frac{1}{4} \brs{T}^2 \brs{\gb}^2.
\end{align*}
Collecting the above discussion it follows that
\begin{align*}
- \brs{\N \gb} + 2 \Re \IP{\gb, T^{\ga} \circ \del \gb} - \IP{Q, \gb \otimes
\bar{\gb}} \leq 0.
\end{align*}
The result follows.
\end{proof}
\end{cor}

\subsection{Evolution of the torsion potential}

In this subsection we derive the evolution of the torsion potential
along a solution to the $1$-form reduced pluriclosed flow.  We begin with a
preliminary calculation of $\gD_{g_{\ga}} \del \ga$.

\begin{lemma} \label{torspotlem1} In local complex coordinates we have
\begin{align*}
\left[ \gD_{g_{\ga}} \del \ga \right]_{ij} =&\ g_{\ga}^{\bq p} \left[
\left(
\del \ga \right)_{ij,\bq p} - g_{\ga}^{\bs r}
\left[ g_{i \bs,p} (\del \ga)_{rj,\bq} + g_{j\bs,p} (\del
\ga)_{ir,\bq} \right]
\right]\\
&\ + \i
g_{\ga}^{\bq p} g_{\ga}^{\bs r} \left[ \ga_{i,\bs p}
\ga_{r,j\bq} + \ga_{\bs,ip} \ga_{j,r\bq} - \ga_{\bs,ip} \ga_{r,j\bq} -
\ga_{j,\bs p} \ga_{r,i\bq} + \ga_{\bs,jp}
\ga_{r,i\bq} - \ga_{\bs,jp} \ga_{i,r\bq} \right].
\end{align*}
\begin{proof}
We directly compute
\begin{align*}
\left[ \gD_{g_{\ga}} \del \ga \right]_{ij} =&\ g_{\ga}^{\bq p} \left[ \N
\bar{\N} \del \ga \right]_{p \bq i j}\\
=&\ g_{\ga}^{\bq p} \left[ \left( \del \ga \right)_{ij,\bq p} - \gG_{p i}^r
(\del
\ga)_{rj,\bq} - \gG_{p j}^r (\del \ga)_{i r,\bq} \right]\\
=&\ g_{\ga}^{\bq p} \left[ \left( \del \ga \right)_{ij,\bq p} -
g_{\ga}^{\bs r} (g_{\ga})_{i \bs,p}(\del \ga)_{rj,\bq} - g_{\ga}^{\bs r}
(g_{\ga})_{j \bs,p} (\del \ga)_{ir,\bq} \right]\\
=&\ g_{\ga}^{\bq p} \left[ \left( \del \ga \right)_{ij,\bq p} -
g_{\ga}^{\bs r} \left[ \left[ g_{i\bs,p} + \i \left( \ga_{i,\bs p} -
\ga_{\bs,ip} \right) \right] \left(
\ga_{j,r\bq} - \ga_{r,j \bq} \right) \right. \right.\\
&\ \hskip 1.1in \left. \left. + \left[ g_{j\bs,p} + \i \left( \ga_{j,\bs p} -
\ga_{\bs,jp}
\right) \right] \left(\ga_{r,i\bq} - \ga_{i,r\bq} \right) \right] \right]\\
=&\ g_{\ga}^{\bq p} \left[ \left( \del \ga \right)_{ij,\bq p} - g_{\ga}^{\bs r}
\left[ g_{i \bs,p} (\del \ga)_{rj,\bq} + g_{j\bs,p} (\del \ga)_{ir,\bq} \right]
\right]\\
&\ \hskip 0.5in - \i
g_{\ga}^{\bq p} g_{\ga}^{\bs r} \left[ \ga_{i,\bs p} \ga_{j,r\bq} - \ga_{i,\bs
p}
\ga_{r,j\bq} - \ga_{\bs,ip} \ga_{j,r\bq} + \ga_{\bs,ip} \ga_{r,j\bq} \right.\\
&\ \hskip 1.3in \left. + \ga_{j,\bs p} \ga_{r,i\bq} - \ga_{j,\bs p}
\ga_{i,r\bq} - \ga_{\bs,jp}
\ga_{r,i\bq} + \ga_{\bs,jp} \ga_{i,r\bq} \right]\\
=&\ g_{\ga}^{\bq p} \left[ \left( \del \ga \right)_{ij,\bq p} - g_{\ga}^{\bs r}
\left[ g_{i \bs,p} (\del \ga)_{rj,\bq} + g_{j\bs,p} (\del \ga)_{ir,\bq} \right]
\right]\\
&\ + \i
g_{\ga}^{\bq p} g_{\ga}^{\bs r} \left[ \ga_{i,\bs p}
\ga_{r,j\bq} + \ga_{\bs,ip} \ga_{j,r\bq} - \ga_{\bs,ip} \ga_{r,j\bq} -
\ga_{j,\bs p} \ga_{r,i\bq} + \ga_{\bs,jp}
\ga_{r,i\bq} - \ga_{\bs,jp} \ga_{i,r\bq} \right].
\end{align*}
\end{proof}
\end{lemma}

\noindent Before computing the next evolution equation we record a general
coordinate formula.

\begin{lemma} \label{delphicoordcalc} Let $(M^{2n}, g, J)$ be a Hermitian
manifold.  In local complex coordinates we have
 \begin{align*}
(\del \delb^*_{g} \gw)_{ij} =&\ \i \left[ g^{\bq p} \left(g_{i\bq,pj} -
g_{j\bq,pi} \right) + g^{\bq r}
g^{\bm p} \left[ g_{r \bm,i} g_{j\bq,p} - g_{r \bm,j} g_{i\bq,p} \right]
\right].
 \end{align*}
 \begin{proof} First we note the basic coordinate calculation
\begin{align*}
\left( \delb^*_{g} \gw \right)_{j} =&\ \i g^{\bar{q} p} \left[ g_{p \bar{q},j} -
g_{j \bar{q},p} \right].
\end{align*}
Using this we compute
\begin{align*}
\left[\del \delb^*_g \gw \right]_{ij} =&\  \del_i \left( \delb^*_g \gw \right)_j
-
\del_j \left(\delb^*_g \gw \right)_i\\
 =&\ \i \del_i \left[ g^{\bq p} (g_{p \bq,j} - g_{j \bq,p}) \right] - \i \del_j
\left[g^{\bq p} (g_{p \bq,i} - g_{i \bq,p}) \right]\\
=&\ \i \left[ g^{\bq p} \left( g_{p\bq,ji} - g_{j\bq,pi} - g_{p\bq,ij} +
g_{i\bq,pj} \right) \right.\\
&\ \qquad \quad \left. + g^{\bq r} g^{\bm p} g_{r \bm,i} \left(g_{j\bq,p} -
g_{p\bq,j} \right) + g^{\bq r} g^{\bm p} g_{r\bm,j} \left( g_{p\bq,i} -
g_{i\bq,p} \right) \right]\\
=&\ \i \left[ g^{\bq p} \left(g_{i\bq,pj} - g_{j\bq,pi} \right) + g^{\bq r}
g^{\bm p} \left[ g_{r \bm,i} \left(g_{j\bq,p} -
g_{p\bq,j} \right) + g_{r\bm,j} \left( g_{p\bq,i} -
g_{i\bq,p} \right) \right] \right]\\
=&\ \i \left[ g^{\bq p} \left(g_{i\bq,pj} - g_{j\bq,pi} \right) + g^{\bq r}
g^{\bm p} \left[ g_{r \bm,i} g_{j\bq,p} - g_{r \bm,j} g_{i\bq,p} \right]
\right].
 \end{align*}
 \end{proof}
\end{lemma}

\begin{lemma} \label{tpest} Given $(M^{2n}, g, J)$, a Hermitian metric $h$,
$\mu \in \Lambda^{1,0}$, and $\ga \in \HH_{\gw}$,
one has
\begin{align*}
\del \Psi (g,h,\mu,\ga) =&\ \gD_{g_{\ga}} \del \bar{\ga} - \tr_{g_{\ga}}
\N^{\ga} T_g +
\del \mu.
\end{align*}
\begin{proof} Starting from Lemma \ref{delphicoordcalc} and plugging in
(\ref{oneformmetric}) yields
\begin{align*} 
& \left( \del{\Psi}(g,h,0,\ga) \right)_{ij} = \left[\del \left(
\delb^*_{g_{\ga}} \gw_{\ga}) - \frac{\i}{2} \del \log
\frac{\det
g_{\ga}}{\det h} \right) \right]_{ij} = \left(\del \delb^*_{g_{\ga}} \gw_{\ga}
\right)_{ij}\\
=&\ \i \left[ g_{\ga}^{\bq p} \left(
g_{i\bq,pj} + \i (\ga_{i,\bq pj} - \ga_{\bq,ipj})  - g_{j \bq,pi} - \i
(\ga_{j,\bq p i} - \ga_{\bq,jpi}) \right) \right.\\
&\ \hskip 0.5in + g_{\ga}^{\bq r} g_{\ga}^{\bm p} \left[ \left(g_{r \bm,i} + \i
\left(
\ga_{r,\bm i} - \ga_{\bm, ri} \right) \right) \left( g_{j \bq,p} + \i
(\ga_{j,\bq p} - \ga_{\bq,pj}) \right) \right.\\
&\ \hskip 1.0in \left. \left. - \left( g_{r \bm,j} + \i(\ga_{r,\bm j} -
\ga_{\bm,rj}) \right) \left( g_{i \bq,p} + \i(\ga_{i,\bq p} - \ga_{\bq,ip})
\right)  \right] \right]\\
=&\ \i \left[ - \i g_{\ga}^{\bq p} \left( \del \ga \right)_{ij,\bq p} +
g_{\ga}^{\bq p} \left( g_{i \bq,pj} - g_{j \bq,pi} \right)\right.\\
&\ \hskip  0.5in + g_{\ga}^{\bq r} g_{\ga}^{\bm p} \left[ g_{r \bm, i} g_{j
\bq,p} + \i g_{r \bm,i} (\ga_{j,\bq p} - \ga_{\bq,pj}) + \i g_{j \bq,p}
(\ga_{r,\bm i} - \ga_{\bm,ri}) \right.\\
&\ \hskip 1.05in - g_{r\bm,j} g_{i\bq,p} - \i g_{r\bm,j} (\ga_{i,\bq p} -
\ga_{\bq,ip}) - \i g_{i\bq,p} (\ga_{r,\bm j} - \ga_{\bm,rj})\\
&\ \hskip 1.05in - \ga_{r,\bm i} \ga_{j,\bq p} +
\ga_{r,\bm i} \ga_{\bq,pj} + \ga_{\bm,ri} \ga_{j,\bq p} - \ga_{\bm,ri}
\ga_{\bq,pj}\\
&\ \hskip 1.05in \left. \left. + \ga_{r,\bm j}
\ga_{i,\bq p} - \ga_{r,\bm j} \ga_{\bq,ip} - \ga_{\bm,rj} \ga_{i,\bq p} +
\ga_{\bm,rj} \ga_{\bq,ip} \right] \right]\\
=&\ g_{\ga}^{\bq p} \left( \del \ga \right)_{ij,\bq p}\\
&\ + \i g_{\ga}^{\bq r} g_{\ga}^{\bm p}\left[ \ga_{r,\bm j}
\ga_{i,\bq p} + \ga_{\bm,ri} \ga_{j,\bq p} - \ga_{r,\bm j} \ga_{\bq,ip} -
\ga_{r,\bm i} \ga_{j,\bq p} + \ga_{r,\bm i} \ga_{\bq,pj} - \ga_{\bm,rj}
\ga_{i,\bq p}
 \right]\\
 &\ + \i \left[ g_{\ga}^{\bq p} \left( g_{i \bq,pj} - g_{j \bq,pi} \right)
\right]\\
&\ + \i g_{\ga}^{\bq r} g_{\ga}^{\bm p} \left[ g_{r \bm, i} g_{j
\bq,p} + \i g_{r \bm,i} (\ga_{j,\bq p} - \ga_{\bq,pj}) + \i g_{j \bq,p}
(\ga_{r,\bm i} - \ga_{\bm,ri}) \right.\\
&\ \hskip 0.85in \left. - g_{r\bm,j} g_{i\bq,p} - \i g_{r\bm,j} (\ga_{i,\bq p} -
\ga_{\bq,ip}) - \i g_{i\bq,p} (\ga_{r,\bm j} - \ga_{\bm,rj}) \right].
\end{align*}
Now comparing against the result of Lemma \ref{torspotlem1} yields
\begin{align*}
& \left( \del{\Psi}(g,h,0,\ga) \right)_{ij}\\
=&\ \left[\gD_{g_{\ga}} \del \ga \right]_{ij} + g_{\ga}^{\bq p} g_{\ga}^{\bs r}
\left[ g_{i \bs,p} (\del \ga)_{rj,\bq} + g_{j\bs,p} (\del \ga)_{ir,\bq}
\right] + \i \left[ g_{\ga}^{\bq p} \left( g_{i \bq,pj} - g_{j \bq,pi} \right)
\right]\\
&\ + \i g_{\ga}^{\bq r} g_{\ga}^{\bm p} \left[ g_{r \bm, i} g_{j
\bq,p} + \i g_{r \bm,i} (\ga_{j,\bq p} - \ga_{\bq,pj}) + \i g_{j \bq,p}
(\ga_{r,\bm i} - \ga_{\bm,ri}) \right.\\
&\ \hskip 0.85in \left. - g_{r\bm,j} g_{i\bq,p} - \i g_{r\bm,j} (\ga_{i,\bq p} -
\ga_{\bq,ip}) - \i g_{i\bq,p} (\ga_{r,\bm j} - \ga_{\bm,rj}) \right].
\end{align*}
It remains to identify the lower order terms.  First, we relabel indices and
combine three of the terms to yield
\begin{align*}
 \i & g_{\ga}^{\bq p} \left[ g_{i \bq,pj} - g_{j \bq,pi} + g_{\ga}^{\bm l} (g_{p
\bm,i} g_{j \bq,l} - g_{p\bm,j} g_{i\bq,l}) \right]\\
 =&\ \i g_{\ga}^{\bq p} \left[ \del_p (-\i T_{j i \bq}) + g_{\ga}^{\bm l} (g_{p
\bm,i}
g_{j \bq,l} - g_{p\bm,j} g_{i\bq,l}) \right]\\
 =&\ \i g_{\ga}^{\bq p} \left[ \N^{\ga}_p (-\i T_{j i \bq}) + (\gG^{\ga})_{p
j}^l (-\i T_{l
i \bq}) + (\gG^{\ga})_{p i}^l (-\i T_{j l \bq}) + g_{\ga}^{\bm l} (g_{p \bm,i}
g_{j
\bq,l} - g_{p\bm,j} g_{i\bq,l}) \right]\\
=&\ \i g_{\ga}^{\bq p} \left[ \N^{\ga}_p (-\i T_{j i \bq}) + g_{\ga}^{\bm l}
\left(
g_{j \bm,p} + \i \del_p (\ga_{j,\bm} - \ga_{\bm,j}) \right) (g_{i\bq,l} -
g_{l\bq,i})\right.\\
&\ \qquad \qquad \qquad \quad \left. + g_{\ga}^{\bm l} \left( g_{i \bm,p} + \i
\del_p (\ga_{i,\bm} - \ga_{\bm,i}) \right) (g_{l\bq,j} - g_{j\bq,l}) \right.\\
&\ \qquad \qquad \qquad \quad \left. + g_{\ga}^{\bm l} (g_{p\bm,i} g_{j\bq,l} -
g_{p\bm,j} g_{i\bq,l}) \right]\\
=&\ g_{\ga}^{\bq p} \left[ - \N^{\ga}_p T_{i j \bq} - g_{\ga}^{\bm l} \left(
(\ga_{j,\bm p} - \ga_{\bm,jp})(g_{i \bq,l} - g_{l \bq,i}) + (\ga_{i,\bm p} -
\ga_{\bm,ip})(g_{l\bq,j} - g_{j\bq,l}) \right) \right].
\end{align*}
Inserting this identity into the calculation above, it remains to identify the
terms of type $\del g \star \del^2 \ga$.   These are, after relabeling indices,
\begin{align*}
\del g \star \del^2 \ga =&\ g_{\ga}^{\bq p} g_{\ga}^{\bm l} \left[ g_{i \bm,p}
(\ga_{j,l\bq} - \ga_{l,j\bq}) + g_{j \bm,p} (\ga_{l,i\bq} - \ga_{i,l\bq}) - g_{p
\bm,i} (\ga_{j,\bq l} - \ga_{\bq,lj})\right.\\
&\ \left. \qquad - g_{j \bq,l} (\ga_{p,\bm i} - \ga_{\bm,p i}) + g_{p \bm,j}
(\ga_{i,\bq l} - \ga_{\bq,il}) + g_{i\bq,l} (\ga_{p,\bm j} - \ga_{\bm,p j})
\right.\\
&\ \left. \qquad - (\ga_{j,\bm p} - \ga_{\bm,jp})(g_{i \bq,l} - g_{l \bq,i}) -
(\ga_{i,\bm p} -
\ga_{\bm,ip})(g_{l\bq,j} - g_{j\bq,l}) \right]\\
=&\ g_{\ga}^{\bq p} g_{\ga}^{\bm l} \left[ g_{i \bm,p} \ga_{j,l\bq} - g_{i\bm,p}
\ga_{l,j\bq} + g_{j\bm,p} \ga_{l,i\bq} - g_{j\bm,p} \ga_{i,l\bq} - g_{p\bm,i}
\ga_{j,\bq l} + g_{p\bm,i} \ga_{\bq,lj} \right.\\
&\ - g_{j\bq,l} \ga_{p,\bm i} + g_{j \bq,l} \ga_{\bm,pi} + g_{p\bm,j} \ga_{i,\bq
l} - g_{p\bm,j} \ga_{\bq,il} + g_{i\bq,l} \ga_{p,\bm j} - g_{i\bq,l}
\ga_{\bm,pj}\\
&\ \left. - g_{i\bq,l} \ga_{j,\bm p} + g_{l\bq,i} \ga_{j,\bm p} + g_{i\bq,l}
\ga_{\bm,jp} - g_{l\bq,i} \ga_{\bm,jp} \right.\\
&\ \left. - g_{l\bq,j} \ga_{i,\bm p} + g_{j\bq,l}
\ga_{i,\bm p} + g_{l\bq,j} \ga_{\bm,ip} - g_{j\bq,l} \ga_{\bm,ip} \right]\\
=&\ \sum_{i=1}^{20} A_i\\
=&\ 0,
\end{align*}
where the penultimate line defines the terms $A_i$ in the order they appear, and
the final line follows from the cancellations $A_1 + A_{13} = A_2 + A_{11} = A_3
+ A_7 = A_4 + A_{18} = A_5 + A_{14} = A_6 + A_{16} = A_8 + A_{20} = A_9 + A_{17}
= A_{10} + A_{19} = A_{12} + A_{15} = 0$.  The lemma follows.
\end{proof}
\end{lemma}

\begin{prop} \label{torsionpotentialev} Let $(M^{2n}, g_t, J)$ be a solution to
pluriclosed flow.  Fix background data $\hat{g}_t, h, \mu$ and a solution
$\ga_t$ to (\ref{reducedflow}).  Then
\begin{align*}
\dt \del \ga =&\ \gD_{g_{\ga}} \del \ga - \tr_{g_{\ga}} \N^{g_{\ga}} T_{\hat{g}}
+ \del \mu,\\
\dt \brs{\del \ga}^2_{g_{\ga}} =&\ \gD_{g_{\ga}} \brs{\del \ga}^2 - \brs{\N \del
\ga}^2 - \brs{\bar{\N} \del \ga}^2 - 2 \IP{Q, \tr\del \ga \otimes
\delb \bar{\ga}} - 2 \Re \IP{\tr_{g_{\ga}} \N^{g_{\ga}} T_{\hat{g}} + \del \mu,
\delb \bga}.
\end{align*}
\begin{proof} The first equation follows directly from Lemma \ref{tpest}.  The
second equation follows from the first and Lemma \ref{01formgenev}.
\end{proof}
\end{prop}

\begin{prop} \label{specialtorspot} Let $(M^{2n}, g_t, J)$ be a solution to
pluriclosed flow.  Fix background data $\hat{g}_t, h, \mu = 0$ and a solution
$\ga_t$ to (\ref{reducedflow}).  Suppose furthermore that
\begin{align*}
\del \hat{\gw}_t = \del \hat{\gw}_0 = \delb \eta.
\end{align*}
Let $\phi = \del \ga - \eta$.  Then
\begin{align*}
\left(\dt - \gD_{g_t} \right) \phi =&\ 0,\\
\left(\dt - \gD_{g_t} \right) \brs{\phi}^2 =&\ - \brs{\N \phi} - \brs{T_{g_t}}^2
- 2 \IP{Q, \phi \otimes \bar{\phi}}.
\end{align*}
\begin{proof} We observe that
\begin{align*} 
\left(\tr_{g_{\ga}} \N^{g_{\ga}} T_{\hat{g}} \right)_{ij} =&\ g_{\ga}^{\bq p} 
\N_p \del
\hat{\gw}_{ij
\bq}= g_{\ga}^{\bq p} \N_p \N_{\bq} \eta_{ij} = \gD_{g_{\ga}} \eta_{ij}.
\end{align*}
Thus using Proposition \ref{torsionpotentialev} and the assumption that $\mu =
0$ we obtain
\begin{align*}
\dt \phi =&\ \dt \del \ga = \gD_{g_t} \del \ga - \tr_g \N^g T_{\hat{g}} =
\gD_{g_t} \left( \del \ga - \eta \right) = \gD_{g_t} \phi.
\end{align*}
This yields the first claimed equation, and then the second follows from Lemma
\ref{01formgenev} and the fact that
\begin{align*}
\bar{\N} \phi =&\ \delb \phi = \delb \left( \del \ga - \eta \right) = -
T_{g_{t}}.
\end{align*}
\end{proof}
\end{prop}

\section{Evans-Krylov Regularity} \label{eksec}

In this section we prove Theorem \ref{EKthm}.  As explained in the introduction,
the idea is to consider a special matrix $W$ defined using a nonlinear
combination of first derivatives of $\ga$ (see Definition \ref{Wdef}).  The
matrix $W$ has two crucial properties which are used in Theorem \ref{EKthm}. 
First, for any choice of $\ga$, $W$ has unit determinant (Lemma \ref{calcl1}). 
Second, $W(\ga_t)$ is a matrix subsolution to a linear uniformly parabolic
equation along the pluriclosed flow.  This is proved in Proposition
\ref{Wsubsoln} after a long series of tedious calculations.  Using these two
crucial properties we establish Theorem \ref{EKthm} by adapting the method
of Evans-Krylov.

\subsection{Setup}

In analogy with the cone of positivity of Definition \ref{pluriclosedclass}, we
define a class of $(1,0)$-forms which corresponds to uniformly parabolic
solutions of pluriclosed flow.

\begin{defn} Given a domain $U \times [a,b] \subset \mathbb C^n \times \mathbb
R$, let $\gw$ denote the standard flat K\"ahler form on $\mathbb C^n$ and let
\begin{align*}
\EE^{\gl,\gL}_{U} = \{ \ga :[a,b] \to \gG (\Lambda^{1,0}(U))\ | \ \forall t \in
[a,b],\ \gl \gw \leq \i \left(\delb \ga_t + \del \bga_t \right) \leq \gL \gw, \
\brs{\del \ga}^2 \leq \gL \}.
\end{align*}
Moreover, given $\ga \in \EE^{\gl,\gL}_{U}$, let $\gw_{\ga} = \i \left( \delb
\ga + \del \bga \right)$, with corresponding metric coefficients 
\begin{align} \label{EKmetric}
g^{\ga}_{i \bj} = \i \left( \ga_{i,\bj} - \ga_{\bj,i} \right).
\end{align}
When some $\ga$ is given we will frequently drop the dependence of $g$ on $\ga$
in the notation.  At various points for notational simplicity we set $\gb_{ij} =
\i \del \ga_{ij},
\gb_{i}^{\bj} = g^{\bj k} \gb_{ik}$.  
\end{defn}

\begin{defn} \label{Wdef} Given $\ga \in \EE^{\gl\gL}_{U}$, let
\begin{align*}
 W = \left(
 \begin{matrix}
g_{i \bj} + \del \ga_{ip} \delb \bga_{\bj \bq} g^{\bq p} & \i \del \ga_{i p}
g^{\bl p}\\
- \i \delb \bga_{\bj \bq} g^{\bq k} & g^{\bl k}
 \end{matrix} \right) = \left(
 \begin{matrix}
  g_{i\bj} + \gb_{ik} \bar{\gb}_{\bj \bl} g^{\bl k} & \gb_i^{\bl}\\
  \gb_{\bj}^k & g^{\bl k}
 \end{matrix}
\right).
\end{align*}
\end{defn}

\begin{rmk} \label{BIrmk} This matrix $W$ can be interpreted as the natural
``Born-Infeld'' metric on the split tangent bundle $T \oplus T^*$, where $\del
\ga$ is playing the role of the skew-symmetric ``b-field.''  This metric first
arose through investigations into mathematical physics \cite{SeiWit, Tsey}. 
Later, investigations into generalized K\"ahler geometry \cite{Hitchin,
Gualtieri} identified natural geometric interpretations of this object, called a
generalized metric on $T \oplus T^*$.   Previously in joint work with Tian
\cite{ST3} we had shown that the pluriclosed flow is diffeomorphism-equivalent
to the renormalization group flow arising from a nonlinear sigma model coupled
to a b-field.  Moreover, in \cite{STGK} the author and Tian exhibited that the
pluriclosed flow preserves generalized K\"ahler geometry in the appropriate
sense.  These connections between the pluriclosed flow and
supersymmetry/generalized K\"ahler geometry inspired the choice of matrix $W$,
which can be shown to obey a remarkable differential inequality which lies at
the heart of Theorem \ref{EKthm}.
\end{rmk}

\begin{lemma} \label{calcl1} Given $\ga \in \Lambda^{1,0}$ such that $\i \left(
\delb \ga + \del \bga \right) > 0$, one has $ \det W(\ga) = 1$.
\begin{proof} First recall the block determinant formula
\begin{align*}
 \det \left(
 \begin{matrix}
  A & B\\
  C & D
 \end{matrix} \right) = \det D \det (A - B D^{-1} C).
\end{align*}
Using this we compute
\begin{align*}
\det W =&\ \det g^{-1} \det \left[ g_{i \bj} + \del \ga_{ip} \delb \bga_{\bj
\bq} g^{\bq p} - \left( \i \del \ga_{ip} g^{\bl p} g_{k \bl} (- \i \delb
\bga)_{\bj \bq} g^{\bq k} \right) \right]\\
=&\ \det g^{-1} \det g\\
=&\ 1.
\end{align*}
\end{proof}
\end{lemma}

Next we derive the evolution equation for $\ga$ along a solution to the
pluriclosed flow in this setting.  Using Lemma \ref{phicoordcalc} we observe
that
\begin{align*}
 \Phi(\ga)_i =&\ g^{\bj k} \left[ \ga_{i,\bj k} - \frac{1}{2} \left(
\ga_{\bj,ik} +
\ga_{k,\bj i} \right) \right].
\end{align*}
Thus in this setting the pluriclosed flow reduces to
\begin{align} \label{alphaflow}
\dt \ga_i =&\ g^{\bj k} \left[ \ga_{i,\bj k} - \frac{1}{2} \left( \ga_{\bj,ik} +
\ga_{k,\bj i} \right) \right].
\end{align}

\subsection{Differential Inequalities}

In this subsection we establish in Proposition \ref{Wsubsoln} that along a
solution to (\ref{alphaflow}), $W(\ga_t)$ is a matrix subsolution to a uniformly
parabolic equation.  This fact is central to the proof of Theorem \ref{EKthm},
and follows from a lengthy calculation broken up into a series of lemmas below.

\begin{lemma} \label{calcl2} Given the setup above,
 \begin{align*}
\left(  \dt - \gD_{g_{\ga}} \right) g_{i \bj} =&\ g^{\bl p} g^{\bq k} \left[
\ga_{p,\bq
\bj} \ga_{i,\bl k} - \ga_{p,\bq \bj} \ga_{\bl,ik} - \ga_{\bq,p\bj} \ga_{i,\bl k}
+
 \ga_{\bq,p\bj}  \ga_{k,\bl i} - \ga_{p,\bq i} \ga_{\bj,k \bl} + \ga_{\bq,pi}
\ga_{\bj,k\bl} \right].
 \end{align*}
\begin{proof} We compute
\begin{align*}
 \dt g_{i \bj} =&\ \i \left( \dot{\ga}_{i,\bj} - \dot{\ga}_{\bj,i} \right)\\
 =&\ \i \left[ g^{\bl k} \left( \ga_{i,\bl k} - \frac{1}{2} \left( \ga_{\bl,ik}
+ \ga_{k,\bl i} \right) \right) \right]_{,\bj} - \i \left[ g^{\bk l} \left(
\ga_{\bj,l\bk} - \frac{1}{2} \left(\ga_{l,\bj
\bk} + \ga_{\bk,l \bj} \right) \right) \right]_{,i}\\
=&\ \i g^{\bl k} \left[ \ga_{i,\bl k \bj} - \ga_{\bj,i k \bl} \right]\\
&\ - \i g^{\bl p} g_{p \bq,\bj} g^{\bq k} \left[ \ga_{i,\bl k} - \frac{1}{2}
\left( \ga_{\bl,ik}
+ \ga_{k,\bl i} \right) \right] + \i  g^{\bk p} g_{p\bq,i} g^{\bq l} \left[
\ga_{\bj,l\bk} - \frac{1}{2}
\left(\ga_{l,\bj
\bk} + \ga_{\bk,l \bj} \right) \right]\\
=&\ \i g^{\bl k} \left[ \ga_{i,\bj k \bl} - \ga_{\bj,ik\bl} \right] + g^{\bl p}
g^{\bq k} \left( \ga_{p,\bq \bj} - \ga_{\bq,p\bj} \right) \left[
\ga_{i,\bl k} - \frac{1}{2} \left( \ga_{\bl,ik}
+ \ga_{k,\bl i} \right) \right]\\
&\ - g^{\bk p} g^{\bq l} \left( \ga_{p,\bq i} - \ga_{\bq,pi} \right) \left[
\ga_{\bj,l\bk} - \frac{1}{2} \left(\ga_{l,\bj
\bk} + \ga_{\bk,l \bj} \right) \right]\\
=&\ \i g^{\bl k} \left[ \ga_{i,\bj k \bl} - \ga_{\bj,ik\bl} \right]\\
&\ + g^{\bl p} g^{\bq k} \left[ \ga_{p,\bq \bj} \ga_{i,\bl k} - \frac{1}{2}
\ga_{p,\bq \bj} \ga_{\bl,ik} - \frac{1}{2} \ga_{p,\bq \bj}  \ga_{k,\bl i}
- \ga_{\bq,p\bj} \ga_{i,\bl k} + \frac{1}{2} \ga_{\bq,p\bj} \ga_{\bl,ik} +
\frac{1}{2} \ga_{\bq,p\bj}  \ga_{k,\bl i}\right.\\
&\ \left. \qquad \qquad - \ga_{p,\bq i} \ga_{\bj,k \bl} + \frac{1}{2} \ga_{p,\bq
i} \ga_{k,\bj \bl} + \frac{1}{2} \ga_{p,\bq i} \ga_{\bl,k\bj} + \ga_{\bq,pi}
\ga_{\bj,k\bl} - \frac{1}{2} \ga_{\bq,pi} \ga_{k,\bj \bl} - \frac{1}{2}
\ga_{\bq,pi} \ga_{\bl,k\bj} \right]\\
=&\ \i g^{\bl k} \left[ \ga_{i,\bj k \bl} - \ga_{\bj,ik\bl} \right] +
\sum_{i=1}^{12} A_i.
\end{align*}
We observe that $A_2 = A_{11}, A_3 + A_8 = 0, A_5 + A_{12} = 0, A_6 = A_{9}$,
finishing the result.
 \end{proof}
\end{lemma}

\begin{lemma} \label{calcl3} Given the setup above,
\begin{align*}
\left( \dt - \gD_{g_{\ga}} \right) g^{\bs r} =&\ - g^{\bs i} g^{\bj r} g^{\bl p}
g^{\bq k}
\left[ \left( \ga_{\bq,p\bj} -
\ga_{\bj,p\bq} \right) \left(\ga_{k,\bl i} - \ga_{i,\bl k} \right) + \left(
\ga_{i,\bq \bl} - \ga_{\bq,i\bl} \right) \left(\ga_{\bj,kp} - \ga_{k,\bj p}
\right) \right].
\end{align*}
\begin{proof} First we compute using Lemma \ref{calcl2}, 
 \begin{align*}
  \dt g^{\bs r} =&\ - g^{\bs i} \dot{g}_{i \bj} g^{\bj r}\\
  =&\ - \i g^{\bs i} g^{\bj r} g^{\bl k} \left[ \ga_{i,\bj k \bl} -
\ga_{\bj,ik\bl} \right]\\
  &\ - g^{\bs i} g^{\bj r} g^{\bl p} g^{\bq k} \left[ \ga_{p,\bq \bj} \ga_{i,\bl
k} - \ga_{p,\bq
\bj} \ga_{\bl,ik} - \ga_{\bq,p\bj} \ga_{i,\bl k} +
 \ga_{\bq,p\bj}  \ga_{k,\bl i} - \ga_{p,\bq i} \ga_{\bj,k \bl} + \ga_{\bq,pi}
\ga_{\bj,k\bl} \right].
 \end{align*}
Next we observe
\begin{align*}
 \gD_{g_{\ga}} g^{\bs r} =&\ g^{\bl k} \left[ g^{\bs r} \right]_{,k\bl}\\
 =&\ g^{\bl k} \left[ - g^{\bs i} g_{i\bj,k} g^{\bj r} \right]_{,\bl}\\
 =&\ g^{\bl k} \left[ g^{\bs a} g_{a \bb,\bl} g^{\bb i} g_{i \bj,k} g^{\bj r} -
g^{\bs i} g_{i \bj,k\bl} g^{\bj r} + g^{\bs i} g_{i \bj,k} g^{\bj a} g_{a
\bb,\bl} g^{\bb r} \right]\\
=&\ - \i g^{\bl k} g^{\bs i} g^{\br j} \left[ \ga_{i,\bj k \bl} - \ga_{\bj,i k
\bl} \right] - g^{\bl k} g^{\bs a} g^{\bb i} g^{\bj r} \left( \ga_{a,\bb \bl} -
\ga_{\bb, a\bl} \right) \left( \ga_{i,\bj k} - \ga_{\bj,ik} \right)\\
&\ - g^{\bl k} g^{\bs i} g^{\bj a} g^{\bb r} \left( \ga_{i,\bj k} - \ga_{\bj,ik}
\right) \left( \ga_{a,\bb \bl} - \ga_{\bb,a \bl} \right)\\
=&\ - \i g^{\bl k} g^{\bs i} g^{\br j} \left[ \ga_{i,\bj k \bl} - \ga_{\bj,i k
\bl} \right]\\
&\ - g^{\bl k} g^{\bs a} g^{\bb i} g^{\bj r} \left[ \ga_{a,\bb \bl} \ga_{i,\bj
k} - \ga_{a,\bb \bl} \ga_{\bj,ik} - \ga_{\bb,a \bl} \ga_{i,\bj k} +
\ga_{\bb,a\bl} \ga_{\bj,ik} \right]\\
&\ - g^{\bl k} g^{\bs i} g^{\bj a} g^{\bb r} \left[ \ga_{i,\bj k} \ga_{a,\bb
\bl} - \ga_{i,\bj k} \ga_{\bb,a \bl} - \ga_{\bj,ik} \ga_{a,\bb\bl} +
\ga_{\bj,ik} \ga_{\bb,a\bl} \right].
\end{align*}
Combining the two calculations above yields
\begin{align*}
 \left( \dt - \gD_{g_{\ga}} \right) g^{\bs r} =&\ - g^{\bs i} g^{\bj r} g^{\bl
p} g^{\bq
k} \left[ \ga_{p,\bq \bj} \ga_{i,\bl
k} - \ga_{p,\bq
\bj} \ga_{\bl,ik} - \ga_{\bq,p\bj} \ga_{i,\bl k} +
 \ga_{\bq,p\bj}  \ga_{k,\bl i} - \ga_{p,\bq i} \ga_{\bj,k \bl} + \ga_{\bq,pi}
\ga_{\bj,k\bl} \right]\\
&\ + g^{\bl k} g^{\bs a} g^{\bb i} g^{\bj r} \left[ \ga_{a,\bb \bl} \ga_{i,\bj
k} - \ga_{a,\bb \bl} \ga_{\bj,ik} - \ga_{\bb,a \bl} \ga_{i,\bj k} +
\ga_{\bb,a\bl} \ga_{\bj,ik} \right]\\
&\ + g^{\bl k} g^{\bs i} g^{\bj a} g^{\bb r} \left[ \ga_{i,\bj k} \ga_{a,\bb
\bl} - \ga_{i,\bj k} \ga_{\bb,a \bl} - \ga_{\bj,ik} \ga_{a,\bb\bl} +
\ga_{\bj,ik} \ga_{\bb,a\bl} \right]\\
=&\ \sum_{i=1}^{14} A_i.
\end{align*}
We observe the cancellations $A_1 + A_{11} = A_2 + A_{13} = A_6 + A_{14} = 0,$
leaving
\begin{align*}
&  \left( \dt - \gD_{g_{\ga}} \right) g^{\bs r}\\
=&\ - g^{\bs i} g^{\bj r} g^{\bl p} g^{\bq
k} \left[ - \ga_{\bq,p\bj} \ga_{i,\bl k} +
 \ga_{\bq,p\bj}  \ga_{k,\bl i} - \ga_{p,\bq i} \ga_{\bj,k \bl} \right]\\
&\ + g^{\bl k} g^{\bs a} g^{\bb i} g^{\bj r} \left[ \ga_{a,\bb\bl} \ga_{i,\bj k}
- \ga_{a,\bb \bl} \ga_{\bj,ik} - \ga_{\bb,a \bl} \ga_{i,\bj k} + \ga_{\bb,a\bl}
\ga_{\bj,ik} \right]\\
&\ + g^{\bl k} g^{\bs i} g^{\bj a} g^{\bb r} \left[ - \ga_{i,\bj k} \ga_{\bb,a
\bl}  \right]\\
=&\ g^{\bs i} g^{\bj r} g^{\bl p} g^{\bq k} \left[ \ga_{\bq,p\bj} \ga_{i,\bl k}
-
 \ga_{\bq,p\bj}  \ga_{k,\bl i} + \ga_{p,\bq i} \ga_{\bj,k \bl} + \ga_{i,\bq \bl}
\ga_{k,\bj p} \right.\\
 &\ \qquad \qquad \quad \left. - \ga_{i,\bq \bl} \ga_{\bj,kp} - \ga_{\bq,i \bl}
\ga_{k,\bj p} + \ga_{\bq,i \bl} \ga_{\bj,k p} - \ga_{i,\bq p} \ga_{\bj,k \bl}
\right]\\
 =&\ - g^{\bs i} g^{\bj r} g^{\bl p} g^{\bq k} \left[ \left( \ga_{\bq,p\bj} -
\ga_{\bj,p\bq} \right) \left(\ga_{k,\bl i} - \ga_{i,\bl k} \right) + \left(
\ga_{i,\bq \bl} - \ga_{\bq,i\bl} \right) \left(\ga_{\bj,kp} - \ga_{k,\bj p}
\right) \right],
\end{align*}
as claimed.
\end{proof}
\end{lemma}

\begin{lemma} \label{calcl4}
 Given the setup above,
\begin{align*}
\left( \dt - \gD_{g_{\ga}} \right) \i \del \ga_{uv} =&\ g^{\bl p} g^{\bq k}
\left[
\ga_{p,\bq v} \ga_{u,\bl k} - \ga_{p,\bq v} \ga_{\bl,uk} - \ga_{\bq,pv}
\ga_{u,\bl k} + \ga_{\bq,pv} \ga_{k,\bl u} - \ga_{p,\bq u} \ga_{v,\bl k} +
\ga_{\bq,pu} \ga_{v,\bl k} \right],\\
\left( \dt - \gD_{g_{\ga}} \right) \left( - \i \delb \bga_{\bu \bv} \right) =&\
g^{\bl p}
g^{\bq k} \left[ \ga_{p,\bq \bu} \ga_{\bv,\bl k} - \ga_{p,\bq \bu}
\ga_{\bl,k\bv} - \ga_{\bq,p\bu} \ga_{\bv,\bl k} + \ga_{\bq,p\bu} \ga_{k,\bv \bl}
- \ga_{p,\bq \bv} \ga_{\bu,\bl k} + \ga_{\bq,p\bv} \ga_{\bu,\bl k} \right].
\end{align*}
\begin{proof} We establish the first equation, with the second following by
conjugation.  We directly compute
\begin{align*}
 \dt \i \del \ga_{uv} =&\ \i \dot{\ga}_{u,v} - \i \dot{\ga}_{v,u}\\
 =&\ \i \left[ g^{\bl k} \left( \ga_{u,\bl k} - \frac{1}{2} \left( \ga_{\bl, u
k} + \ga_{k,\bl u} \right) \right) \right]_{,v} - \i \left[ g^{\bl k} \left(
\ga_{v,\bl k} - \frac{1}{2} \left( \ga_{\bl, v k} + \ga_{k,\bl v} \right)
\right) \right]_{,u}\\
=&\ \gD_{g_{\ga}} \i \del \ga_{uv} - \i g^{\bl p} g_{p \bq,v} g^{\bq k} \left(
\ga_{u,\bl
k} - \frac{1}{2} \left( \ga_{\bl, u k} + \ga_{k,\bl u} \right) \right)\\
&\ + \i g^{\bl p} g_{p \bq,u} g^{\bq k} \left( \ga_{v,\bl k} - \frac{1}{2}
\left( \ga_{\bl, v k} + \ga_{k,\bl v} \right) \right)\\
=&\ \gD_{g_{\ga}} \i \del \ga_{uv} + g^{\bl p} g^{\bq k} \left( \ga_{p,\bq v} -
\ga_{\bq,pv} \right) \left( \ga_{u,\bl k} - \frac{1}{2} \left( \ga_{\bl,uk} +
\ga_{k,\bl u} \right) \right)\\
&\ - g^{\bl p} g^{\bq k} \left( \ga_{p,\bq u} - \ga_{\bq,pu} \right) \left(
\ga_{v,\bl k} - \frac{1}{2} \left( \ga_{\bl, v k} + \ga_{k,\bl v} \right)
\right)\\
=&\ \gD_{g_{\ga}} \i \del \ga_{uv}\\
&\ + g^{\bl p} g^{\bq k} \left[ \ga_{p,\bq v} \ga_{u,\bl k} - \frac{1}{2}
\ga_{p,\bq v} \ga_{\bl,uk} - \frac{1}{2} \ga_{p,\bq v} \ga_{k,\bl u} -
\ga_{\bq,pv} \ga_{u,\bl k} + \frac{1}{2} \ga_{\bq,pv} \ga_{\bl,uk} + \frac{1}{2}
\ga_{\bq,pv} \ga_{k,\bl u} \right.\\
&\ \qquad \quad \left. - \ga_{p,\bq u} \ga_{v,\bl k} + \frac{1}{2} \ga_{p,\bq u}
\ga_{\bl,vk} + \frac{1}{2} \ga_{p,\bq u} \ga_{k,\bl v} + \ga_{\bq,pu} \ga_{v,\bl
k} - \frac{1}{2} \ga_{\bq,pu} \ga_{\bl,vk} - \frac{1}{2} \ga_{\bq,pu} \ga_{k,\bl
v} \right]\\
=&\ \gD_{g_{\ga}} \i \del \ga_{uv} + \sum_{i=1}^{12} A_i.
\end{align*}
We observe that $A_2 = A_{12}, A_3 + A_9 = 0, A_5 + A_{11} = 0, A_6 = A_8$.  The
result follows.
\end{proof}
\end{lemma}

\begin{lemma} \label{calcl5}
 Given the setup above,
\begin{align*}
\left( \dt - \gD_{g_{\ga}} \right) & \i \del \ga_{ur} g^{\bs r}\\
=&\ - g^{\bl k} g^{\bs p} g^{\bq r} \left[ \left( \ga_{u,rk} - \ga_{r,uk}
\right) + \gb_u^{\bv} \left( \ga_{\bv,rk} - \ga_{r,\bv k} \right) \right] \left(
\ga_{p,\bq \bl} - \ga_{\bq,p\bl} \right)\\
&\ + g^{\bl k} g^{\bs p} g^{\bq r} \left[ \left( \ga_{u,\bl r} - \ga_{\bl,ur}
\right) + \gb_u^{\bv} \left( \ga_{\bv, r \bl} - \ga_{\bl,r\bv} \right) \right]
\left( \ga_{k,\bq p} - \ga_{p,\bq k} \right).
\end{align*}
 \begin{proof} First we compute
 \begin{align*}
  \gD_{g_{\ga}} \i \del \ga_{u r} g^{\bs r} =&\ g^{\bl k} \left( \i \del
\ga_{ur} g^{\bs
r} \right)_{,k\bl}\\
  =&\ \i g^{\bl k} \left[ (\del \ga)_{ur,k\bl} g^{\bs r} + (\del \ga)_{ur,k}
(g^{\bs r})_{,\bl} + (\del \ga)_{ur,\bl} (g^{\bs r})_{,k} + \del \ga_{ur}
(g^{\bs r})_{,k\bl} \right]\\
  =&\ \gD_{g_{\ga}} \left[ \i \del \ga_{ur} \right] g^{\bs r} + \i \del \ga_{ur}
\gD_{g_{\ga}}
g^{\bs r}\\
  &\ + \i g^{\bl k} \left[ - \left( \ga_{u,r k} - \ga_{r,uk} \right) g^{\bs p}
g_{p\bq,\bl} g^{\bq r} - (\ga_{u,r \bl} - \ga_{r,u\bl} ) g^{\bs p} g_{p\bq,k}
g^{\bq r} \right]\\
=&\ \gD_{g_{\ga}} \left[ \i \del \ga_{ur} \right] g^{\bs r} + \i \del \ga_{ur}
\gD_{g_{\ga}}
g^{\bs r}\\
&\ + g^{\bl k} g^{\bs p} g^{\bq r} \left[ \left( \ga_{u,rk} - \ga_{r,uk} \right)
\left( \ga_{p,\bq \bl} - \ga_{\bq,p \bl} \right) + \left( \ga_{u,r\bl} -
\ga_{r,u\bl} \right) \left( \ga_{p,\bq k} - \ga_{\bq,pk} \right) \right].
 \end{align*}
Using this in conjunction with Lemmas \ref{calcl3}, \ref{calcl4} we yield
\begin{align*}
& \left( \dt - \gD_{g_{\ga}} \right) \i \del \ga_{u r} g^{\bs r}\\
=&\ \left[ \left( \dt -
\gD_{g_{\ga}} \right) \i \del \ga_{u r} \right] g^{\bs r} + \left[ \i \del
\ga_{ur}
\right] \left( \dt - \gD_{g_{\ga}} \right) g^{\bs r}\\
 &\ - g^{\bl k} g^{\bs p} g^{\bq r} \left[ \left( \ga_{u,rk} - \ga_{r,uk}
\right)
\left( \ga_{p,\bq \bl} - \ga_{\bq,p \bl} \right) + \left( \ga_{u,r\bl} -
\ga_{r,u\bl} \right) \left( \ga_{p,\bq k} - \ga_{\bq,pk} \right) \right]\\
=&\ g^{\bl p} g^{\bq k} g^{\bs r} \left[
\ga_{p,\bq r} \ga_{u,\bl k} - \ga_{p,\bq r} \ga_{\bl,uk} - \ga_{\bq,pr}
\ga_{u,\bl k} + \ga_{\bq,pr} \ga_{k,\bl u} - \ga_{p,\bq u} \ga_{r,\bl k} +
\ga_{\bq,pu} \ga_{r,\bl k} \right]\\
&\ - \left[ \i \del \ga \right]_{ur} g^{\bs i} g^{\bj r} g^{\bl p} g^{\bq k}
\left[ \left( \ga_{\bq,p\bj} -
\ga_{\bj,p\bq} \right) \left(\ga_{k,\bl i} - \ga_{i,\bl k} \right) + \left(
\ga_{i,\bq \bl} - \ga_{\bq,i\bl} \right) \left(\ga_{\bj,kp} - \ga_{k,\bj p}
\right) \right]\\
&\ - g^{\bl k} g^{\bs p} g^{\bq r} \left[ \left( \ga_{u,rk} - \ga_{r,uk}
\right)
\left( \ga_{p,\bq \bl} - \ga_{\bq,p \bl} \right) + \left( \ga_{u,r\bl} -
\ga_{r,u\bl} \right) \left( \ga_{p,\bq k} - \ga_{\bq,pk} \right) \right]\\
=&\ - \left[ \i \del \ga \right]_{ur} g^{\bs i} g^{\bj r} g^{\bl p} g^{\bq k}
\left[ \left( \ga_{\bq,p\bj} -
\ga_{\bj,p\bq} \right) \left(\ga_{k,\bl i} - \ga_{i,\bl k} \right) + \left(
\ga_{i,\bq \bl} - \ga_{\bq,i\bl} \right) \left(\ga_{\bj,kp} - \ga_{k,\bj p}
\right) \right]\\
&\ + g^{\bl k} g^{\bs p} g^{\bq r} \left[ 
\ga_{k,\bq p} \ga_{u,\bl r} - \ga_{k,\bq p} \ga_{\bl,ur} - \ga_{\bq,kp}
\ga_{u,\bl r} + \ga_{\bq,kp} \ga_{r,\bl u} - \ga_{k,\bq u} \ga_{p,\bl r} +
\ga_{\bq,ku} \ga_{p,\bl r} \right.\\
&\ \qquad \qquad \left. - \ga_{u,rk} \ga_{p,\bq \bl} + \ga_{u,rk} \ga_{\bq,p\bl}
+ \ga_{r,uk} \ga_{p,\bq\bl} - \ga_{r,uk} \ga_{\bq,p\bl} \right.\\
&\ \qquad \qquad \left. - \ga_{u,r\bl} \ga_{p,\bq k} + \ga_{u,r\bl} \ga_{\bq,pk}
+ \ga_{r,u\bl} \ga_{p,\bq k} - \ga_{r,u\bl} \ga_{\bq,pk} \right]\\
=&\ - \left[ \i \del \ga \right]_{ur} g^{\bs i} g^{\bj r} g^{\bl p} g^{\bq k}
\left[ \left( \ga_{\bq,p\bj} -
\ga_{\bj,p\bq} \right) \left(\ga_{k,\bl i} - \ga_{i,\bl k} \right) + \left(
\ga_{i,\bq \bl} - \ga_{\bq,i\bl} \right) \left(\ga_{\bj,kp} - \ga_{k,\bj p}
\right) \right]\\
&\ + \sum_{i=1}^{14} A_i.
\end{align*}
 We observe the cancellations $A_3 + A_{12} = A_4 + A_{14} = A_5 + A_{13} = 0$,
and apply further simplifications to yield
 \begin{align*}
& \left( \dt - \gD_{g_{\ga}} \right) \i \del \ga_{ur} g^{\bs r}\\
=&\ - \left[ \i \del \ga \right]_{ur} g^{\bs i} g^{\bj r} g^{\bl p} g^{\bq k}
\left[ \left( \ga_{\bq,p\bj} -
\ga_{\bj,p\bq} \right) \left(\ga_{k,\bl i} - \ga_{i,\bl k} \right) + \left(
\ga_{i,\bq \bl} - \ga_{\bq,i\bl} \right) \left(\ga_{\bj,kp} - \ga_{k,\bj p}
\right) \right]\\
&\ + g^{\bl k} g^{\bs p} g^{\bq r} \left[ 
\ga_{k,\bq p} \ga_{u,\bl r} - \ga_{k,\bq p} \ga_{\bl,ur} +
\ga_{\bq,ku} \ga_{p,\bl r} - \ga_{u,rk} \ga_{p,\bq \bl} \right.\\
&\ \qquad \qquad \left.  + \ga_{u,rk} \ga_{\bq,p\bl}
+ \ga_{r,uk} \ga_{p,\bq\bl} - \ga_{r,uk} \ga_{\bq,p\bl} - \ga_{u,r\bl}
\ga_{p,\bq k} \right]\\
=&\ -\gb_u^{\bj} g^{\bs i} g^{\bl p} g^{\bq k} \left[ \left( \ga_{\bq,p\bj} -
\ga_{\bj,p\bq} \right) \left(\ga_{k,\bl i} - \ga_{i,\bl k} \right) + \left(
\ga_{i,\bq \bl} - \ga_{\bq,i\bl} \right) \left(\ga_{\bj,kp} - \ga_{k,\bj p}
\right) \right]\\
&\ + g^{\bl k} g^{\bs p} g^{\bq r} \left[ \left( \ga_{r,uk} - \ga_{u,rk} \right)
\left( \ga_{p,\bq \bl} - \ga_{\bq,p\bl} \right) + \left(\ga_{u,\bl r} -
\ga_{\bl,ur} \right) \left( \ga_{k,\bq p} - \ga_{p,\bq k} \right)
\right]\\
=&\ -\gb_u^{\bv} g^{\bs p} g^{\bl k} g^{\bq r} \left[ \left( \ga_{\bq,k\bv} -
\ga_{\bv,k\bq} \right) \left(\ga_{r,\bl p} - \ga_{p,\bl r} \right) + \left(
\ga_{p,\bq \bl} - \ga_{\bq,p\bl} \right) \left(\ga_{\bv,rk} - \ga_{r,\bv k}
\right) \right]\\
&\ + g^{\bl k} g^{\bs p} g^{\bq r} \left[ \left( \ga_{r,uk} - \ga_{u,rk} \right)
\left( \ga_{p,\bq \bl} - \ga_{\bq,p\bl} \right) + \left(\ga_{u,\bl r} -
\ga_{\bl,ur} \right) \left( \ga_{k,\bq p} - \ga_{p,\bq k} \right)
\right]\\
=&\ - g^{\bl k} g^{\bs p} g^{\bq r} \left[ \left( \ga_{u,rk} - \ga_{r,uk}
\right) + \gb_u^{\bv} \left( \ga_{\bv,rk} - \ga_{r,\bv k} \right) \right] \left(
\ga_{p,\bq \bl} - \ga_{\bq,p\bl} \right)\\
&\ + g^{\bl k} g^{\bs p} g^{\bq r} \left[ \left( \ga_{u,\bl r} - \ga_{\bl,ur}
\right) + \gb_u^{\bv} \left( \ga_{\bv, r \bl} - \ga_{\bl,r\bv} \right) \right]
\left( \ga_{k,\bq p} - \ga_{p,\bq k} \right),
\end{align*}
as required.
\end{proof}
\end{lemma}

\begin{lemma} \label{calcl7} Given the setup above,
\begin{align*}
& \left( \dt - \gD_{g_{\ga}} \right) \left( \del \ga_{ar} g^{\bs r} \delb
\bga_{\bs \bb}
\right)\\
=&\ - \del \ga_{ar} \delb \bga_{\bs \bb} g^{\bs i} g^{\bj r} g^{\bl p} g^{\bq k}
\left[ \left( \ga_{\bq,p\bj} -
\ga_{\bj,p\bq} \right) \left(\ga_{k,\bl i} - \ga_{i,\bl k} \right) + \left(
\ga_{i,\bq \bl} - \ga_{\bq,i\bl} \right) \left(\ga_{\bj,kp} - \ga_{k,\bj p}
\right) \right]\\
&\ - g^{\bl k} g^{\bs r} \left[ \left( \ga_{a,rk} - \ga_{r,ak} \right) \left(
\ga_{\bs,\bb\bl} - \ga_{\bb,\bs \bl} \right) + \left( \ga_{a,r\bl} -
\ga_{r,a\bl} \right) \left( \ga_{\bs,\bb k} - \ga_{\bb,\bs k} \right) \right]\\
&\ - \i \delb \bga_{\bs \bb} g^{\bl k} g^{\bs p} g^{\bq r} \left[ 
\ga_{k,\bq p} \ga_{a,\bl r} - \ga_{k,\bq p} \ga_{\bl,ar} +
\ga_{\bq,ka} \ga_{p,\bl r} - \ga_{a,rk} \ga_{p,\bq \bl} \right.\\
&\ \qquad \qquad \left.  + \ga_{a,rk} \ga_{\bq,p\bl}
+ \ga_{r,ak} \ga_{p,\bq\bl} - \ga_{r,ak} \ga_{\bq,p\bl} - \ga_{a,r\bl}
\ga_{p,\bq k} \right]\\
&\ + \i \del \ga_{a r} g^{\bl k} g^{\bs p} g^{\bq r} \left[ \ga_{p,\bq k}
\ga_{\bs,\bb \bl} - \ga_{p,\bq k} \ga_{\bb,\bs \bl} - \ga_{\bq,pk}
\ga_{\bs,\bb\bl} + \ga_{\bq,pk} \ga_{\bb,\bs\bl} \right.\\
&\ \qquad \qquad \qquad \left. + \ga_{\bq,p\bl} \ga_{\bb,\bs k} - \ga_{\bl,p\bq}
\ga_{\bb,\bs k} + \ga_{\bl,p\bq} \ga_{k,\bb \bs} - \ga_{p,\bl \bb} \ga_{\bq,\bs
k} \right].
\end{align*}
\begin{proof} We first of all compute
\begin{align*}
\gD_{g_{\ga}} & \left( \del \ga_{ar} g^{\bs r} \delb \bga_{\bs \bb} \right)\\
= &\ g^{\bl k} \left[ \left( \del \ga_{ar} g^{\bs r} \right) \delb \bga_{\bs
\bb} \right]_{,k\bl}\\
=&\ \gD_{g_{\ga}} \left[ \del \ga_{a r} g^{\bs r} \right] \delb \bga_{\bs b} +
\left( \del
\ga_{ar} g^{\bs r} \right) \gD_{g_{\ga}} \delb \bga_{\bs \bb} + g^{\bl k} \left(
\del
\ga_{a r} g^{\bs r} \right)_{,k} (\delb \bga)_{\bs \bb,\bl} + g^{\bl k} \left(
\del \ga_{a r} g^{\bs r} \right)_{,\bl} (\delb \bga)_{\bs \bb,k}\\
=&\ \gD_{g_{\ga}} \left[ \del \ga_{a r} g^{\bs r} \right] \delb \bga_{\bs b} +
\left( \del
\ga_{ar} g^{\bs r} \right) \gD_{g_{\ga}} \delb \bga_{\bs \bb}\\
&\ + g^{\bl k} g^{\bs r} \left[ \left( \ga_{a,rk} - \ga_{r,ak} \right) \left(
\ga_{\bs,\bb\bl} - \ga_{\bb,\bs \bl} \right) + \left( \ga_{a,r\bl} -
\ga_{r,a\bl} \right) \left( \ga_{\bs,\bb k} - \ga_{\bb,\bs k} \right) \right]\\
&\ - \del \ga_{a r} g^{\bl k} g^{\bs p} g_{p \bq,k} g^{\bq r} \left(
\ga_{\bs,\bb\bl} - \ga_{\bb,\bs \bl} \right) - \del \ga_{a r} g^{\bl k}  g^{\bs
p} g_{p \bq,\bl} g^{\bq r} \left( \ga_{\bs,\bb k} - \ga_{\bb,\bs k} \right)\\
=&\ \gD_{g_{\ga}} \left[ \del \ga_{a r} g^{\bs r} \right] \delb \bga_{\bs b} +
\left( \del
\ga_{ar} g^{\bs r} \right) \gD_{g_{\ga}} \delb \bga_{\bs \bb}\\
&\ + g^{\bl k} g^{\bs r} \left[ \left( \ga_{a,rk} - \ga_{r,ak} \right) \left(
\ga_{\bs,\bb\bl} - \ga_{\bb,\bs \bl} \right) + \left( \ga_{a,r\bl} -
\ga_{r,a\bl} \right) \left( \ga_{\bs,\bb k} - \ga_{\bb,\bs k} \right) \right]\\
&\ - \i \del \ga_{a r} g^{\bl k} g^{\bs p} g^{\bq r} \left( \ga_{p,\bq k} -
\ga_{\bq,pk} \right) \left( \ga_{\bs,\bb\bl} -\ga_{\bb,\bs \bl} \right)\\
&\  - \i \del \ga_{a r} g^{\bl k}  g^{\bs p} g^{\bq r} \left( \ga_{p,\bq \bl} -
\ga_{\bq,p\bl} \right) \left( \ga_{\bs,\bb k} - \ga_{\bb,\bs k} \right).
\end{align*}
Also we have the basic calculation
\begin{align*}
& \left( \dt - \gD_{g_{\ga}} \right) \left( \del \ga_{ar} g^{\bs r} \delb
\bga_{\bs \bb}
\right)\\
=&\ \left[ \dt \left( \del \ga_{ar} g^{\bs r} \right) \right] \delb \bga_{\bs
\bb} + \left( \del \ga_{ar} g^{\bs r} \right) \dt \delb \bga_{\bs \bb} -
\gD_{g_{\ga}}
\left( \del \ga_{ar} g^{\bs r} \delb \bga_{\bs \bb} \right)\\
=&\ \left[ \left( \dt - \gD_{g_{\ga}} \right) \left( \i \del \ga_{a r} g^{\bs r}
\right)
\right] \left( - \i \delb \bga_{\bs \bb} \right) + \left( \i \del \ga_{a r}
g^{\bs r} \right) \left( \dt - \gD_{g_{\ga}} \right) (- \i \delb \bga_{\bs
\bb})\\
&\ + \left[ \left( \gD_{g_{\ga}} \del \ga_{a r} g^{\bs r} \right) \delb
\bga_{\bs \bb} +
\left( \del \ga_{a r} g^{\bs r} \right) \gD_{g_{\ga}} \delb \bga_{\bs \bb} -
\gD_{g_{\ga}} \left(
\del \ga_{ar} g^{\bs r} \delb \bga_{\bs \bb} \right) \right]\\
=&\ E_1 + E_2 + E_3.
\end{align*}
Using Lemma \ref{calcl5} we compute
\begin{align*}
E_1 =&\ - \del \ga_{ar} \delb \bga_{\bs \bb} g^{\bs i} g^{\bj r} g^{\bl p}
g^{\bq k}
\left[ \left( \ga_{\bq,p\bj} -
\ga_{\bj,p\bq} \right) \left(\ga_{k,\bl i} - \ga_{i,\bl k} \right) + \left(
\ga_{i,\bq \bl} - \ga_{\bq,i\bl} \right) \left(\ga_{\bj,kp} - \ga_{k,\bj p}
\right) \right]\\
&\ - \i \delb \bga_{\bs \bb} g^{\bl k} g^{\bs p} g^{\bq r} \left[ 
\ga_{k,\bq p} \ga_{a,\bl r} - \ga_{k,\bq p} \ga_{\bl,ar} +
\ga_{\bq,ka} \ga_{p,\bl r} - \ga_{a,rk} \ga_{p,\bq \bl} \right.\\
&\ \qquad \qquad \left.  + \ga_{a,rk} \ga_{\bq,p\bl}
+ \ga_{r,ak} \ga_{p,\bq\bl} - \ga_{r,ak} \ga_{\bq,p\bl} - \ga_{a,r\bl}
\ga_{p,\bq k} \right].
\end{align*}
Also  using Lemma \ref{calcl4} we have
\begin{align*}
E_2 =&\ \i \del \ga_{a r} g^{\bs r} g^{\bl p} g^{\bq k} \left[ \ga_{p,\bq \bs}
\ga_{\bb,\bl k} - \ga_{p,\bq \bs} \ga_{\bl,k\bb} - \ga_{\bq,p\bs} \ga_{\bb,\bl
k} + \ga_{\bq,p\bs} \ga_{k,\bb \bl} - \ga_{p,\bq \bb} \ga_{\bs,\bl k} +
\ga_{\bq,p\bb} \ga_{\bs,\bl k} \right].
\end{align*}
Also from the calculation above we have
\begin{align*}
E_3 =&\ - g^{\bl k} g^{\bs r} \left[ \left( \ga_{a,rk} - \ga_{r,ak} \right)
\left( \ga_{\bs,\bb\bl} - \ga_{\bb,\bs \bl} \right) + \left( \ga_{a,r\bl} -
\ga_{r,a\bl} \right) \left( \ga_{\bs,\bb k} - \ga_{\bb,\bs k} \right) \right]\\
&\ + \i \del \ga_{a r} g^{\bl k} g^{\bs p} g^{\bq r} \left( \ga_{p,\bq k} -
\ga_{\bq,pk} \right) \left( \ga_{\bs,\bb\bl} -\ga_{\bb,\bs \bl} \right)\\
&\  + \i \del \ga_{a r} g^{\bl k}  g^{\bs p} g^{\bq r} \left( \ga_{p,\bq \bl} -
\ga_{\bq,p\bl} \right) \left( \ga_{\bs,\bb k} - \ga_{\bb,\bs k} \right)\\
=&\ - g^{\bl k} g^{\bs r} \left[ \left( \ga_{a,rk} - \ga_{r,ak} \right) \left(
\ga_{\bs,\bb\bl} - \ga_{\bb,\bs \bl} \right) + \left( \ga_{a,r\bl} -
\ga_{r,a\bl} \right) \left( \ga_{\bs,\bb k} - \ga_{\bb,\bs k} \right) \right]\\
&\ + \i \del \ga_{a r} g^{\bl k} g^{\bs p} g^{\bq r} \left[ \ga_{p,\bq k}
\ga_{\bs,\bb \bl} - \ga_{p,\bq k} \ga_{\bb,\bs \bl} - \ga_{\bq,pk}
\ga_{\bs,\bb\bl} + \ga_{\bq,pk} \ga_{\bb,\bs\bl} \right.\\
&\ \qquad \qquad \qquad \left. + \ga_{p,\bq \bl} \ga_{\bs,\bb k} - \ga_{p,\bq
\bl} \ga_{\bb,\bs k} - \ga_{\bq,p\bl} \ga_{\bs,\bb k} + \ga_{\bq,p\bl}
\ga_{\bb,\bs k} \right].
\end{align*}
Collecting these calculations and relabeling indices yields
\begin{align*}
& \left( \dt - \gD_{g_{\ga}} \right) \left( \del \ga_{ar} g^{\bs r} \delb
\bga_{\bs \bb}
\right)\\
=&\ - \del \ga_{ar} \delb \bga_{\bs \bb} g^{\bs i} g^{\bj r} g^{\bl p} g^{\bq k}
\left[ \left( \ga_{\bq,p\bj} -
\ga_{\bj,p\bq} \right) \left(\ga_{k,\bl i} - \ga_{i,\bl k} \right) + \left(
\ga_{i,\bq \bl} - \ga_{\bq,i\bl} \right) \left(\ga_{\bj,kp} - \ga_{k,\bj p}
\right) \right]\\
&\ - g^{\bl k} g^{\bs r} \left[ \left( \ga_{a,rk} - \ga_{r,ak} \right) \left(
\ga_{\bs,\bb\bl} - \ga_{\bb,\bs \bl} \right) + \left( \ga_{a,r\bl} -
\ga_{r,a\bl} \right) \left( \ga_{\bs,\bb k} - \ga_{\bb,\bs k} \right) \right]\\
&\ - \i \delb \bga_{\bs \bb} g^{\bl k} g^{\bs p} g^{\bq r} \left[ 
\ga_{k,\bq p} \ga_{a,\bl r} - \ga_{k,\bq p} \ga_{\bl,ar} +
\ga_{\bq,ka} \ga_{p,\bl r} - \ga_{a,rk} \ga_{p,\bq \bl} \right.\\
&\ \qquad \qquad \qquad \left.  + \ga_{a,rk} \ga_{\bq,p\bl}
+ \ga_{r,ak} \ga_{p,\bq\bl} - \ga_{r,ak} \ga_{\bq,p\bl} - \ga_{a,r\bl}
\ga_{p,\bq k} \right]\\
&\ + \i \del \ga_{a r} g^{\bl k} g^{\bs p} g^{\bq r} \left[ \ga_{p,\bq k}
\ga_{\bs,\bb \bl} - \ga_{p,\bq k} \ga_{\bb,\bs \bl} - \ga_{\bq,pk}
\ga_{\bs,\bb\bl} + \ga_{\bq,pk} \ga_{\bb,\bs\bl} \right.\\
&\ \qquad \qquad \qquad \left. + \ga_{p,\bq \bl} \ga_{\bs,\bb k} - \ga_{p,\bq
\bl} \ga_{\bb,\bs k} - \ga_{\bq,p\bl} \ga_{\bs,\bb k} + \ga_{\bq,p\bl}
\ga_{\bb,\bs k} \right.\\
&\ \left. \qquad \qquad \qquad + \ga_{p,\bl \bq} \ga_{\bb,\bs k} - \ga_{p,\bl
\bq} \ga_{\bs,k\bb} - \ga_{\bl,p\bq} \ga_{\bb,\bs k} + \ga_{\bl,p\bq} \ga_{k,\bb
\bs} - \ga_{p,\bl \bb} \ga_{\bq,\bs k} + \ga_{\bl,p\bb} \ga_{\bq,\bs k} \right].
\end{align*}
Let us label the final 14 terms above as $\sum_{i=1}^{14} A_i$ as usual.  Then
observe the cancellations $A_5 + A_{10} = A_6 + A_9 = A_7 + A_{14} = 0$.  The
result follows.
\end{proof}
\end{lemma}

\begin{lemma} \label{calcl8} Given the setup above,
\begin{align*}
& \left( \dt - \gD_{g_{\ga}} \right) \left[ g_{a \bb} - \del \ga_{ar} g^{\bs r}
\delb
\bga_{\bs \bb} \right]\\
=&\ - g^{\bl k} g^{\bs r} \left[ \left( \ga_{a,r\bl} - \ga_{\bl,ar} \right) +
\gb_a^{\bq} \left( \ga_{\bq,r \bl} - \ga_{\bl,r \bq} \right) \right] \left[
\left( \ga_{\bb,\bs k} - \ga_{k,\bs \bb} \right) + \gb_{\bb}^q \left( \ga_{q,\bs
k} - \ga_{k,\bs p} \right) \right]\\
&\ - g^{\bl k} g^{\bs r} \left[ \left( \ga_{a,rk} - \ga_{r,ak} \right) +
\gb_a^{\bq} \left( \ga_{\bq,rk} - \ga_{r,\bq k} \right) \right] \left[ \left(
\ga_{\bb,\bs \bl} - \ga_{\bs,\bb\bl} \right) + \gb_{\bb}^{q} \left( \ga_{q,\bs
\bl} - \ga_{\bs,q \bl} \right) \right].
\end{align*}
\begin{proof} We combine the results of Lemmas \ref{calcl2}, \ref{calcl7}. 
First we rewrite the $g^{-2} \del^2 \ga^{*2}$ terms.  These take the form
\begin{align*}
E =&\ g^{\bl p} g^{\bq k} \left[ \ga_{p,\bq
\bb} \ga_{a,\bl k} - \ga_{p,\bq \bb} \ga_{\bl,ak} - \ga_{\bq,p\bb} \ga_{a,\bl k}
+
 \ga_{\bq,p\bb}  \ga_{k,\bl a} - \ga_{p,\bq a} \ga_{\bb,k \bl} + \ga_{\bq,pa}
\ga_{\bb,k\bl} \right]\\
&\ + g^{\bl k} g^{\bs r} \left[ \left( \ga_{a,rk} - \ga_{r,ak} \right) \left(
\ga_{\bs,\bb\bl} - \ga_{\bb,\bs \bl} \right) + \left( \ga_{a,r\bl} -
\ga_{r,a\bl} \right) \left( \ga_{\bs,\bb k} - \ga_{\bb,\bs k} \right) \right]\\
=&\ g^{\bl k} g^{\bs r} \left[ \left( \ga_{a,rk} - \ga_{r,ak} \right) \left(
\ga_{\bs,\bb\bl} - \ga_{\bb,\bs \bl} \right)  + \ga_{a,r\bl} \ga_{\bs,\bb k} -
\ga_{a,r\bl} \ga_{\bb,\bs k} - \ga_{r,a\bl} \ga_{\bs,\bb k} + \ga_{r,a\bl}
\ga_{\bb,\bs k} \right.\\
&\ \left. + \ga_{k,\bs
\bb} \ga_{a,\bl r} - \ga_{k,\bs \bb} \ga_{\bl,ar} - \ga_{\bs,k\bb} \ga_{a,\bl r}
+
 \ga_{\bs,k\bb}  \ga_{r,\bl a} - \ga_{k,\bs a} \ga_{\bb,r \bl} + \ga_{\bs,ka}
\ga_{\bb,r\bl} \right]\\
=&\ g^{\bl k} g^{\bs r} \left[ \left( \ga_{a,rk} - \ga_{r,ak} \right) \left(
\ga_{\bs,\bb\bl} - \ga_{\bb,\bs \bl} \right)  - \ga_{a,r\bl} \ga_{\bb,\bs k} +
\ga_{k,\bs
\bb} \ga_{a,\bl r} - \ga_{k,\bs \bb} \ga_{\bl,ar} + \ga_{\bs,ka}
\ga_{\bb,r\bl} \right]\\
=&\ g^{\bl k} g^{\bs r} \left[ \left( \ga_{a,rk} - \ga_{r,ak} \right) \left(
\ga_{\bs,\bb\bl} - \ga_{\bb,\bs \bl} \right)  - \left( \ga_{a,r\bl} -
\ga_{\bl,ar} \right) \left( \ga_{\bb,\bs k} - \ga_{k,\bs \bb} \right)  \right]\\
=&\ - g^{\bl k} g^{\bs r} \left[ \left( \ga_{a,rk} - \ga_{r,ak} \right) \left(
\ga_{\bb,\bs\bl} - \ga_{\bs,\bb \bl} \right)  + \left( \ga_{a,r\bl} -
\ga_{\bl,ar} \right) \left( \ga_{\bb,\bs k} - \ga_{k,\bs \bb} \right)  \right].
\end{align*}
This yields
\begin{align*}
& \left( \dt - \gD_{g_{\ga}} \right) \left[ g_{a \bb} - \del \ga_{ar} g^{\bs r}
\delb
\bga_{\bs \bb} \right]\\
=&\ - \del \ga_{ar} \delb \bga_{\bb \bs} g^{\bs i} g^{\bj r} g^{\bl p} g^{\bq k}
\left[ \left( \ga_{\bq,p\bj} -
\ga_{\bj,p\bq} \right) \left(\ga_{k,\bl i} - \ga_{i,\bl k} \right) + \left(
\ga_{i,\bq \bl} - \ga_{\bq,i\bl} \right) \left(\ga_{\bj,kp} - \ga_{k,\bj p}
\right) \right]\\
&\ - g^{\bl k} g^{\bs r} \left[ \left( \ga_{a,rk} - \ga_{r,ak} \right) \left(
\ga_{\bb,\bs\bl} - \ga_{\bs,\bb \bl} \right)  + \left( \ga_{a,r\bl} -
\ga_{\bl,ar} \right) \left( \ga_{\bb,\bs k} - \ga_{k,\bs \bb} \right)  \right]\\
&\ + \i \delb \bga_{\bs \bb} g^{\bl k} g^{\bs p} g^{\bq r} \left[ 
\ga_{k,\bq p} \ga_{a,\bl r} - \ga_{k,\bq p} \ga_{\bl,ar} +
\ga_{\bq,ka} \ga_{p,\bl r} - \ga_{a,rk} \ga_{p,\bq \bl} \right.\\
&\ \qquad \qquad \left.  + \ga_{a,rk} \ga_{\bq,p\bl}
+ \ga_{r,ak} \ga_{p,\bq\bl} - \ga_{r,ak} \ga_{\bq,p\bl} - \ga_{a,r\bl}
\ga_{p,\bq k} \right]\\
&\ - \i \del \ga_{a r} g^{\bl k} g^{\bs p} g^{\bq r} \left[ \ga_{p,\bq k}
\ga_{\bs,\bb \bl} - \ga_{p,\bq k} \ga_{\bb,\bs \bl} - \ga_{\bq,pk}
\ga_{\bs,\bb\bl} + \ga_{\bq,pk} \ga_{\bb,\bs\bl} \right.\\
&\ \qquad \qquad \qquad \left. + \ga_{\bq,p\bl} \ga_{\bb,\bs k} - \ga_{\bl,p\bq}
\ga_{\bb,\bs k} + \ga_{\bl,p\bq} \ga_{k,\bb \bs} - \ga_{p,\bl \bb} \ga_{\bq,\bs
k} \right]\\
=&\ - \del \ga_{ar} \delb \bga_{\bb \bs} g^{\bs i} g^{\bj r} g^{\bl p} g^{\bq k}
\left[ \left( \ga_{\bq,p\bj} -
\ga_{\bj,p\bq} \right) \left(\ga_{k,\bl i} - \ga_{i,\bl k} \right) + \left(
\ga_{i,\bq \bl} - \ga_{\bq,i\bl} \right) \left(\ga_{\bj,kp} - \ga_{k,\bj p}
\right) \right]\\
&\ - g^{\bl k} g^{\bs r} \left[ \left( \ga_{a,rk} - \ga_{r,ak} \right) \left(
\ga_{\bb,\bs\bl} - \ga_{\bs,\bb \bl} \right)  + \left( \ga_{a,r\bl} -
\ga_{\bl,ar} \right) \left( \ga_{\bb,\bs k} - \ga_{k,\bs \bb} \right)  \right]\\
&\ + \i \delb \bga_{\bs \bb} g^{\bl k} g^{\bs p} g^{\bq r} \left[ \left(
\ga_{k,\bq p} - \ga_{p,\bq k} \right) \left( \ga_{a,\bl r} - \ga_{\bl,ar}
\right) + (\ga_{a,rk} - \ga_{r,ak}) (\ga_{\bq,p\bl} - \ga_{p,\bq \bl}) \right]\\
&\ - \i \del \ga_{a r} g^{\bl k} g^{\bs p} g^{\bq r} \left[ (\ga_{p,\bq k} -
\ga_{\bq,pk})(\ga_{\bs,\bb \bl} - \ga_{\bb,\bs \bl}) + (\ga_{\bb,\bs k} -
\ga_{k,\bs \bb})(\ga_{\bq,p\bl} - \ga_{\bl,p\bq} )
\right]\\
=&\ \sum_{i=1}^8 A_i.
\end{align*}
We combine terms, relabeling indices to
yield
\begin{align*}
 & A_1 + A_4 + A_5 + A_8\\
=&\ - \gb_{ar} \bar{\gb}_{\bb \bs} g^{\bs i} g^{\bj r} g^{\bl p} g^{\bq k}
\left( \ga_{\bq,p\bj} -
\ga_{\bj,p\bq} \right) \left(\ga_{k,\bl i} - \ga_{i,\bl k} \right) - g^{\bl k}
g^{\bs r} \left( \ga_{a,r\bl} -
\ga_{\bl,ar} \right) \left( \ga_{\bb,\bs k} - \ga_{k,\bs b} \right)\\
&\ + \gb_{\bb \bs} g^{\bl k} g^{\bs p} g^{\bq r} \left(
\ga_{k,\bq p} - \ga_{p,\bq k} \right) \left( \ga_{a,\bl r} - \ga_{\bl,ar}
\right) - \gb_{a r} g^{\bl k} g^{\bs p} g^{\bq r} (\ga_{\bb,\bs k} -
\ga_{k,\bs \bb})(\ga_{\bq,p\bl} - \ga_{\bl,p\bq} )\\
=&\ - g^{\bl k} g^{\bs r} \left[ \left( \ga_{a,r\bl} - \ga_{\bl,ar} \right) +
\gb_a^{\bq} \left( \ga_{\bq,r \bl} - \ga_{\bl,r \bq} \right) \right] \left[
\left( \ga_{\bb,\bs k} - \ga_{k,\bs \bb} \right) + \gb_{\bb}^q \left( \ga_{q,\bs
k} - \ga_{k,\bs p} \right) \right].
\end{align*}
Similarly we compute
\begin{align*}
& A_2 + A_3 + A_6 + A_7\\
 =&\ - \gb_{ar} \gb_{\bb \bs} g^{\bs i} g^{\bj r} g^{\bl p} g^{\bq k}
 \left(
\ga_{i,\bq \bl} - \ga_{\bq,i\bl} \right) \left(\ga_{\bj,kp} - \ga_{k,\bj p}
\right) - g^{\bl k} g^{\bs r} \left( \ga_{a,rk} - \ga_{r,ak} \right) \left(
\ga_{\bb,\bs\bl} - \ga_{\bs,\bb \bl} \right)\\
&\ + \gb_{\bb \bs} g^{\bl k} g^{\bs p} g^{\bq r} (\ga_{a,rk} - \ga_{r,ak})
(\ga_{\bq,p\bl} - \ga_{p,\bq \bl} ) - \gb_{a r} g^{\bl k} g^{\bs p} g^{\bq r}
(\ga_{p,\bq k} -
\ga_{\bq,pk})(\ga_{\bs,\bb \bl} - \ga_{\bb,\bs \bl})\\
=&\ - g^{\bl k} g^{\bs r} \left[ \left( \ga_{a,rk} - \ga_{r,ak} \right) +
\gb_a^{\bq} \left( \ga_{\bq,rk} - \ga_{r,\bq k} \right) \right] \left[ \left(
\ga_{\bb,\bs \bl} - \ga_{\bs,\bb\bl} \right) + \gb_{\bb}^{q} \left( \ga_{q,\bs
\bl} - \ga_{\bs,q \bl} \right) \right].
\end{align*}
The result follows.
\end{proof}
\end{lemma}

\begin{prop} \label{Wsubsoln} Let $\ga_t$ be a solution of (\ref{alphaflow})
such that $\i (\delb \ga + \delb \bga) > 0$ for all $t$.  Then
\begin{align*}
\left( \dt - \gD_{g_{\ga}} \right) W(\ga_t) \leq 0.
\end{align*}
\begin{proof} As the calculations above are in an arbitrary coordinate basis,
given a vector $v = (v_1,v_2) \in T^{1,0} \mathbb C^{2n}$ we can extend it to a
coordinate basis and apply Lemmas \ref{calcl3}, \ref{calcl5} and \ref{calcl8} to
yield
\begin{align*}
& \left[\left(\dt - \gD_{g_{\ga}} \right) W\right](v,\bar{v})\\
=&\ - g^{\bar{v}_2 i} g^{\bj v_2} g^{\bl p}
g^{\bq k}
\left[ \left( \ga_{\bq,p\bj} -
\ga_{\bj,p\bq} \right) \left(\ga_{k,\bl i} - \ga_{i,\bl k} \right) + \left(
\ga_{i,\bq \bl} - \ga_{\bq,i\bl} \right) \left(\ga_{\bj,kp} - \ga_{k,\bj p}
\right) \right]\\
&\ - g^{\bl k} g^{\bs r} \left[ \left( \ga_{v_1,r\bl} - \ga_{\bl,v_1 r} \right)
+
\gb_{v_1}^{\bq} \left( \ga_{\bq,r \bl} - \ga_{\bl,r \bq} \right) \right] \left[
\left( \ga_{\bar{v}_1,\bs k} - \ga_{k,\bs \bar{v}_1} \right) + \gb_{\bar{v}_1}^q
\left( \ga_{q,\bs
k} - \ga_{k,\bs p} \right) \right]\\
&\ - g^{\bl k} g^{\bs r} \left[ \left( \ga_{v_1,rk} - \ga_{r,v_1k} \right) +
\gb_{v_1}^{\bq} \left( \ga_{\bq,rk} - \ga_{r,\bq k} \right) \right] \left[
\left(
\ga_{\bar{v}_1,\bs \bl} - \ga_{\bs,\bar{v}_1\bl} \right) + \gb_{\bar{v}_1}^{q}
\left( \ga_{q,\bs
\bl} - \ga_{\bs,q \bl} \right) \right]\\
&\ - g^{\bl k} g^{\bar{v}_2 p} g^{\bq r} \left[ \left( \ga_{v_1,rk} -
\ga_{r,v_1k}
\right) + \gb_{v_1}^{\bs} \left( \ga_{\bs,rk} - \ga_{r,\bs k} \right) \right]
\left(
\ga_{p,\bq \bl} - \ga_{\bq,p\bl} \right) + \mbox{conj.}\\
&\ + g^{\bl k} g^{\bar{v}_2 p} g^{\bq r} \left[ \left( \ga_{v_1,\bl r} -
\ga_{\bl,v_1r}
\right) + \gb_{v_1}^{\bs} \left( \ga_{\bs, r \bl} - \ga_{\bl,r\bs} \right)
\right]
\left( \ga_{k,\bq p} - \ga_{p,\bq k} \right) + \mbox{conj.}\\
=&\ \sum_{i=1}^{8} A_i,
\end{align*}
where we have labeled each term, including the conjugate terms.  Using the
Cauchy-Schwarz inequality we conclude
\begin{align*}
 A_5 + A_6 \leq A_4 + A_2, \qquad A_7 + A_8 \leq A_1 + A_3.
\end{align*}

\end{proof}
\end{prop}

We now rewrite these estimates in a different framework.  In
particular, we note that the matrix $W$ only involves the ``gauge-invariant''
quantities $g$ and $\gb$, subject to the integrability condition $\del \gw =
\delb \gb$, or, in coordinates,
\begin{align} \label{comp}
 \gb_{ij,\bk} = g_{i \bk,j} - g_{j \bk,i}
\end{align}
To that end the entire discussion can be expressed
using these quantities, as we now observe.  In particular, one can reinterpret
the evolution equations above as the system
\begin{gather} \label{gfsys}
 \begin{split}
\left(\dt - \gD_{g} \right) g_{i \bj} =&\ - g^{\bm p} g^{\bq n} \left[
g_{n \bm,i} g_{p \bj,\bq} + g_{n\bm,\bj} g_{i \bq,p} - g_{p\bq,i} g_{n \bm,\bj}
\right],\\
\left(\dt - \gD_{g} \right) \gb_{i j} =&\ - g^{\bq p} g^{\bs r} \left[
g_{i \bs,p} \gb_{r j,\bq} + g_{j \bs,p} \gb_{ir,\bq} \right].
 \end{split}
\end{gather}

\begin{lemma} \label{gaugeinv} Let $(g_t,\gb_t)$ be a solution to (\ref{gfsys})
satisfying (\ref{comp}).  Then
 \begin{align*}
\left( \dt - \gD_{g} \right) & \left[ g_{a \bb} - \gb_{ar} g^{\bs r}
\gb_{\bs \bb} \right]\\
=&\ - g^{\bl k} g^{\bs r} \left[ \left(g_{a\bl,r} +
\gb_a^{\bq} \gb_{\bq \bl,r} \right) \left(
g_{\bb k,\bs} + \gb_{\bb}^q \gb_{q k,\bs} \right) + \left(
\gb_{ar,k} + 
\gb_a^{\bq} g_{\bq r,k} \right) \left( \gb_{\bb\bs,\bl} + \gb_{\bb}^{q} g_{q
\bs,\bl} \right)\right],\\
\left( \dt - \gD_{g} \right) g^{\bs r} =&\ - g^{\bs i} g^{\bj r} g^{\bl p}
g^{\bq k}
\left( \gb_{\bq \bj,p} \gb_{ki,\bl}  + g_{i\bq,\bl} g_{\bj k,p} \right),\\
\left( \dt - \gD_{g} \right) \gb_u^{\bs} =&\ - g^{\bl k} g^{\bs p} g^{\bq
r} \left[ \gb_{ur,k} + \gb_u^{\bv} g_{\bv r,k} \right] g_{p\bq,\bl} + g^{\bl k}
g^{\bs p} g^{\bq r} \left[ g_{u\bl,r} + \gb_u^{\bv} \gb_{\bv \bl,r} \right]
\gb_{k p,\bq}.
\end{align*}
\end{lemma}

\begin{prop} Let $(g_t,\gb_t)$ be a solution to (\ref{gfsys}) satisfying
(\ref{comp}).  Let $W(g,\gb)$ be given by Definition \ref{Wdef}.  Then
\begin{align*}
 \left( \dt - \gD_g \right) W(g,\gb) \leq 0.
\end{align*}
\end{prop}

\subsection{Proof of Theorem \ref{EKthm}}

In this subsection we establish the Evans-Krylov estimate for the pluriclosed
flow.  The structure of the proof exploits ideas similar to those of (\cite{SW}
Theorem 1.1), which themselves are closely modeled after the proof of
Evans-Krylov.  The main step is to establish $C^{\ga}$ regularity of $W$ for a
uniformly
parabolic solution to (\ref{alphaflow}), after which Theorem \ref{EKthm} follows
from Schauder estimates and blowup arguments.  The proof is closely modeled
after (\cite{Lieb} Lemma
14.6), relying crucially on Lemma \ref{calcl1} and Proposition \ref{Wsubsoln}. 
To begin we recall some standard notation and results.

\begin{defn} Given $(w,s) \in \mathbb C^n \times \mathbb R$, let
\begin{align*}
 Q((w,s),R) := &\ \{ (z,t) \in \mathbb C^n \times \mathbb R |\ t \leq s, \quad
\max \{ \brs{z - w}, \brs{t - s}^{\frac{1}{2}} \} < R \}\\
 \Theta(R) :=&\ Q((w,s - 4 R^2),R).
\end{align*}
\end{defn}

\begin{thm} \label{pweakharnack} (\cite{Lieb} Theorem 7.37) Let $u$ be a
nonnegative function on $Q(4R)$ such that 
\begin{align*}
 - u_t + a^{ij} u_{ij} \leq 0,
\end{align*}
where
\begin{align*}
 \gl \gd_i^j \leq a^{ij} \leq \gL \gd_i^j.
\end{align*}
There are positive constants $C, p > 1$ depending only on $n,\gl,\gL$ such that
\begin{align} \label{weakharnack}
 \left( R^{-n-2} \int_{\Theta(R)} u^p \right)^{\frac{1}{p}} \leq C \inf_{Q(R)}
u.
\end{align}
\end{thm}

\begin{prop} \label{oscprop} Suppose $\ga \in \EE^{\gl,\gL}_{Q(R)}$
satisfies (\ref{alphaflow}).  There are positive constants $\gg, C$
depending only on $n,\gl,\gL$ such that for all $\rho < R$,
\begin{align*}
\osc_{Q(\rho)} W \leq C(n,\gl,\gL) \left(
\frac{\rho}{R} \right)^{\gg} \osc_{Q(R)} W.
\end{align*}

 \begin{proof} Note that the $\log \det$ operator is $(\gl,\gL)$ elliptic on a
convex set containing the range of $W$. Using this and Lemma \ref{calcl1} shows
that for any two points $(x,t_1),(y,t_2) \in Q(4R)$ there exists a matrix
$a^{ij}$, $\gl \gd_i^j \leq a^{i{\bj}} \leq \gL \gd_i^j$ such that
\begin{gather} \label{conclogdet}
\begin{split}
0 =&\ \log \det W(x,t_1) - \log \det W(y,t_2)\\
=&\ a^{i{\bj}}((x,t_1),(y,t_2)) \left( W_{i \bj}(x,t_1) - W_{i
\bj}(y,t_2) \right).
\end{split}
\end{gather}
It follows from (\cite{Lieb} Lemma 14.5) that we can choose unit vectors
$v_{\ga}$ and functions $f_{\ga} = f_{\ga}((x,t_1),(y,t_2))$, such
that
\begin{align*}
a^{i\bj} = \sum_{\ga = 1}^N f_{\ga} v_{\ga}^i
\bar{v_{\ga}^j},
\end{align*}
and moreover $\gl_* \leq f_{\ga} \leq \gL_*$, where $\gl_*, \gL_*$ only depend
on $\gl,\gL$.  Now let $w_{\ga} := W_{v_{\ga} \bar{v_{\ga}}}$. Then
(\ref{conclogdet}) reads
\begin{align} \label{wrel}
 \sum f_{\ga} \left(w_{\ga} (y,t_2) - w_{\ga}(x,t_1) \right) = 0.
\end{align}
Let
\begin{align*}
 M_{s \ga} = \sup_{Q(sR)} w_{\ga}, \qquad m_{s \ga} = \inf_{Q(sR)} w_{\ga},
\qquad P(sR) = \sum_{\ga} M_{s \ga} - m_{s \ga}.
\end{align*}
Observe using Proposition \ref{Wsubsoln} that every $M_{2 \ga} - w_{\ga}$ is a
supersolution to a uniformly parabolic equation.  Thus by Theorem
\ref{pweakharnack} we conclude
\begin{align} \label{EK10}
\left( R^{-n-2} \int_{\Theta(R)} (M_{2\ga} - w_{\ga})^p \right)^{\frac{1}{p}}
\leq C \left( M_{2\ga} - M_{\ga} \right).
\end{align}
Now observe that (\ref{wrel}) yields for every pair $(x,t_1), (y,t_2) \in
Q(2R)$,
\begin{align*}
 f_{\ga} \left( w_{\ga}(y,t_2) - w_{\ga}(x,t_1) \right) = \sum_{\gb \neq \ga}
f_{\gb} \left( w_{\gb}(x,t_1) - w_{\gb}(y,t_2) \right).
\end{align*}
It follows directly that
\begin{align*}
 w_{\ga}(y,t_2) - m_{2\ga} \leq C \sum_{\gb \neq \ga} M_{2 \gb} -
w_{\gb}(y,t_2).
\end{align*}
Integrating this over $\Theta(R)$, applying Minkowski's inequality and
(\ref{EK10})
yields
\begin{gather} \label{EK20}
 \begin{split}
 \left( R^{-n-2} \int_{\Theta(R)} \left( w_{\ga} - m_{2 \ga} \right)^p
\right)^{\frac{1}{p}} \leq&\ \left( C R^{-n-2} \int_{\Theta(R)} \left( \sum_{\gb
\neq \ga} M_{2 \gb} - w_{\gb} \right)^p \right)^{\frac{1}{p}}\\
\leq&\ C \sum_{\gb \neq \ga} \left( R^{-n-2} \int_{\Theta(R)} \left( M_{2 \gb} -
w_{\gb} \right)^p \right)^{\frac{1}{p}}\\
\leq&\ C \sum_{\gb \neq \ga} M_{2 \gb} - M_{\gb}.
 \end{split}
\end{gather}
Now we use (\ref{EK10}) and (\ref{EK20}) and Minkowski's inequality to yield
\begin{align*}
 M_{2\gb} - m_{2 \gb} =&\ \left( R^{-n-2} \int_{\Theta(R)} \left(M_{2 \gb} -
m_{2 \gb} \right)^p \right)^{\frac{1}{p}}\\
 =&\ \left( R^{-n-2} \int_{\Theta(R)} \left( M_{2\gb} - w_{\gb} + (w_{\gb} -
m_{2 \gb}) \right)^p \right)^{\frac{1}{p}}\\
\leq&\ C \sum_{\ga} M_{2 \ga} - M_{\ga}\\
\leq&\ C \sum_{\ga} M_{2 \ga} - M_{\ga} + m_{\ga} - m_{2 \ga}\\
=&\ C \left( P(2R) - P(R) \right).
\end{align*}
Summing over $\gb$ and rearranging yields for some constant $0 < \mu < 1$ the
inequality
\begin{align*}
 P(R) \leq \mu P(2R).
\end{align*}
A standard iteration argument now yields the statement of the theorem.
 \end{proof}
\end{prop}

Proposition \ref{oscprop} is the crucial point in establishing Theorem
\ref{EKthm}.  To use it we first obtain a $C^{\ga}$ estimate on the metric, and
then employ blowup arguments.

\begin{cor} \label{alphaestimate} Suppose $\ga \in \EE^{\gl,\gL}_{Q(2)}$
satisfies (\ref{alphaflow}).  There are positive constants $\gg, C$
depending only on $n,\gl,\gL$ such
\begin{align*}
\brs{g_{\ga}}_{C^{\gg}(Q(1))} \leq C.
\end{align*}
\begin{proof} A standard argument using Proposition \ref{oscprop} shows that for
a solution $\ga$ as in the statement, there are constants $\gg,C$ depending on
$n,\gl,\gL$ such that $\brs{W}_{C^{\gg} (Q(1))} \leq C$.  Examining the lower
right block of $W$, this yields a $C^{\gg}$ estimate for $g^{-1}$.  Using this
and looking at the upper right and lower left blocks of $W$ yields a $C^{\gg}$
estimate for $\del \ga$.  Finally, combining these estimates and considering the
upper left block of $W$ yields a $C^{\gg}$ estimate for $g$, as required.
\end{proof}
\end{cor}

Next we bootstrap these estimates to get higher regularity for the metric. 
Since equation (\ref{alphaflow}) is degenerate parabolic, we need to use the
induced equation on the metric, which is strictly parabolic.  To that end, for a
Hermitian manifold $(M^n, J,g)$, let $h$ denote an auxiliary Hermitian metric,
which in local calculations we will take to be a flat metric, and let $\gU(g,h)
= \N_g - \N_h$ be the difference of the two Chern connections.  Furthermore, let
\begin{align*}
f_k = f_k(g,h) := \sum_{j=0}^k \brs{  \N_g^{j} \gU(g,h)}^{\frac{2}{1+j}}.
\end{align*}
We now state a basic smoothing estimate.

\begin{lemma} \label{smoothing} Fix constants $\gl,\gL,K$, and let $g_t$ be a
solution to pluriclosed flow on $B_R(0) \times [0,T]$ such that
\begin{align*}
 \gl g_{\mbox{\tiny E}} \leq g \leq \gL g_{\mbox{\tiny E}}, \qquad 
 \sup_{B_R \times [0,T]} f_1(x,t) \leq K.
\end{align*}
Given $k \in \mathbb N$, there exists $C = C(R,T,\gl,\gL,K)$ such that
\begin{align*}
 \sup_{B_{\frac{R}{2}} \times [\frac{T}{2},T]} f_k(x,t) \leq C.
\end{align*}
\begin{proof} We recall that in complex coordinates the pluriclosed flow
equation can be expressed as
\begin{align*}
 \dt g_{i \bj} =&\ g^{\bl k} g_{i \bj,k\bl} + \del g * \del g.
\end{align*}
Since an estimate on $f_1$ implies a $C^2$ estimate for the metric itself, the
lemma follows in a standard way using cutoff functions and Schauder estimates
(\cite{Lieb} Theorem 4.9).
\end{proof}
\end{lemma}

\begin{prop} \label{ekhigherreg} Suppose $\ga \in \EE^{\gl,\gL}_{Q(2)}$
satisfies (\ref{alphaflow}).  Given $k \in \mathbb N$, there exists $C =
C(n,\gl,\gL,k)$ such that
\begin{align*}
\sup_{Q(1)} f_k (g_{\ga}) \leq C.
\end{align*}
\begin{proof} We use a blowup/contradiction argument.  Fix constants
$\gl,\gL,k,\gg$ as in the statement and suppose the statement were false. 
Choose a sequence of solutions $\{\ga^i_t\}$ as in the statement satisfying the
hypotheses, but for which there exists a sequence $\gl_i \to \infty$ together
with points $(x_i,t_i) \in Q\left(1 \right)$ such that
\begin{align*}
\gl_i = f_k(g_i,x_i,t_i).
\end{align*}
We now claim that for $i$ sufficiently large there exists a new point
$\til{x}_i,\til{t}_i$ such that
\begin{align} \label{ekhigher10}
(\til{x}_i,\til{t}_i) \in Q\left(\frac{3}{2} \right), \qquad
\til{\gl}_i := f_k(g_i,\til{x}_i,\til{t}_i) \geq \gl_i, \qquad
\sup_{Q(\til{\gl}_i^{-1},\til{x}_i,\til{t}_i)} f_k \leq 2 \til{\gl_i}.
\end{align}
To show this we make an inductive choice of points.  Fix some $i$, and let
$(y_0,s_0) = (x_i,t_i)$.  By construction $(y_0,s_0)$ satisfies the first two
conditions of (\ref{ekhigher10}).  Given now some point $(y_j,s_j)$ satisfying
the first two conditions of (\ref{ekhigher10}), if it satisfies the third
condition we set $(\til{x}_i,\til{t_i}) = (y_j,s_j)$.  Otherwise, set $\mu_j =
f_k(g_i,y_j,s_j)$ and choose $(y_{j+1},s_{j+1}) \in Q(\mu_j^{-1},
\til{y}_j,\til{s}_j)$ such that $f_k(g_i,y_{j+1},s_{j+1}) \geq 2 \mu_j$.  Note
that by construction, for any $j$ we have
\begin{align*}
 \brs{s_N} \leq 1 + \sum_{j=1}^N \mu_j^{-1} \leq 1 + \gl_i^{-1}
\sum_{j=1}^{\infty} 2^{-j} < \frac{9}{4}.
\end{align*}
for $\gl_i$ sufficiently large.  A similar estimate shows that $\brs{y_N} <
\frac{3}{2}$ for any $N$.  Thus our inductive choice is well-defined, and since
the solution is smooth this process terminates at some finite $N$, finishing the
proof of the claim.

We rescale around these points to finish the argument.  In particular, let us
simplify notation and consider triples $(g_i,x_i,t_i)$ of solutions defined on
$Q(2)$ with the points $(x_i,t_i) \in Q(\frac{3}{2})$, and $\gl_i =
f_k(g_i,x_i,t_i) \to \infty$, and finally
\begin{align*}
 \sup_{Q(\gl_i^{-1},x_i,t_i)} f_k \leq 2 \gl_i.
\end{align*}
Now let
\begin{align*}
 \til{g}_i(x,t) = g_i\left( x_i + \gl_i^{-\frac{1}{2}} x, t_i + \gl_i^{-1}
t\right).
\end{align*}
By construction each solution $\til{g}_i$ is defined on $Q(1)$ with
\begin{align*}
 \sup_{Q(1)} f_k(\til{g}_i) = 1.
\end{align*}
By Lemma \ref{smoothing}, we conclude that there exists a subsequence of
$\til{g}_i$ converging in $C^{\infty}$ to a limiting solution $g_{\infty}$ such
that $f_k(g_{\infty},0,0) = 1$.  But on the other hand by Corollary
\ref{alphaestimate} we have an a priori $C^{\ga}$ estimate for the metrics $g_i$
at the points $(x_i,t_i)$.  After the blowup this implies that the metric
$g_{\infty}$ is constant in space and time, and so in particular
$f_k(g_{\infty},0,0) = 0$.  This is a contradiction, finishing the proof.
\end{proof}
\end{prop}

\begin{cor} \label{complexrigidity} Let $\ga_t$ be a solution to
(\ref{alphaflow}) on $(-\infty,0] \times \mathbb C^n$ such that $\ga_t \in
\EE_{\mathbb C^n,\gl,\gL}$ for all $t \in (-\infty,0]$.  Then $\frac{\del}{\del
x} g_{\mu} \equiv \dt g_{\mu} \equiv 0$.
\begin{proof} Suppose there exists a point such that $\brs{\frac{\del}{\del x}
g_{\ga}} \neq
0$.  By translating in space and time we can assume without loss of generality
this point is $(0,0)$.  Fix some $A > 0$ and consider
\begin{align*}
\mu(x,t) := A^{-1} \ga(A x, A^2 t).
\end{align*}
By direct calculation one verifies that $\mu$ is a solution to
(\ref{alphaflow}) on $(-\infty,0] \times \mathbb C^n$ and moreover $\mu \in
 \EE_{\mathbb C^n}^{\gl,\gL}$ for all $t \in (-\infty,0]$.  Also observe
that $\brs{\frac{\del}{\del x} g_{\mu}}(0,0) = A \brs{\frac{\del}{\del x}
g_{\ga}}(0,0)$. 
For $A$ chosen
sufficiently large this contradicts the result of Proposition \ref{ekhigherreg},
finishing the proof.
\end{proof}
\end{cor}

\begin{proof}[Proof of Theorem \ref{EKthm}] If the statement of the theorem is
false, then there exists a sequence $\{(g_t^i,\ga_t^i,\hat{g}_t^i,h^i,\mu^i)\}$
of solutions such that $g_0,\hat{g},h$ satisfy uniform geometric bounds (i.e.
lower bounds on injectivity radii and uniform bounds on curvature and all
covariant derivatives of curvature), but such that there exist points $(x_i,t_i)
\in M_i \times [0,\tau)$, $\tau \leq 1$, such that
\begin{align*}
 t_i f_k(g_{t_i},h)(x_i) = \sup_{M_i \times [0,\tau)}
t f_k(g_t,h) \to \infty.
\end{align*}
We now perform a blowup, aiming to get a contradiction to
Proposition \ref{ekhigherreg}.
Fix a constant $A > 0$, let $\gs_i := f_k(g_{t_i},h)(x_i),$
and set
\begin{align*}
\gl_i := \left(A \gs_i \right)^{-\frac{1}{2}}.
\end{align*}
Due to the uniform estimates on
the background data, we can choose a small radius $R > 0$ so that around each
point $(x_i,t_i)$ we have a normal coordinate chart for $\hat{g}_i$ of radius
$R$.  Using these coordinates we define
\begin{align*}
\ga'_i(x,t) =&\ \gl_i^{-1} \ga_i(x_i + \gl_i x, t_i + \gl_i^2 t),\\
\mu'_i(x,t) =&\ \gl_i^{-1} \mu_i(x_i + \gl_i x, t_i + \gl_i^2 t),\\
\hat{g}'_i(x,t) =&\ \hat{g}_i(x_i + \gl_i x, t_i + \gl_i^2 t),\\
h'_i(x,t) =&\ h_i(x_i + \gl_i x, t_i + \gl_i^2 t).
\end{align*}
By simple computations one observes that the resulting blowup data still define
a solution to reduced pluriclosed flow.  We also observe by construction that
for
sufficiently large $i$ each of the resulting blowup solutions exists on $[-1,0]
\times B_1(0)$.  Also by construction, we obtain that for sufficiently large
$i$,
\begin{align} \label{blowuplarge}
 f_k(g'_i,h'_i)(0,0) = A^{-1}, \qquad \sup_{[-1,0] \times B_1(0)} f_k(g',h) \leq
2
A^{-1}.
\end{align}
We observe by construction that $\{\hat{g}'_i \}$ converges to the standard
Euclidean metric on $B_1(0)$, constant in time.  Moreover, the sequence $\{h'_i
\}$ will converge to a different, time-independent flat metric on $B_1(0)$.

We note that by the metric estimates, one has that $\{\del \ga'_i\}$ has uniform
$C^{k}$ estimates along the sequence.  We need to account for the gauge
invariance to obtain a full estimate for $\ga'$ however.  We can choose a
function $f_i \in C^{\infty}(M)$ such that $\del^*_{\hat{g}_i'} \left[
\ga_i'(-1) + \del f \right] = 0$.  Moreover we can choose a constant $(1,0)$
form $\eta_i$ such that $\left[\ga'_i(-1) + \del f_i + \eta_i \right](0,-1) =
0$.  Let $\til{\ga}'_i(t) = \ga'_i(t) + \del f_i + \eta_i$.  Using the evolution
equation for $\ga'_i$ and the estimates on $f_k$, one obtains a $C^{k-2}$
estimate on $\del^*_{\hat{g}_i'} \til{\ga}_i'(t)$ and a $C^0$ estimate on
$\til{\ga}_i'$ on $B_1(0) \times [-1,0]$.  Using estimates coming from Hodge
theory we obtain that $\{\til{\ga}'_i\}$ converges subsequentially
to a solution to (\ref{gfsys}) on $[-1,0] \times B_1(0)$
such that
\begin{align*}
 f_k(g_{\infty},h_{\infty})(0,0) = A^{-1}.
\end{align*}
For $A$ chosen sufficiently small this contradicts Proposition
\ref{ekhigherreg}.
\end{proof}

\section{Global existence and convergence on nonpositively curved backgrounds}
\label{ltesec}

In this section we first prove Theorem \ref{uplowbndthm}, and then use it and
Theorem \ref{EKthm} to establish Theorem \ref{PCFLTE}.

\subsection{Proof of Theorem \ref{uplowbndthm}} \label{estsec}

In this subsection we employ the evolution equations of \S \ref{tpevsec} to
establish Theorem \ref{uplowbndthm}.  To begin we recall two evolution equations
for pluriclosed flow from \cite{SW}.

\begin{lemma} \label{volumeformev} Let $(M^{2n}, g_t, J)$ be a solution to
pluriclosed flow, and let $h$ denote another Hermitian metric on $(M, J)$.  Then
\begin{align*}
 \left( \frac{\del}{\del t} - \gD \right) \log \frac{\det g}{\det h} =&\
\brs{T}^2 - \tr_g \rho(h).
\end{align*}
\begin{proof} We directly compute using (\ref{PCFCC}),
\begin{align*}
 \dt \log \frac{\det g}{\det h} =&\ g^{\bj i} \left( \dt g \right)_{i \bj}\\
 =&\ g^{\bj i} \left[ g^{\bq p} g_{i \bj,p\bq} - g^{\bq p}
g^{\bs r} g_{i \bs,p} g_{r \bj,\bq} + Q_{i \bj} \right].
\end{align*}
Also
\begin{align*}
 \gD \log \frac{\det g}{\det h} =&\ g^{\bq p} \left[ \log \frac{\det g}{\det h}
\right]_{,p\bq}\\
 =&\ g^{\bq p} \left[ g^{\bj i} g_{i \bj,p} - h^{\bj i} h_{i \bj,p}
\right]_{,\bq}\\
 =&\ g^{\bq p} \left[ g^{\bj i} g_{i\bj,p\bq} - g^{\bj k} g_{k \bl,\bq} g^{\bl
i} g_{i \bj,p} - h^{\bj i} h_{i \bj,p\bq} + h^{\bj k} h_{k \bl,\bq} h^{\bl i}
h_{i \bj,p} \right].
\end{align*}
Combining the above calculations yields
\begin{align*}
 \left(\dt - \gD \right) \log \frac{\det g}{\det h} =&\ \brs{T}^2 - \tr_g
\rho(h),
\end{align*}
as required.
\end{proof}
\end{lemma}

\begin{lemma} \label{traceev} Let $(M^{2n}, g_t, J)$ be a solution to
pluriclosed flow, and let $h$ denote another Hermitian metric on $(M, J)$.  Then
\begin{align*}
 \left( \frac{\del}{\del t} - \gD \right) \tr_{g} h =&\ -
\brs{\gU(g,h)}^2_{g^{-1},g^{-1},h} - \IP{h,Q} + \Omega_h(g,g).
\end{align*}
\begin{proof} We directly compute
\begin{align*}
 \dt \tr_g h =&\ \dt g^{\bj i} h_{i \bj}\\
 =&\ - g^{\bj k} \left( \dt g_{k \bl} \right) g^{\bl i} h_{i \bj}\\
 =&\ - g^{\bj k} g^{\bl i} h_{i \bj} \left[ g^{\bq p} g_{k \bl,p\bq} - g^{\bq p}
g^{\bs r} g_{k \bs,p} g_{r \bl,\bq} + Q_{k \bl} \right].
\end{align*}
On the other hand
\begin{align*}
 \gD \tr_g h =&\ g^{\bq p} \left[ g^{\bj i} h_{i \bj} \right]_{,p\bq}\\
 =&\ g^{\bq p} \left[ - g^{\bj k} g_{k \bl,p} g^{\bl i} h_{i \bj} + g^{\bj i}
h_{i\bj,p} \right]_{,\bq}\\
 =&\ g^{\bq p} \left[ g^{\bj r} g_{r \bs,\bq} g^{\bs k} g_{k \bl,p} g^{\bl i}
h_{i \bj} - g^{\bj k} g_{k \bl,p\bq} g^{\bl i} h_{i \bj} + g^{\bj k} g_{k \bl,p}
g^{\bl r} g_{r\bs,\bq} g^{\bs i} h_{i \bj} \right.\\
&\ \qquad \left. - g^{\bj k} g_{k \bl,p} g^{\bl i} h_{i \bj,\bq} - g^{\bj k}
g_{k \bl,\bq} g^{\bl i} h_{i\bj,p} + g^{\bj i} h_{i \bj,p\bq} \right].
 \end{align*}
Combining the above calculations yields
\begin{align*}
 \left( \dt - \gD \right) \tr_g h =&\ - g^{\bq p} g^{\bj r} g^{\bs k} g^{\bl
i}h_{i \bj} g_{r \bs,\bq} g_{k \bl,p}\\
&\ + g^{\bq p} \left[g^{\bj k} g_{k \bl,p} g^{\bl i} h_{i \bj,\bq} + g^{\bj k}
g_{k \bl,\bq} g^{\bl i} h_{i\bj,p} - g^{\bj i} h_{i \bj,p\bq} \right]
- \IP{h,Q}_g\\
=&\ - \brs{\gU(g,h)}^2_{g^{-1},g^{-1},h} - \IP{h,Q} + \Omega_h(g,g).
\end{align*}
\end{proof}
\end{lemma}

\begin{proof}[Proof of Theorem \ref{uplowbndthm}] Assuming the setup of the
theorem, let
\begin{align*}
 \Phi(x,t) := 1 + \tr_g h + \log \frac{\det g}{\det h} + 2 \brs{\del \ga}^2.
\end{align*}
Combining Proposition \ref{torsionpotentialev} and Lemmas \ref{traceev},
\ref{volumeformev}
yields
 \begin{align*}
\left( \dt - \gD \right) \Phi =&\ -
\brs{\gU(g,h)}^2_{g^{-1},g^{-1},h} - \IP{h,Q} + \Omega_h(g,g) + \brs{T}^2 -
\tr_g \rho(h)\\
&\ - 2 \brs{\N \del
{\ga}}^2 - 2 \brs{\bar{\N} \del \ga}^2 - 4\IP{Q, \tr\del \ga \otimes
\delb \bga} + 4 \Re \i \IP{\tr_g \N T_{\hat{g}} + \del \mu, \delb \bga}.
 \end{align*}
We proceed to estimate the various terms above.  First using the lower bound of
the metric,
\begin{align*}
\Omega_h(g,g) = g^{\bl k} g^{\bj i} (\Omega^h)_{i \bj k \bl} \leq C (\tr_g h)^2
\leq&\ C(\gl,h,g_0).
\end{align*}
Similarly
\begin{align*}
- \tr_g \rho(h) \leq C \tr_g h \leq C(\gl,h,g_0).
\end{align*}
Also we express
\begin{align*}
4 \IP{ \tr_{g_{\ga}} \N T_{\hat{g}}, \delb \bga} =&\ 4 g^{\bq p} \N_{p}
(T_{\hat{g}})_{i j \bq}
\delb \bga_{\bk \bl} g^{\bk i} g^{\bl j}\\
=&\ 4 g^{\bq p} g^{\bk i} g^{\bl j} \delb \bga_{\bk \bl} \left[ (\N^h)_p
(T_{\hat{g}})_{i j \bq} +
\gU(g,h)_{p i}^r (T_{\hat{g}})_{r j \bq} + \gU(g,h)_{p j}^r (T_{\hat{g}})_{i r
\bq} \right].
\end{align*}
Then we estimate in a basis where $h = \Id$ and $g$ is diagonalized,
\begin{align*}
4 g^{\bq p} g^{\bk i} g^{\bl j} \delb \bga_{\bk \bl} (\N^h)_p (T_{\hat{g}})_{i j
\bq} \leq&\ 4 \left[
\gl^{-1} \sqrt{g^{\bk k}} \sqrt{g^{\bl l}} \delb \bga_{\bk \bl} \right]
\left[ \sqrt{g^{\bk k}} \sqrt{g^{\bl l}} (\N^h)_p (T_{\hat{g}})_{i j \bp}
\right]\\
\leq&\ 4 \gl^{-2}  \left[ \brs{\del \ga}^2 + \brs{\N^h T_{\hat{g}}}_h \right]\\
\leq&\ C(\gl,g_0,h) \Phi.
\end{align*}
Also we have
\begin{align*}
4 g^{\bq p} g^{\bk i} g^{\bl j} \delb \bga_{\bk \bl} \gU_{p i}^r
(T_{\hat{g}})_{r
j \bq} \leq&\ 4
\left( \sqrt{g^{\bp p}} \sqrt{g^{\bk k}} \gU_{p k}^r \sqrt{h_{r\br}} \right)
\left( \sqrt{g^{\bp p}} \sqrt{ g^{\bl l}} \sqrt{ h^{\br r}} (T_{\hat{g}})_{r l
\bp} \delb
\bga_{\bk \bl} \sqrt{g^{\bk k}} \sqrt{g^{\bl l}} \right)\\
\leq&\ \frac{1}{2} \brs{\gU}^2_{g^{-1},g^{-1},h} + C(\gl,g_0,\hat{g})
\brs{\del \ga}^2\\
\leq&\ \frac{1}{2} \brs{\gU}^2_{g^{-1},g^{-1},h} + C(\gl,g_0,\hat{g}) \Phi.
\end{align*}
A similar estimate yields
\begin{align*}
\IP{\del \mu, \delb \bga} \leq&\ C(\gl,g_0,\mu) \Phi.
\end{align*}
Next we observe that, using the Cauchy-Schwarz inequality and the lower bound on
$g$,
\begin{align*}
\brs{T}^2 =&\ \brs{\del \gw}^2\\
=&\ \brs{\delb \del \ga + \del \hat{\gw}}^2\\
\leq&\ 2 \brs{\bar{\N} \del \ga}^2 + 2 \brs{\del \hat{\gw}}^2\\
\leq&\ 2 \brs{\bar{\N} \del \ga} + C.
\end{align*}
Also we observe that since $Q \geq 0$ we have $ - \IP{h,Q} \leq 0$, and
$-\IP{Q,\tr \del \ga \otimes \delb \bga} \leq 0$.  Collecting the above
estimates
yields
\begin{align*}
\left(\dt - \gD \right) \Phi \leq C(\gl,\hat{g},h,\mu,g_0) \Phi.
\end{align*}
By the maximum principle we conclude the result.
\end{proof}

\subsection{Proof of Theorem \ref{PCFLTE}}

In this subsection we combine Theorem \ref{EKthm}, Theorem \ref{uplowbndthm},
and further a priori estimates related to the Schwarz Lemma \cite{YauSL} and the
Calabi-Yau Theorem \cite{Yau} to yield the long time
existence and convergence results of Theorem \ref{PCFLTE}.  To begin we record a
rigidity result which will be used in the end of the proof
of Theorem \ref{PCFLTE}.

\begin{lemma} \label{rigiditylemma} Let $(M^{2n}, h, J)$ be a compact Hermitian
manifold
with $\rho(h) \leq 0$.  Suppose $g$ is a pluriclosed metric which is a steady
soliton.  Then $g$ is a
Calabi-Yau metric.
\begin{proof} As the metric is a soliton the pluriclosed evolution consists of
pullback by a diffeomorphism generated by a vector field $X$.  Thus by Lemma
\ref{volumeformev} we have
\begin{align*}
X \cdot \log \frac{\det g}{\det h} = \dt \log \frac{\det g}{\det h} = \gD \log
\frac{\det g}{\det h} + \brs{T}^2 - \tr_g \rho(h) \geq \gD \log \frac{\det
g}{\det h} + \brs{T}^2.
\end{align*}
It follows from the strong maximum principle that $\brs{T}^2 = 0$, and therefore
$g$ is K\"ahler, and hence a compact steady Ricci soliton, which is an Einstein
metric by \cite{HamNonsing}.  Thus $g$ is K\"ahler-Einstein.
\end{proof}
\end{lemma}

\begin{lemma} \label{nonplowbnd} Let $(M^{2n}, J, g_t)$ be a solution to
pluriclosed flow.  Suppose
there exists a Hermitian metric $h$ on $M$ with nonpositive bisectional
curvature.  Then
\begin{align*}
 \sup_{M \times [0,T)} \tr_{g_t} h \leq \sup_{M} \tr_{g_0} h.
\end{align*}
\begin{proof} Inspecting the result of Lemma \ref{traceev}, we note that the
matrix $Q$ is positive semidefinite, hence $\IP{h,Q} \geq 0$.  Also, by
hypothesis $h$ has nonpositive bisectional curvature, and hence $\Omega_h(g,g)
\leq 0$.  The result thus follows from the maximum principle.
\end{proof}
\end{lemma}

\begin{lemma} Let $(M^{2n}, J)$ be a compact complex manifold admitting a
Hermitian metric $h$ of strictly negative bisectional curvature.  Let $\gd$
denote the
infimum of the absolute value of the bisectional curvatures of $h$.  Let $g_0$
be a pluriclosed metric on $M$ and let $\gL
= \left[\sup_M \tr_{g_0} h\right]^{-1}$.  The solution to pluriclosed flow with
initial condition $g_0$ satisfies
 \begin{align} \label{hyplowbnd}
  \sup_{M} \tr_{g_{\tau}} h \leq \frac{1}{\gL + \gd \tau}.
\end{align}
\begin{proof} From Lemma \ref{traceev} we yield, 
\begin{align*}
 \left( \dt - \gD \right) \tr_g h =&\ - \brs{\gU(g,h)}^2_{g^{-1},g^{-1},h} -
\IP{h,Q} + \Omega_h(g,g)\\
\leq&\ - \gd (\tr_g h)^2.
\end{align*}
Applying the maximum principle we conclude that $\sup_{M} \tr_{g_t} h$ is
bounded above by the solution to the ODE
\begin{align*}
\frac{d F}{dt} =&\ - \gd F^2, \qquad F(0) = \sup_{M} \tr_{g_0} h.
\end{align*}
Solving this ODE yields the estimate (\ref{hyplowbnd}).
\end{proof}
\end{lemma}

For the next two lemmas we require the $1$-form reduction, and we specify the
relevant background data in the case of a background metric of nonpositive
bisectional curvature here.  In particular, let $h$ denote the given background
metric of nonpositive bisectional curvature.  Then in particular $\rho(h) \leq
0$.  We furthermore set $\hat{\gw}_t = \gw_0 - t \rho(h) > 0$.  In particular
this means we set $\mu = 0$.  With these choices made and an application of
Lemma \ref{reduction} we will assume a solution to (\ref{reducedflow}) with
respect to this background data.

\begin{lemma} \label{torspotev} Let $(M^{2n}, J)$ be a compact complex manifold
admitting a Hermitian metric $h$ of strictly negative bisectional curvature. 
Let $g_t$ denote a solution to the pluriclosed flow on $(M, J)$, and let $\ga_t$
denote the corresponding solution to (\ref{reducedflow}).  There exists a
constant $C = C(g_0,h)$ such that for any smooth existence time $\tau > 0$, one
has
\begin{align*}
\sup_{M} \brs{\del \ga_{\tau}}^2 \leq C.
\end{align*}
\begin{proof} Let
\begin{align*}
 \Phi := 1 + \tr_{g_t} h + \brs{\del \ga}^2.
\end{align*}
Combining Proposition \ref{torsionpotentialev} and Lemma \ref{traceev} yields
\begin{align*}
 \left( \dt - \gD \right) \Phi =&\ -
\brs{\gU(g,h)}^2_{g^{-1},g^{-1},h} - \IP{h,Q} + \Omega_h(g,g)\\
&\ - \brs{\N \del
{\ga}}^2 - \brs{\bar{\N} \del {\ga}}^2 - 2 \IP{Q, \tr\del \ga \otimes
\delb \bar{\ga}} + 2 \Re \i \IP{\tr_g \N^g \hat{T}, \delb \bga}.
\end{align*}
All terms have a favorable sign except the last one.  To estimate this we first
of all express 
\begin{align*}
\IP{ \tr_{g} \N^g \hat{T}, \delb \bga} =&\ g^{\bq p} \N_{p}
\hat{T}_{i j \bq}
\delb \bga_{\bk \bl} g^{\bk i} g^{\bl j}\\
=&\ g^{\bq p} g^{\bk i} g^{\bl j} \delb \bga_{\bk \bl} \left[ \N^h_p
\hat{T}_{i j \bq} +
\gU(g,h)_{p i}^r \hat{T}_{r j \bq} + \gU(g,h)_{p j}^r \hat{T}_{i r
\bq} \right].
\end{align*}
Then we estimate in a basis where $h = \Id$ and $g$ is diagonalized,
\begin{align*}
 g^{\bq p} g^{\bk i} g^{\bl j} \delb \bga_{\bk \bl} \N^h_p \hat{T}_{i j
\bq} \leq&\ 
\frac{1}{(\gL + \gd t)^2} \left[ \sqrt{g^{\bk k}} \sqrt{g^{\bl l}} \delb
\bga_{\bk \bl}
\right]
\left[ \N^h_p \hat{T}_{kl \bp}
\right]\\
\leq&\ \frac{1}{(\gL + \gd t)^2} \left[ \brs{\del \ga}^2 + \brs{\N^h
\hat{T}}_h^2 \right]\\
\leq&\ \frac{C \Phi}{(\gL + \gd t)^2}.
\end{align*}
Also we have
\begin{align*}
 g^{\bq p} g^{\bk i} g^{\bl j} \delb \bga_{\bk \bl} \gU_{p i}^r \hat{T}_{r j
\bq}
\leq&\
\left( \sqrt{g^{\bp p}} \sqrt{g^{\bk k}} \gU_{p k}^r \sqrt{h_{r\br}} \right)
\left( \sqrt{g^{\bp p}} \sqrt{ g^{\bl l}} \sqrt{ h^{\br r}} \hat{T}_{r l \bp}
\delb
\bga_{\bk \bl} \sqrt{g^{\bk k}} \sqrt{g^{\bl l}} \right)\\
\leq&\ \frac{1}{2} \brs{\gU}^2_{g^{-1},g^{-1},h} + \frac{C}{(\gL + \gd t)^2}
\brs{\del \ga}^2\\
\leq&\ \frac{1}{2} \brs{\gU}^2_{g^{-1},g^{-1},h} + \frac{C \Phi}{(\gL + \gd
t)^2}.
\end{align*}
Combining these inequalities yields the differential inequality
\begin{align*}
 \left( \dt - \gD \right) \Phi \leq&\ \frac{C \Phi}{\left(\gL + \gd t
\right)^2}.
\end{align*}
By the maximum principle, we obtain that
\begin{align*}
 \sup_{M \times \{t\}} \Phi \leq F_t,
\end{align*}
where
\begin{align*}
 \frac{d F}{dt} = \frac{C F}{(\gL + \gd t)^2}, \qquad F_0 = \sup_{M \times
\{0\}}
\Phi.
\end{align*}
This can be solved explicitly, with
\begin{align*}
F_t = F_0 \exp \left[ \frac{C
\gd^{-1}}{\gL} - \frac{C \gd^{-1}}{\gL + \gd t} \right].
\end{align*}
The proposition follows.
\end{proof}
\end{lemma}

\begin{lemma} \label{hypulbnds} Let $(M^{2n}, J)$ be a compact complex manifold
admitting a Hermitian metric $h$ of constant negative bisectional curvature
$-\gd$. 
Given $g_0$ a pluriclosed metric there exists a
constant $C = C(g_0,h)$ such that the solution to pluriclosed flow with initial
condition $g_0$ satisfies
 \begin{align*}
C^{-1} (1 + \gd t) h \leq g_t \leq C (1 + \gd t) h.
\end{align*}
\begin{proof} Let
\begin{align*}
\Phi = \log \frac{\det g}{\det h} + 2 \left( 1 + \tr_{g_t} h + \brs{\del
\ga}^2
\right).
\end{align*}
Inspecting the proof of Lemma \ref{torspotev} yields the estimate
\begin{align*}
\left(\dt - \gD \right) \left(1 + \tr_{g_t} h + \brs{\del \ga}^2 \right) \leq&\
- \frac{1}{2} \brs{\gU(g,h)}^2_{g^{-1},g^{-1},h} - \brs{\bar{\N} \del \ga}^2 +
\frac{C}{(\gL + \gd t)^2}.
\end{align*}
Combining this estimate with Lemma \ref{volumeformev} yields
we yield
\begin{align*}
 \left( \dt - \gD \right) \Phi \leq&\ \brs{T}_g^2 - \tr_g \rho(h)  -
\brs{\gU(g,h)}^2_{g^{-1},g^{-1},h} - 2 \brs{\bar{\N} \del \ga}^2 +
\frac{C}{(\gL + \gd t)^2}.
\end{align*}
First, since $h$ has constant negative bisectional curvature $\gd$, we have
$\rho(h) = - n\gd h$ and estimate
\begin{align*}
- \tr_g \rho(h) = n \gd \tr_g h \leq \frac{n \gd}{\gL + \gd t}.
\end{align*}
Moreover, we estimate
\begin{align*}
\brs{T}^2 =&\ \brs{\del \hat{\gw} + \bar{\N} \del \ga}^2\\
\leq&\ 2 \brs{\del \hat{\gw}}^2 + 2 \brs{\bar{\N} \del \ga}^2\\
\leq&\ C \left(\tr_g h \right)^3 + 2 \brs{\bar{\N} \del \ga}^2\\
\leq&\ \frac{C}{(\gL + \gd t)^3} + 2 \brs{\bar{\N} \del \ga}^2.
\end{align*}
Combining these yields the estimate
\begin{align*}
\left( \dt - \gD \right) \Phi \leq \frac{n \gd}{\gL + \gd t} + \frac{C}{(\gL +
\gd t)^2}.
\end{align*}
Applying the maximum principle yields
\begin{align*}
\sup_{M \times \{t\}} \Phi \leq \sup_{M \times \{0\}} \Phi + C + n \ln \left(\gL
+
\gd t \right).
\end{align*}
Exponentiating yields
\begin{align*}
\frac{\det g}{\det h} \leq C \left( \gL + \gd t \right)^n.
\end{align*}
By the arithmetic geometric mean and (\ref{hyplowbnd}) we conclude
\begin{align*}
\tr_{h} g_t \leq \left( \tr_{g_t} h \right)^{n-1} \frac{\det g}{\det h} \leq&\ C
\left( \gL + \gd t \right).
\end{align*}
The lemma follows.
\end{proof}
\end{lemma}

\begin{proof} [Proof of Theorem \ref{PCFLTE}]  First we establish statement (1).
 Using the metric $h$ of nonpositive bisectional curvature we obtain a uniform
lower bound for the metric from Lemma \ref{nonplowbnd}.  The result now follows
by applying Theorems \ref{EKthm} and \ref{uplowbndthm}.

Now we address statement (2).  Let $h$ denote a flat K\"ahler metric on
$(M^{2n}, J)$.  Setting $\hat{g}_t =
g_0$ and $\mu = 0$ we obtain a solution to $(\hat{g},h,\mu)$-reduced pluriclosed
flow via Lemma \ref{reduction}.  We now claim that $[H_0] = 0$.  By symmetrizing
the K\"ahler form over the complex Lie group action on the torus we obtain a new
left-invariant K\"ahler form.  However a direct calculation shows that
left-invariant Hermitian metrics on tori are automatically K\"ahler, and so have
vanishing torsion class.  Since symmetrization preserves cohomology classes this
means that $[H_0] = 0$.  By the $\del\delb$-lemma it follows that $[\del \gw_0]
= 0 \in
H^{2,1}(M)$, thus we can choose $\eta$
and $\phi$ according to Proposition \ref{specialtorspot}.  Using Lemma
\ref{nonplowbnd} we have an a priori lower bound for the metric along the flow. 
On the other hand using the result of Proposition \ref{specialtorspot} and Lemma
\ref{volumeformev} we obtain
\begin{align*}
\left(\dt - \gD \right) \left[\log \frac{\det g}{\det h} + \brs{\phi}^2 \right]
\leq&\ - \brs{\N \phi}^2 - 2 \IP{Q, \phi \otimes \bar{\phi}}.
\end{align*}
By applying the maximum principle we obtain an a priori upper bound for the
metric along the flow.  Theorem \ref{EKthm} then implies uniform $C^{\infty}$
estimates for all times.

We now address the convergence at infinity.  Using the uniform $C^{\infty}$
bounds, any sequence of times $t_j \to \infty$ admits a smooth subsequential
limiting metric on the same complex manifold.  Moreover, these same uniform
$C^{\infty}$ estimates imply that the Perelman-type $\FF$ functional for the
pluriclosed flow (\cite{ST3} Theorem 1.1) has a uniform upper bound for all
times.  It follows that any subsequential limit as described above is a
pluriclosed steady soliton.  From Lemma \ref{rigiditylemma} it follows that this
limiting metric $g_{\infty}$ is a K\"ahler-Einstein metric.  Each K\"ahler class
on the torus admits a unique flat K\"ahler-Einstein metric, hence $g_{\infty}$
is flat.  It now follows
from the linear/dynamic stability result of (\cite{ST1} Theorem 1.2) that the
whole flow converges exponentially to $g_{\infty}$, as required.

Now we establish statement (3).  By Lemma \ref{hypulbnds} we obtain that for any
smooth interval of existence $[0,\tau)$ one has upper and lower bounds on the
metric, and a uniform estimate of $\brs{\del \bga}^2$.  It follows from Theorem
\ref{EKthm} that the flow exists smoothly on $[0,\infty)$.  We now translate
these estimates to the normalized flow.  In particular, if $g_t$ denotes the
solution to pluriclosed flow and we let 
\begin{align*}
\check{g}_s := e^{-s} g_{e^s},
\end{align*}
then $\check{g}_s$ is the unique solution to (\ref{npcf}) with initial condition
$g_0$.  Moreover, setting $\check{\ga}_s = e^{-s} \ga_{e^s}$, it follows that
$\check{\ga}_t$ is a solution to the normalized $1$-form flow.  Observe from
Lemma \ref{hypulbnds} that $\check{g}$ has uniform upper and lower bounds. 
Also, from Lemma \ref{torspotev} there is a uniform upper bound for $\brs{\del
\check{\ga}}^2$ due to its natural scaling invariance.  From Theorem \ref{EKthm}
we obtain uniform $C^{k,\ga}$ estimates on $g_t$ for all $k,\ga$ and all $t >
0$.  We now employ a blowdown argument to show convergence at infinity.  Choose
a sequence $t_j \to \infty$ and consider the sequence of solutions $\{g_j(t) =
t_j^{-1} g(t_j t)\}$.  By the estimates we have already shown there is a
subsequence converging to a limit flow $g_{\infty}(t)$ defined on $(0,\infty)$. 
Using the monotone expanding entropy functional for
pluriclosed flow (\cite{ST3} Corollary 6.8), it follows that $g_{\infty}$ is
K\"ahler-Einstein, finishing the proof.
\end{proof}

\section{Global existence and convergence results on generalized K\"ahler
manifolds} \label{GEGK}
\subsection{Setup} \label{GKsetup}

In this section we exploit the estimates above to establish new long time
existence results for the pluriclosed flow in the setting of commuting
generalized K\"ahler geometry.  We briefly recall here the
discussion in \cite{SPCFSTB} wherein the pluriclosed flow in the setting of
generalized K\"ahler geometry with commuting complex structures is reduced to a
fully nonlinear parabolic PDE.

\subsubsection{Differential geometric aspects}

Let $(M^{2n}, g, J_A, J_B)$ be a generalized K\"ahler manifold satisfying
$[J_A,J_B] = 0$.  Define
\begin{align*}
 \Pi := J_A J_B \in \End(TM).
\end{align*}
It follows that $\Pi^2 = \Id$, and $\Pi$ is $g$-orthogonal, hence $\Pi$ defines
a $g$-orthogonal decomposition into its $\pm 1$ eigenspaces, which we denote
\begin{align*}
 TM = T_+ M \oplus T_- M.
\end{align*}
Moreover, on the complex manifold $(M^{2n}, J_A)$ we can similarly decompose the
complexified tangent bundle $T_{\mathbb C}^{1,0}$.  For notational simplicity we
will denote
\begin{align*}
 T_{\pm}^{1,0} := \ker \left( \Pi \mp I \right) : T^{1,0}_{\mathbb C}
(M, J_A) \to T^{1,0}_{\mathbb C} (M, J_A).
\end{align*}
We use similar notation to denote the pieces of the complex cotangent bundle. 
Other tensor bundles inherit similar decompositions.  The one of most importance
to us is
\begin{align*}
 \Lambda^{1,1}_{\mathbb C}(M, J_A) =&\ \left(\Lambda^{1,0}_+ \oplus
\Lambda^{1,0}_- \right) \wedge \left(\Lambda^{0,1}_+ \oplus \Lambda^{0,1}_-
\right)\\
=&\ \left[\Lambda^{1,0}_+ \wedge \Lambda_+^{0,1} \right] \oplus
\left[\Lambda^{1,0}_+ \wedge \Lambda_-^{0,1} \right] \oplus
\left[\Lambda^{1,0}_- \wedge \Lambda_+^{0,1} \right] \oplus
\left[\Lambda^{1,0}_- \wedge \Lambda_-^{0,1} \right].
\end{align*}
Given $\mu \in \Lambda^{1,1}_{\mathbb C}(M, J_A)$ we will denote this
decomposition as
\begin{align} \label{oneoneproj}
 \mu := \mu^+ + \mu^{\pm} + \mu^{\mp} + \mu^-.
\end{align}
These decompositions allow us to decompose differential operators as well.  In
particular we can express 
\begin{align*}
d = d_+ + d_-, \qquad \del = \del_+ + \del_-, \qquad \delb = \delb_+ + \delb_-.
\end{align*}
As explained in \cite{SPCFSTB}, the crucial differential operator governing the
local generality of generalized
K\"ahler metrics in this setting is
\begin{align*}
\square := \i \left( \del_+ \delb_+ - \del_- \delb_- \right).
\end{align*}
In particular, locally a generalized K\"ahler metric in this setting is in the
image of $\square$.

\subsubsection{A characteristic class}

\begin{defn} \label{chidef} Let $(M^{2n}, J_{A}, J_B)$ be a bicomplex manifold
such that $[J_A,J_B] = 0$.  Let
\begin{align*}
\chi(J_A,J_B) = c^+_1(T^{1,0}_+) - c^-_1(T^{1,0}_{+}) +
c_1^-(T^{1,0}_{-}) - c_1^+(T^{1,0}_{-}).
\end{align*}
The meaning of this formula is the following: fix Hermitian metrics $h_{\pm}$ on
the holomorphic line bundles $\det T^{1,0}_{\pm}$, and use these to define
elements of $c_1(T^{1,0}_{\pm})$, and then project according to the
decomposition (\ref{oneoneproj}).  In particular, given such metrics $h_{\pm}$
we let $\rho(h_{\pm})$ denote the associated representatives of
$c_1(T_{\pm}^{1,0})$, and then let
\begin{align*}
 \chi(h_{\pm}) = \rho^+(h_+) - \rho^{-}(h_+) + \rho^-(h_-) - \rho^+(h_-).
\end{align*}
This definition yields a well-defined class
in a certain cohomology group, defined in \cite{SPCFSTB}, which we now describe.
\end{defn}

\begin{defn} Let $(M^{2n}, J_A, J_B)$ be a bihermitian manifold with $[J_A,J_B]
= 0$.  Given $\phi_A \in \Lambda^{1,1}_{J_A,\mathbb R}$, let $\phi_B = - \phi_A(
\Pi
\cdot, \cdot) \in \Lambda^{1,1}_{J_B,\mathbb R}$.  We say that $\phi_A$ is
\emph{formally generalized K\"ahler} if
\begin{gather} \label{FGK}
\begin{split}
d^c_{J_A} \phi_A =&\ - d^c_{J_B} \phi_B,\\
d d^c_{J_A} \phi_A =&\ 0.
\end{split}
\end{gather}
\end{defn}

\begin{defn} Let $(M^{2n}, g, J_A, J_B)$ denote a generalized K\"ahler manifold
such that $[J_A,J_B] = 0$.  Let
\begin{align*}
 \mathcal H := \frac{ \left\{ \phi_A \in \Lambda^{1,1}_{J_A, \mathbb R}\ |  \
\phi_A 
\mbox{ satisfies } (\ref{FGK}) \right\}}{ \left\{ \square f \ | \ f \in
C^{\infty}(M) \right\}}.
\end{align*}
\end{defn}

As shown in \cite{SPCFSTB}, the operator $\chi$ yields a well-defined class in $\HH$, analogous to the first Chern class of a K\"ahler manifold.  With this kind of cohomology space, we can define the analogous notion to the
``K\"ahler cone,'' which we refer to as $\mathcal P$, the ``positive cone.''

\begin{defn}
 Let $(M^{2n}, g, J_A, J_B)$ denote a generalized K\"ahler manifold
such that $[J_A,J_B] = 0$.  Let
\begin{align*}
 \mathcal P := \left\{ [\phi] \in \HH\ |\ \exists\ \gw \in [\phi], \gw > 0
\right\}.
\end{align*}
\end{defn}

\begin{defn} Let $(M^{2n}, g,J_{A}, J_B)$ be a generalized K\"ahler manifold
such that $[J_A,J_B] = 0$.  We say that $\chi = \chi(J_A,J_B) > 0$, (resp.
$(\chi < 0,\ \chi =
0)$ if $\chi \in \mathcal P$, (resp. $- \chi \in \mathcal P, \chi = 0$).
\end{defn}

\subsubsection{Pluriclosed flow in commuting generalized K\"ahler geometry}
With this setup we describe how to reduce (\ref{PCF}) to a scalar PDE in the
setting of commuting generalized K\"ahler manifolds.  First we recall that it
follows from (\cite{SPCFSTB} Proposition 3.2, Lemma 3.4)
that the pluriclosed flow in this setting reduces to
\begin{align} \label{PCFGK}
 \dt \gw =&\ - \chi(\gw).
\end{align}
From the discussion above, we see that a solution to (\ref{PCFGK})
induces a solution to an ODE in $\mathcal P$, namely
\begin{align*}
 [\gw_t] = [\gw_0] - t \chi.
\end{align*}
Now we can make a definition which specializes Definition \ref{taustardef} to
this setting.
\begin{defn} \label{GKtaustardef} Given $(M^{2n}, g, J_A, J_B)$ a generalized
K\"ahler manifold with $[J_A,J_B] = 0$, let
\begin{align*}
 \tau^*(g) := \sup \left\{ t \geq 0\ |\ [\gw] - t \chi \in \mathcal P \right\}.
\end{align*}
\end{defn}
Now fix $\tau < \tau^*$, so that by hypothesis if we fix arbitrary metrics
$\til{h}_{\pm}$ on $T^{1,0}_{\pm}$, there exists $a \in C^{\infty}(M)$ such that
\begin{align*}
 \gw_0 - \tau \chi(\til{h}_{\pm}) + \square a > 0.
\end{align*}
Now set $h_{\pm} = e^{\pm \frac{a}{2 \tau}} \til{h}_{\pm}$.  Thus $\gw_0 - \tau
\chi(h_{\pm}) > 0$, and by convexity it follows that
\begin{align*}
 \hat{\gw}_t := \gw_0 - t \chi(h_{\pm}) > 0
\end{align*}
is a smooth one-parameter family of generalized K\"ahler metrics.  Furthermore,
given a function $f
\in C^{\infty}(M)$, let
\begin{align*}
 \gw_f := \hat{\gw} + \square f,
\end{align*}
with $g^f$ the associated Hermitian metric.  Now suppose that $u_t$ satisfies
\begin{align} \label{scalarPCF}
 \dt u =&\ \log \frac{\det g_+^u \det h_-}{\det h_+ \det g_-^u}.
\end{align}
An elementary calculation using the transgression formula for the first Chern
class (\cite{SPCFSTB} Lemma 3.4) yields that $\gw_u$ solves (\ref{PCFGK}).  This
reduction can be refined in the case $\chi = 0$, as we will detail below.

\subsection{Long time existence}

\begin{proof}[Proof of Theorem \ref{CYtype}] First, we choose two background
metrics using the given topological hypotheses.  In particular, since $\chi(J_A,
J_B) = 0$ we may choose a Hermitian metric $h$ such that $\chi(h_{\pm}) = 0$. 
Also, since $c_1^{BC}(J_A) \leq  0$, we may choose a Hermitian metric $\hat{h}$
such
that $\rho(\hat{h}) \leq 0$.  Following the discussion in \S \ref{GKsetup}, we
can
reduce the pluriclosed flow in this setting to a scalar PDE.  In this case the
background metric is $\hat{g}_t = g_0$, and then setting $g^u = g_0 + \square
u_t$, we
let $u$ solve
\begin{align} \label{CYflow}
\dt u =&\ \log \frac{\det g_+^u \det h_-}{\det h_+ \det g_-^u}, \qquad u_0
\equiv 0.
\end{align}
Using that $\chi(h_{\pm}) = 0$ it follows as described above that $g^u_t$ is the
solution to
pluriclosed flow with initial condition $g_0$.

We now derive a priori estimates for $g_t^u$.  Using (\cite{SPCFSTB} Lemma 4.2),
we observe the evolution equation
\begin{align*}
\left( \dt - \gD \right) \dot{u} = 0.
\end{align*}
It follows from the maximum principle that for any smooth existence
time $\tau$ we have
\begin{align} \label{CY10}
\sup_{M \times \{\tau\}} \brs{\dot{u}} \leq \sup_{M \times \{0\}} \brs{\dot{u}}.
\end{align}
A direct integration in time then shows that there is a constant depending the
initial data such that
\begin{align} \label{CY15}
\sup_{M \times \{\tau\}} \brs{u} \leq C\left(1 + \tau \right).
\end{align}
Also, from Lemma \ref{volumeformev} and the fact that $\rho(\hat{h}) \leq 0$ we
observe that
\begin{align*}
\left( \dt - \gD \right) \log \frac{\det g^u}{\det \hat{h}} = \brs{T}^2 - \tr_g
\rho(\hat{h}) \geq 0.
\end{align*}
Thus from the maximum principle we conclude that
\begin{align} \label{CY20}
\inf_{M \times \{\tau\}} \log \frac{\det g^u}{\det \hat{h}} \geq \inf_{M \times
\{0\}}\log \frac{\det g^u}{\det \hat{h}}.
\end{align}
Combining (\ref{CY10}) and (\ref{CY20}) and some elementary identities yields
that there is a constant $C$ depending on the background data such that
\begin{align} \label{CY30}
\log \frac{\det g_-^u}{\det h_-} \geq - C.
\end{align}
Since $\rank T_-^{1,0} = 1$ we conclude that
\begin{align} \label{CY40}
g_-^u \geq e^{-C} h_-.
\end{align}
Next we aim to derive a lower bound for $g_+$.  To do that we fix some $\gl \in
\mathbb R$ and set
\begin{align*}
\Phi = \log \tr_{g_+} h_+ - \gl u.
\end{align*}
Using (\cite{SPCFSTB} Lemma 4.3) we obtain
\begin{align*}
\left( \dt - \gD \right) \Phi =&\ \frac{1}{\tr_{g_+} h_+} \left[ -
\brs{\gU(g_+,h_+)}_{g^{-1},g^{-1},h}^2 + \frac{\brs{\N \tr_{g_+}
h_+}_{g}^2}{\tr_{g_+} h_+} - \tr
g_+^{-1} h_+ g_+^{-1} Q + g^{\bq p} g^{\bgb_+ \ga_+} \Omega^{h_+}_{p\bq \ga_+
\bgb_+} \right]\\
&\ - \gl \dot{u} + \gl \gD u.
\end{align*}
We estimate the various terms of this equation.  First,using (\ref{CY40}) we
have
\begin{align*}
\frac{1}{\tr_{g_+} h_+} g^{\bq p} g^{\bgb_+ \ga_+} \Omega^{h_+}_{p\bq \ga_+
\bgb_+} \leq&\ \frac{C}{\tr_{g_+} h_+} g^{\bq p} g^{\bgb_+ \ga_+} h_{p \bq}
h_{\ga_+ \bgb_+}\\
=&\ C \tr_g h\\
=&\ C \left[\tr_{g_+} h_+ + \tr_{g_-} h_- \right]\\
\leq&\ C \tr_{g_+} h_+ + C.
\end{align*}
Next we estimate using the Cauchy-Schwarz inequality, computing in coordinates
where at the point in question
$h_+ = \Id$ and $g$
is diagonal,
\begin{align*}
 \frac{\brs{\N \tr_{g_+} h_+}^2}{\tr_{g_+} h_+} =&\ \left( \sum {g_+^{\bi i}}
\right)^{-1} g^{\bj j} \N_{j} \tr_{g_+} h_+ \N_{\bj} \tr_{g_+} h_+\\
=&\ \left( \sum {g_+^{\bi i}} \right)^{-1} g^{\bj j} \left[ - g_+^{\bl m} g^+_{m
\bp,j} g_+^{\bp q} h_{q \bl} + g^{\bl m} h_{m \bl,j} \right] \left[ - g_+^{\br
s} g^+_{s \bar{t},\bj} g_+^{\bar{t} u} h_{u \br} + g_+^{\br s} h_{s
\br,\bj}\right]\\
=&\ \left( \sum {g_+^{\bi i}} \right)^{-1} g^{\bj j} \left[ g_+^{\bl m}
\gU(g_+,h_+)_{jm}^q h_{q \bl} \right] \left[ g_+^{\br s} \bar{\gU}(g_+,h_+)_{\bj
\br}^{\bar{t}} h_{s \bar{t}} \right]\\
=&\ \left( \sum {g_+^{\bi i}} \right)^{-1} \left[ \left( g_+^{\bl l}
\right)^{\frac{1}{2}} \left[ \left(g^{\bj j} g_+^{\bl l} \right)^{\frac{1}{2}}
\gU(g_+,h_+)_{jm}^l \right] \right] \left[ \left( g_+^{\br r}
\right)^{\frac{1}{2}} \left( g^{\bj j} g_+^{\br r} \right)^{\frac{1}{2}}
\bar{\gU}(g_+,h_+)_{\bj
\br}^{\bar{s}} \right]\\
\leq&\ \left( \sum {g_+^{\bi i}} \right)^{-1} \left[ \sum g_+^{\bl l}
\right]^{\frac{1}{2}} \left[
\brs{\gU(g_+,h_+)}^2_{g^{-1},g^{-1},h} \right]^{\frac{1}{2}} \left[ \sum
g_+^{\br r} \right]^{\frac{1}{2}} \left[ \brs{\gU(g_+,h_+)}^2_{g^{-1},g^{-1},h}
\right]^{\frac{1}{2}}\\
=&\ \brs{\gU(g_+,h_+)}^2_{g^{-1},g^{-1},h}.
\end{align*}
Also, since $Q \geq 0$ we have $\tr g_+^{-1} h_+ g^{-1}_+ Q \geq 0$.  Lastly,
using the definition of $g^u$ we observe that
\begin{align*}
\gD u =&\ g_u^{\bgb_+ \ga_+} u_{,\ga_+ \bgb_+} + g_u^{\bgb_- \ga_-} u_{,\ga_-
\bgb_-}\\
=&\ g_u^{\bgb_+ \ga_+} \left( g^u_{\ga_+ \bgb_+} - g^0_{\ga_+ \bgb_+} \right) +
g_u^{\bgb_- \ga_-} \left( g^0_{\ga_- \bgb_-} - g^u_{\ga_- \bgb_-} \right)\\
=&\ \rank T_+^{1,0} - \rank T_-^{1,0} - \tr_{g^u_+} g^0_+ + \tr_{g^u_-} g^0_-\\
\leq&\ - \tr_{g_+^u} g^0_+ + C\\
\leq&\ - c \tr_{g_+} h_+ + C,
\end{align*}
for some constants $c, C$.
Combining the above estimates together  with (\ref{CY10}), (\ref{CY15}), and
choosing $\gl$
sufficiently large yields
\begin{align*}
\left( \dt - \gD \right) \Phi \leq&\ C \gl + \left(C - \gl c \right) \tr_{g^u_+}
h_+.\\
\leq&\ C \gl - \frac{\gl}{2} \tr_{g_+} h_+\\
=&\ C \gl - \frac{\gl}{2} e^{\gl u + \Phi}\\
\leq&\ C \gl - \frac{\gl}{2} e^{- \gl C(1 + t) + \Phi}.
\end{align*}
At a spacetime maximum where $\Phi \geq \gl C(1 + t) + \log 2C$, we
yield $(\dt - \gD) \Phi \leq 0$.  It follows from the maximum principle that
there is a constant $C$ depending on the initial data such that
\begin{align*}
\sup_{M \times \{t\}} \Phi \leq C \left(1 + t \right).
\end{align*}
From the definition of $\Phi$ and the a priori estimate for $u$ this implies a
lower bound for $g_+$ on any finite time interval.  The theorem now follows from
(\cite{SPCFSTB} Proposition 5.3, \cite{SW} Theorem 1.2).
\end{proof}

\subsection{Yau's oscillation estimate on generalized K\"ahler manifolds}

In this subsection we adapt Yau's potential oscillation estimate \cite{Yau} to
the
setting of commuting generalized K\"ahler geometry.  We use ideas of Cherrier
\cite{Cher} in a similar way as exploited by Tosatti-Weinkove \cite{TW,TW2}, who
proved a direct generalization of Yau's estimate in the Hermitian setting.

\begin{thm} \label{Yauosc} Let $(M^{2n}, g, J_A, J_B)$ be a compact generalized
K\"ahler manifold, and let $u \in C^{\infty}(M)$ satisfy
\begin{align*}
 \left( \gw_+ + \i \del_+ \delb_+ u \right)^k =&\ e^{F_+} \gw_+^k,\\
 \left( \gw_- - \i \del_- \delb_- u \right)^l =&\ e^{F_-} \gw_-^l,\\
 \tr_{\gw} \i \del \delb u >&\ - \gl.
\end{align*}
There exists a constant $C = C(\sup F_+,\inf F_-, \gl)$ such that
\begin{align*}
 \osc_M u = \sup_M u - \inf_M u \leq C.
\end{align*}
\end{thm}

\begin{rmk} Usually in this type of estimate it is only an upper bound on the
volume form which is required.  Due to the extra minus sign appearing in the
definition of $\gw^{u}_-$ the estimate depends on a lower bound for this partial
volume form.  Also note that the third hypothesis of Theorem \ref{Yauosc}, the
lower
bound on the Laplacian of $u$, is satisfied automatically in the K\"ahler
setting assuming $u$ defines a K\"ahler form, i.e. $\gw + \i \del \delb u > 0$. 
It is an interesting challenge to remove this hypothesis.
\end{rmk}

To begin we record a certain estimate inspired by a lemma of Cherrier
\cite{Cher}.

\begin{lemma} \label{moserlemma} There exist $C,p_0$ so that for all $p \geq
p_0$ we have
\begin{align*}
\int_M \brs{\del e^{-\frac{p}{2} u}}_g^2 \gw_+^k \wedge \gw_-^l \leq C p \int_M
e^{-p u} \gw_+^k \wedge \gw_-^l.
\end{align*}
\end{lemma}

\begin{proof}
We set
\begin{align*}
\mu_{+} =&\ \sum_{i=0}^{k-1} \gw_{u,+}^i \wedge \gw_+^{k-i-1}, \qquad \mu_- =
\sum_{i=0}^{l-1} \gw_{u,-}^i \wedge \gw_-^{l-i-1}.
\end{align*}
Observe by the generalized K\"ahler conditions that $d_{\pm} \mu_{\pm} = 0$. 
Using the bounds on the volume forms and integrating by parts yields
\begin{gather} \label{chlem10}
 \begin{split}
C \int_M e^{-p u} \gw_+^k \wedge \gw_-^l \geq&\ \int_M e^{-pu} \left(
\gw_{u,+}^k - \gw_+^k \right) \wedge \gw_-^l\\
=&\ \int_M e^{-pu} \i \del_+ \delb_+ u \wedge \mu_+ \wedge \gw_-^l\\
=&\ p \int_M e^{-pu} \i \del_+ u \wedge \delb_+ u \wedge \mu_+ \wedge \gw_-^l +
\int_M e^{-p u} \i \delb_+ u \wedge \mu_+  \wedge \del_+ \gw_-^l.
 \end{split}
\end{gather}
Next, using the Cauchy-Schwarz inequality one obtains
\begin{gather} \label{CSestimate}
 \begin{split}
& \brs{ \frac{\i \delb_+ u \wedge \gw_{u,+}^{i} \wedge \gw_{+}^{k-i-1} \wedge
\del_+ \gw_- \wedge \gw_-^{l-1}}{\gw_+^k \wedge \gw_-^l}}\\
&\ \qquad \qquad \leq \frac{C}{\ge} \frac{\i \del_+ u \wedge \delb_+ u \wedge
\gw_{u,+}^i
\wedge \gw_+^{k-i-1} \wedge \gw_-^l}{\gw_+^k \wedge \gw_-^l} + C \ge \frac{
\gw_{u,+}^i \wedge \gw_{+}^{k-i} \wedge \gw_-^l}{\gw_+^k \wedge \gw_-^l}.  
 \end{split}
\end{gather}
Using this we have
\begin{gather*}
 \begin{split}
- \int_M e^{-pu} \i \delb_+ u \wedge \mu_+ \wedge \del_+ \gw_-^l =&\ - l \int_M
e^{-pu} \i \delb_+ u \wedge \left( \sum_{i=0}^{k-1} \gw_{u,+}^i \wedge
\gw_+^{k-i-1} \right) \wedge \del_+ \gw_- \wedge \gw_-^{l-1}\\
\leq&\ \frac{C}{\ge} \sum_{i=0}^{k-1} \int_M e^{-pu} \i \del_+ u \wedge \delb_+
u \wedge
\gw_{u,+}^i \wedge \gw_+^{k-i-1} \wedge \gw_-^l\\
&\ + \ge C \sum_{i=0}^{k-1} \int_M e^{-pu} \gw_{u,+}^i \wedge \gw_{+}^{k-i}
\wedge \gw_-^l.  
 \end{split}
\end{gather*}
Now choose $\ge$ sufficiently small with respect to controlled constants and
choose $p_0$ sufficiently large to obtain
\begin{gather} \label{chlem20}
 \begin{split}
- \int_M e^{-pu} \i \delb_+ u \wedge \mu_+ \wedge \del_+ \gw_-^l \leq&\
\frac{p}{4} \int_M e^{-pu} \i \del_+ u \wedge \delb_+ u \wedge \mu_+ \wedge
\gw_-^l\\
&\ + \ge C \int_M e^{-pu} \gw_+^k \wedge \gw_-^l + \ge C \sum_{i=1}^{k-1} \int_M
e^{-pu} \gw_{u,+}^i \wedge \gw_+^{k-i} \wedge \gw_-^l.  
 \end{split}
\end{gather}
Combining (\ref{chlem10}), (\ref{chlem20}) we get
\begin{gather} \label{chlem30}
 \begin{split}
\frac{p}{2} \int_M e^{-pu} \i \del_+ u \wedge \delb_+ u \wedge \mu_+ \leq C
\int_M e^{-pu} \gw_+^k \wedge \gw_-^l + \ge C \sum_{i=1}^{k-1} \int_M e^{-pu}
\gw_{u,+}^i \wedge \gw_+^{k-i} \wedge \gw_-^l.  
 \end{split}
\end{gather}
Next we claim that for $j=0,\dots,k$ there exist constants $C_j$ such that
\begin{gather} \label{chlem40}
 \begin{split}
\frac{p}{2^{j+1}} \int_M & e^{-pu} \i \del_+ u \wedge \delb_+ u \wedge \mu_+
\wedge \gw_-^l\\
\leq&\ C_j \int_M e^{-pu} \gw_+^k \wedge \gw_-^l + \ge C_j \sum_{i=1}^{k-j}
\int_M e^{-pu} \gw_{u,+}^i \wedge \gw_+^{k-i} \wedge \gw_-^l.  
 \end{split}
\end{gather}
We prove this by induction on $j$, the case $j=0$ being equivalent to
(\ref{chlem30}).  We assume the result for a general $j$ and show it for $j+1$. 
First we separate the last term of (\ref{chlem40}) to yield
\begin{gather} \label{chlem45}
\begin{split}
\ge C_j \sum_{i=1}^{k-j} \int_M e^{-pu} \gw_{u,+}^i \wedge \gw_+^{k-i} \wedge
\gw_-^l =&\ \ge C_j \sum_{i=1}^{k-j} \int_M e^{-p u} \gw_{u,+}^{i-1} \wedge
\gw_+^{k-i+1} \wedge \gw_-^l\\
&\ + \ge C_j \sum_{i=1}^{k-j} \int_M e^{-pu} \i \del_+ \delb_+ u \wedge
\gw_{u,+}^{i-1} \wedge \gw_+^{k-i} \wedge \gw_-^l\\
=:&\ I + II.
\end{split}
\end{gather}
The term $I$ can be included on the right hand side of (\ref{chlem40}) in
showing the inductive step.  For term $II$ we further integrate by parts
\begin{align*}
II =&\ \ge C_j \sum_{i=0}^{k-j-1} \left[ p \int_M e^{- p u} \i \del_+ u \wedge
\delb_+ u \wedge \gw_{u,+}^{i} \wedge \gw_+^{k-i-1}  \wedge \gw_-^l \right.\\
&\ \left. \qquad - \int_M e^{-pu} \i \delb_+ u \wedge \gw_{u,+}^i \wedge
\gw_+^{k-i-1} \wedge \del_+ \gw_-^l \right]\\
=&\ II_A + II_B.
\end{align*}
Choosing $\ge$ sufficiently small yields
\begin{align} \label{chlem50}
II_A \leq&\ \frac{p}{2^{j+3}} \int_M e^{-p u} \i \del_+ u \wedge \delb_+ u
\wedge \mu_+ \wedge \gw_-^l.
\end{align}
Next using (\ref{CSestimate}) we have
\begin{gather} \label{chlem60}
\begin{split}
 II_B \leq&\ \ge C C_j \sum_{i=0}^{k-j-1} \left[ \int_M e^{-pu} \i \del_+ u
\wedge \delb_+ u \wedge \gw_{u,+}^i \wedge \gw_+^{k-i-1} \wedge \gw_-^l
\right.\\
&\ \left. \qquad \qquad + \int_M e^{-pu} \gw_{u,+}^i \wedge \gw_+^{k-i} \wedge
\gw_-^l \right]\\
\leq&\ \frac{p}{2^{j+3}} \int_M e^{-p u} \i \del_+ u \wedge \delb_+ u \wedge
\mu_+ \wedge \gw_-^l + \ge C_j \sum_{i=0}^{k-j - 1} \int_M e^{-pu} \gw_{u,+}^i
\wedge \gw_+^{k-i} \wedge \gw_-^l,
\end{split}
\end{gather}
where the last line follows by ensuring $\ge$ is chosen sufficiently small. 
Plugging (\ref{chlem50}) and (\ref{chlem60}) into (\ref{chlem45}) finishes the
proof of (\ref{chlem40}), which for $j=k$ can be rewritten as
\begin{gather} \label{chlem65}
\int_M \brs{\del_+ e^{-\frac{p}{2} u}}^2_g \gw_+^k \wedge \gw_-^l \leq C p
\int_M e^{-pu} \gw_+^k \wedge \gw_-^l.
\end{gather}
Arguing similarly and using a \emph{lower} bound for the partial volume form
$\gw_{u,-}^l$ we have
\begin{gather*}
 \begin{split}
C \int_M e^{-p u} \gw_+^k \wedge \gw_-^l \geq&\ \int_M e^{-p u} \gw_+^k \wedge
\left( \gw_-^l - \gw_{u,-}^l \right)\\
=&\ \int_M e^{-pu} \gw_+^k \wedge \i \del_- \delb_- u \wedge \mu_-\\
=&\ p \int_M e^{-pu} \i \del_- u \wedge \delb_- u \wedge \mu_- \wedge \gw_+^{k}
+ \int_M e^{-pu} \i \delb_- u \wedge \mu_- \wedge \del_- \gw_+^k.  
 \end{split}
\end{gather*}
A similar series of estimates as detailed above yields the inequality
\begin{gather} \label{chlem70}
\int_M \brs{\del_- e^{-\frac{p}{2} u}}^2_g \gw_+^k \wedge \gw_-^l \leq C p
\int_M e^{-pu} \gw_+^k \wedge \gw_-^l.
\end{gather}
Combining (\ref{chlem65}) and (\ref{chlem70}) yields the result.
\end{proof}

\begin{lemma} \label{L1estimate} There exists a constant $C$ such that, if we
set $v = u - \inf_M u$,
 \begin{align*}
  \nm{v}{L^1} \leq \nm{v - \underbar{$v$}}{L^1} + C.
 \end{align*}
\begin{proof} Combining Lemma \ref{moserlemma} with the Sobolev inequality, we
see that for $p \geq p_0$ and $\gl = \frac{n}{n-1}$ we have
\begin{align*}
\nm{e^{-u}}{L^{\gl p}} =&\ \left( \int_M e^{-\gl p u} \gw^n
\right)^{\frac{1}{\gl p}}\\
\leq&\ \left[ C \int_M \brs{\del e^{-\frac{p}{2} u}}^2 \gw^n  + C \int_M e^{- p
u}
\gw^n \right]^{\frac{1}{p}}\\
\leq&\ C^{\frac{1}{p}} p^{\frac{1}{p}} \nm{e^{-u}}{L^p}.
\end{align*}
Iterating this estimate yields
\begin{gather} \label{L1est9}
 \begin{split}
 \nm{e^{-u}}{L^{\infty}} =&\ \lim_{l \to \infty} \nm{e^{- u}}{L^{p_0 \gl^l}}\\
 \leq&\ \lim_{l \to \infty} C^{\sum_{i=0}^l \gl^{-i}} \prod_{i=0}^{l-1} \left(
\gl^i p_0 \right)^{\frac{1}{\gl^i p_o}} \nm{e^{-u}}{L^{p_0}}\\
\leq&\ C_0 \nm{e^{-u}}{L^{p_0}}.  
 \end{split}
\end{gather}
We use this in conjunction with an argument from (\cite{TW2} Lemma 3.2) to show
that there exist uniform constants $C_1$ and $\gd$ so that
\begin{align} \label{L1est10}
 \brs{ \{u < \inf_M u + C_1 \} } \geq \gd.
\end{align}
Let $w = p_0 u$, so that (\ref{L1est9}) reads
\begin{align} \label{Lest11}
 e^{-\inf w} \leq C_0 \int_M e^{-w} \gw^n.
\end{align}
Now we split
\begin{align*}
 \int_M e^{-w} = \int_{ \left\{ e^{-w} \geq \frac{1}{2} \int_M e^{-w} \right\}}
e^{-w} + \int_{ \left\{ e^{-w} < \frac{1}{2} \int_M e^{-w} \right\}} e^{-w} = I
+ II.
\end{align*}
Since
\begin{align*}
 \brs{ \{e^{-w} \geq \frac{1}{2} \int_M e^{-w} \}} \leq&\ \brs{ \{ e^{-w} \geq
\frac{1}{2 C_0} e^{- \inf w} \}}= \brs{ \{w < \inf_M w + C_1 \} },
\end{align*}
we obtain, combining with (\ref{Lest11}),
\begin{align*}
 I \leq \brs{ \{w < \inf_M w + C_1 \} } \sup_M e^{-w} \leq C_0  \brs{ \{w <
\inf_M w + C_1 \} } \int_M e^{-w}.
\end{align*}
On the other hand we directly estimate using that $\int_M d \mu = 1$,
\begin{align*}
 II \leq&\ \frac{1}{2} \int_M e^{-w}.
\end{align*}
Thus we obtain
\begin{align*}
 \int_M e^{-w} \leq \left[ C_0 \brs{ \{w < \inf_M w + C_1 \} } + \frac{1}{2}
\right]
\int_M e^{-w}
\end{align*}
Rearranging this yields
\begin{align*}
 \brs{ \{w < \inf_M w + C_1 \} } \geq \gd
\end{align*}
for some uniform constant $\gd > 0$.  Since $p_0$ is some fixed constant, we
obtain (\ref{L1est10}).  Using this we have
\begin{align*}
\nm{v}{L^1} =&\ \int_M v d\mu\\
=&\ \underbar{$v$}\\
\leq&\ \gd^{-1} \int_{\{v \leq C_1\}}  \underbar{$v$}\\
\leq&\ \int_{\{v \leq C_1\}}  \brs{v - \underbar{$v$}} + C_1\\
\leq&\ \nm{v - \underbar{$v$}}{L^1} + C_1.
\end{align*}
\end{proof}
\end{lemma}

\begin{lemma} \label{moser3} (\cite{TW} Lemma 2.3) Let $(M^{2n}, \gw_G,J)$ be a
compact complex manifold with Gauduchon metric $\gw_G$.  Let $f \in
C^{\infty}(M)$ satisfy
\begin{align*}
 \gD_{\gw_G} f \geq - C_0.
\end{align*}
Then there exist constants $C_1, C_2$ depending on $(M^{2n}, \gw_G, J)$ and
$C_0$ such that for all $p \geq 1$,
\begin{align} \label{m310}
\int_M \brs{\del f^{\frac{p+1}{2}}}^2_{\gw_G} \gw_G^n \leq C_1 p \int_M f^p
\gw_g^n,
\end{align}
and
\begin{align} \label{m320}
 \sup_M f \leq C_2 \max \left\{ \int_M f \gw_G^n, 1 \right\}.
\end{align}
\end{lemma}

\begin{proof} [Proof of Theorem \ref{Yauosc}] Let $v = u - \inf_M u$.  By
Gauduchon's theorem \cite{Gauduchon} $\gw$ admits a conformally related
Gauduchon metric $\gw_G = e^{\phi} \gw$, for some smooth function $\phi$.  By
assumption we have $\gD_{\gw_G} v = e^{-\phi} \gD_{\gw} v > - C$.  Thus, by
Lemma \ref{moser3}, it suffices to estimate the $L^1$ norm of $v$.  Using
(\ref{m310}) with $p=1$, the Poincar\'e inequality for $\gw_G$, and Lemma
\ref{L1estimate}, we have
\begin{align*}
 \nm{v}{L^1} \leq&\ C + \nm{v - \underbar{$v$}}{L^1}\\
 \leq&\ C + \nm{v-\underbar{$v$}}{L^2}\\
 \leq&\ C + C \nm{\del v}{L^2}\\
 \leq&\ C + C \nm{v}{L^1}^{\frac{1}{2}}.
\end{align*}
The estimate for $\nm{v}{L^1}$, and hence the theorem, follows.
\end{proof}

\subsection{Convergence}

We now establish the convergence claims of Theorem \ref{CYconv}.  Before getting
to the proof we record a topological splitting result for
K\"ahler-Einstein manifolds.

\begin{thm} \label{rigiditythm} (\cite{AG, Beauville,Kob}) Let $(M, J)$  be a
compact complex manifold which admits a K\"ahler-Einstein metric $g$, and whose
tangent bundle splits as a direct sum of two holomorphic sub-bundles $T_{\pm}
M$.  Then $T_{\pm} M$ are parallel with respect to the Levi-Civita connection of
$g$.  In particular, $(M, g, J)$ is a local K\"ahler product of two
K\"ahler-Einstein manifolds tangent to $T_{\pm} M$.
\end{thm}

\begin{proof} Using the Bochner-Kodaira identity shows that for a holomorphic
endomorphism $Q$ of $TM$ one has for a K\"ahler-Einstein background,
\begin{align*}
\int_M \brs{\N Q} = \int_M \IP{ [g^{-1} \Rc,Q],Q} = \int_M \IP{ [ \gl I,Q],Q} =
0.
\end{align*}
Thus $Q$ is parallel, and the theorem follows from the de Rham decomposition
theorem.
\end{proof}

\begin{proof}[Proof of Theorem \ref{CYconv}] To begin we establish more refined
versions of the estimates of Theorem \ref{CYtype}.  In particular we claim
uniform equivalence of the metrics and uniform $C^{\infty}$ estimates for all
time.  As in Theorem \ref{CYtype} we choose a Hermitian metric $h$ such that
$\chi(h_{\pm}) = 0$.  Furthermore since now $c_1^{BC}(J_A) = 0$ we choose a
Hermitian metric $\hat{h}$ such that $\rho(\hat{h}) = 0$.  We now consider two
reductions of the pluriclosed flow. First, following Theorem \ref{CYtype} we set
$\hat{g} = g_0$, let $g^u = g_0 + \square u$ and let $u$ solve
\begin{align*}
 \dt u =&\ \log \frac{\det g_+^u \det h_-}{\det h_+ \det g_-^u}, \qquad u_0
\equiv 0.
\end{align*}
On the other hand, with the choices $\hat{g}, \hat{h}$, and $\mu = 0$, using
Lemma \ref{reduction} we choose a solution $\ga_t$ to $(\hat{g}_t,
\hat{h},0)$-reduced pluriclosed flow.  Since by hypothesis $[\del
\hat{\gw}_0] = 0$, there exists $\eta \in \Lambda^{2,0}$ such that
\begin{align*}
 \del \hat{\gw}_t = \del \hat{\gw}_0 = \delb \eta.
\end{align*}
Thus choose $\phi = \del \ga - \eta$ and by Proposition \ref{specialtorspot} we
have
\begin{align} \label{conv10}
 \left(\dt - \gD_{g_t} \right) \brs{\phi}^2 =&\ - \brs{\N \phi}^2 - \brs{T}^2 -
2 \IP{Q, \phi \otimes \bar{\phi}}.
\end{align}
This differential inequality is very helpful in obtaining estimates of the
metric, and is the reason for considering the $\ga$-reduction of the pluriclosed
flow as well as the scalar reduction.  While of course these two reductions are
related, the scalar reduction involves also the background metric $h$, and the
nature of these background terms interferes with simply using $\del_- \del_+ u$
as the torsion potential.

First we observe that the estimates (\ref{CY10}), (\ref{CY20}) and (\ref{CY30})
of Theorem \ref{CYtype} still hold in this setting.  Combining Lemma
\ref{volumeformev} with (\ref{conv10}) yields
\begin{align*}
 \left(\dt - \gD \right) \left[ \log \frac{\det g^u}{\det \hat{h}} +
\brs{\phi}^2 \right] \leq 0.
\end{align*}
It follows from the maximum principle that
\begin{align} \label{conv20}
 \sup_{M \times \{t\}} \left[ \log \frac{\det g^u}{\det \hat{h}} + \brs{\phi}^2
\right] \leq C.
\end{align}
Combining (\ref{conv20}) with (\ref{CY10}) and (\ref{CY20}) and making
elementary manipulations yields
\begin{align} \label{conv25}
 \brs{\log \frac{ \det g^u_{\pm}}{\det h_{\pm}}} \leq C.
\end{align}
Since $\rank T_-^{1,0} = 1$ we conclude that
\begin{align} \label{conv30}
 C^{-1} h_- \leq g_-^u \leq C h_-.
\end{align}
Note that by combining the basic fact $g^u_+ > 0$ with $g_-^u \leq C h_-$ we
conclude that there is a constant $C$ such that
\begin{align*}
 \tr_{g_0} \i \del \delb u > - C.
\end{align*}
We now construct a normalized potential for the metric.  In particular, we let
$v_t$ solve
\begin{align*}
 \dt v =&\ \log \frac{\det g_+^u \det h_-}{\det h_+ \det g_-^u} - \frac{\int_M
\log \frac{\det g_+^u \det h_-}{\det h_+ \det g_-^u} \gw_0^n}{\int_M \gw_0^n},
\qquad v_0
\equiv 0.
\end{align*}
Certainly $v_t$ only differs from $u_t$ by a constant, thus $g^u = g^v$.  The
estimates (\ref{conv25}) above imply
\begin{align} \label{vestimates}
 \brs{\log \frac{ \det g^v_{\pm}}{\det h_{\pm}}} \leq C, \qquad \tr_{g_0} \i
\del\delb v > -C.
\end{align}
Moreover, an elementary calculation shows that for all $t$,
\begin{align} \label{conv40}
 \int_M v \gw_0^n = 0.
\end{align}
With the estimates (\ref{conv25}) and (\ref{conv30}) in place, Theorem
\ref{Yauosc} implies that for any time $t$,
\begin{align} \label{conv50}
\osc_{M \times \{t\}} v \leq C.
\end{align}
Estimates (\ref{conv25}), (\ref{conv40}) and (\ref{conv50}) together imply that
for all $t$,
\begin{align} \label{conv60}
 \sup_{M \times \{t\}} \left[ \brs{v} + \brs{\dt v} \right] \leq C.
\end{align}
Now let
\begin{align*}
\Phi = \log \tr_{g_+} h_+ - \gl v.
\end{align*}
Arguing as in the proof of Theorem \ref{CYtype}, and using (\ref{conv60}) yields
\begin{align*}
\left( \dt - \gD \right) \Phi \leq C - C^{-1} e^{-C + \Phi}.
\end{align*}
A direct application of the maximum principle then yields a uniform upper bound
for $\Phi$, which after manipulations yields
\begin{align} \label{conv70}
 g^u_+ \geq C^{-1} h_+.
\end{align}
Combining (\ref{conv20}), (\ref{conv30}), and (\ref{conv70}) yields
\begin{align*}
 C^{-1} h \leq g_t^u \leq C h.
\end{align*}
We invoke (\cite{SW} Theorem 1.2) to obtain uniform $C^{\infty}$ estimates for
$g_t$
for all times.

We now establish exponential $C^{\infty}$ convergence of the flow.  This follows
from an argument of Li-Yau type \cite{LiYau}, which we only sketch, as it
follows standard lines.  At this point we have established uniform upper and
lower bounds as well as uniform estimates on all space and time derivatives for
the metrics.  Now let $f_t$ denote a positive solution to the time-dependent
heat
equation $(\dt - \gD_{g_t}) f = 0$.  A lengthy series of
estimates shows that there exists $\ga > 0$ and a constant $C$ such that
\begin{align*}
\brs{\N f}^2 - \ga f_t \leq C \left(1 + \frac{1}{t} \right).
\end{align*}
Integrating this over paths in spacetime yields the Harnack estimate
\begin{align*}
\sup_{M} u(x,t_1) \leq \inf_M u(x,t_2) \left( \frac{t_2}{t_1} \right)^C \exp
\left( \frac{C}{t_2 - t_1} + C(t_2 - t_1) \right).
\end{align*}
Now for $n \in \mathbb N$ we define
\begin{align*}
\mu_n(x,t) =&\ \sup_M f(x,n-1) - f(x,n-1 + t),\\
\nu_n(x,t) =&\ f(x,n-1+t) - \inf_{M} f(x,n-1),\\
\osc(t) =&\ \sup_M f(x,t) - \inf_M f(x,t).
\end{align*}
All $\mu_n$ and $\nu_n$ are solutions to the time-dependent heat equation and so
by the Harnack estimate with $t_1 = \frac{1}{2}$, $t_2 = 1$ we obtain
\begin{align*}
\sup_M f(x,n-1) - \inf f(x,n-\tfrac{1}{2}) \leq&\ C \left( \sup_M f(x,n-1) -
\sup_M f(x,n) \right),\\
\sup_M f(x,n-\tfrac{1}{2}) - \inf f(x,n-1) \leq&\ C \left( \inf_M f(x,n) -
\inf_M f(x,n-1) \right).
\end{align*}
Adding these together yields
\begin{align*}
\osc(n-1) + \osc(n - \frac{1}{2}) \leq C (\osc(n-1) - \osc(n)),
\end{align*}
hence
\begin{align*}
\osc(n) \leq \gl \osc(n-1),
\end{align*}
where $\gl = \frac{C-1}{C} < 1$.  Since the oscillation function is
nonincreasing it follows that $\osc(t) \leq C e^{-\gl t}$.  Applying this
discussion to $\dt u$ shows that it converges exponentially to a constant. 
Since $u$ and $v$ differ only by time dependent constants, it follows that $\dt
v$ also converges exponentially to a constant, which
must be zero by (\ref{conv40}).  It follows directly that the metric is
converging exponentially,
and that the limiting metric satisfies $\chi(g^{v_\infty}_{\pm}) = 0$.  Lemma
\ref{rigiditylemma} now implies that $g^{v_{\infty}}$ is Calabi-Yau.  The
remaining
claims of the theorem follow directly from Theorem \ref{rigiditythm}.
\end{proof}

\begin{proof}[Proof of Corollary \ref{CYcor}] To prove the corollary we simply
show that the assumptions imply the setup of Theorem \ref{CYconv}.  Fix
Hermitian metrics $\til{h}_{\pm}$ on $T^{1,0}_{\pm}$.  Since by assumption
$c_1^{BC}(T^{1,0}_{\pm}) = 0$, there exist smooth functions $f_{\pm}$ such that
\begin{align*}
 \rho(e^{f_{\pm}} \til h_{\pm}) = \rho(\til{h}_{\pm}) - \i \del \delb f = 0.
\end{align*}
Now let $h = e^{f_+} \til{h}_+ \oplus e^{f_-} \til{h}_-$.  It follows directly
that
\begin{align*}
 \rho(h) = \rho(h_+) + \rho(h_-) = 0, \qquad \chi(h_{\pm}) = \rho^+(h_+) -
\rho^-(h_+) + \rho^-(h_-) - \rho^+(h_-) = 0.
\end{align*}
Thus $c_1^{BC}(J_A) = 0$, and $\chi(J_A,J_B) = 0$, and so the corollary follows
from Theorem \ref{CYconv}.
\end{proof}

\bibliographystyle{hamsplain}

\end{document}